\documentclass{pspum-l}

\usepackage{amssymb}

\usepackage{url}
\usepackage{tikz-cd}

\newcommand{\R}{\mathbb{R}}
\newcommand{\C}{\mathbb{C}}
\newcommand{\Z}{\mathbb{Z}}
\newcommand{\Q}{\mathbb{Q}}
\newcommand{\Qellbar}{\overline{\mathbb{Q}_{\ell}}}
\newcommand{\Qell}{\mathbb{Q}_{\ell}}
\newcommand{\W}{\mathrm{W}}
\newcommand{\WD}{\mathrm{WD}}
\newcommand{\I}{\mathrm{I}}
\newcommand{\Frob}{\mathrm{Frob}}
\newcommand{\GL}{\mathrm{GL}}
\newcommand{\SL}{\mathrm{SL}}

\newcommand{\rs}{\mathrm{rs}}
\newcommand{\id}{\mathrm{id}}
\newcommand{\tr}{\operatorname{tr}}
\newcommand{\ab}{\mathrm{ab}}
\newcommand{\ad}{\mathrm{ad}}
\newcommand{\der}{\mathrm{der}}
\newcommand{\cont}{\mathrm{cont}}
\newcommand{\alg}{\mathrm{alg}}
\newcommand{\bas}{\mathrm{bas}}
\newcommand{\Isoc}{\mathrm{Isoc}}

\newcommand{\inv}{\operatorname{inv}}
\newcommand{\res}{\operatorname{res}}
\newcommand{\brd}{\operatorname{brd}}
\newcommand{\Ad}{\operatorname{Ad}}
\newcommand{\Lie}{\operatorname{Lie}}
\newcommand{\Hom}{\operatorname{Hom}}
\newcommand{\End}{\operatorname{End}}
\newcommand{\Aut}{\operatorname{Aut}}
\newcommand{\Out}{\operatorname{Out}}
\newcommand{\Isom}{\operatorname{Isom}}
\newcommand{\Irr}{\operatorname{Irr}}
\newcommand{\Res}{\operatorname{Res}}
\newcommand{\Gal}{\operatorname{Gal}}
\newcommand{\Cent}{\operatorname{Cent}}
\newcommand{\Rep}{\operatorname{Rep}}
\newcommand{\Fib}{\operatorname{Fib}}

\newcommand{\Spec}{\operatorname{Spec}}
\newcommand{\LL}{\operatorname{LL}}
\newcommand{\LLss}{\operatorname{LL}^\mathrm{ss}}
\newcommand{\onto}{\twoheadrightarrow}
\newcommand{\Ghat}{\widehat{G}}
\newcommand{\Hhat}{\widehat{H}}
\newcommand{\That}{\widehat{T}}

\newcommand{\Mhat}{\widehat{M}}

\newcommand{\Sbar}{\overline{S}}
\newcommand{\E}{\mathcal{E}}
\newcommand{\Epur}{\mathcal{E}^\mathrm{pur}}
\newcommand{\Eiso}{\mathcal{E}^\mathrm{iso}}
\newcommand{\Erig}{\mathcal{E}^\mathrm{rig}}
\newcommand{\ol}[1]{\overline{#1}}
\newcommand{\ul}[1]{\underline{#1}}

\newcommand{\etalchar}[1]{$^{#1}$}

\newtheorem{theorem}{Theorem}[section]
\newtheorem{lemma}[theorem]{Lemma}
\newtheorem{proposition}[theorem]{Proposition}
\newtheorem{conjecture}[theorem]{Conjecture}

\theoremstyle{definition}
\newtheorem{definition}[theorem]{Definition}

\theoremstyle{remark}
\newtheorem{remark}[theorem]{Remark}

\numberwithin{equation}{section}

\begin{document}

\title{The local Langlands conjecture}


\author{Olivier Taïbi}
\address{ENS de Lyon site Monod \\ UMPA UMR 5669 CNRS \\ 46, allée d’Italie \\ 69364 Lyon Cedex 07 \\ France }
\email{olivier.taibi@ens-lyon.fr}
\thanks{The author was employed by CNRS and supported in part by ANR project COLOSS: ANR-19-CE40-0015}


\date{May 9, 2025}

\begin{abstract}
  We formulate the local Langlands conjecture for connected reductive groups over local fields, including the internal parametrization of L-packets using endoscopy.
\end{abstract}

\maketitle

\section{Introduction}

The first goal of these notes is to state the local Langlands conjecture for connected reductive groups \(G\) over a local field \(F\), that is the existence of a map \(\LL\) with finite fibers associating Langlands parameters to irreducible smooth representations (\((\mathfrak{g}, K)\)-modules in the case where \(F\) is Archimedean) over an algebraically closed field \(C\) of characteristic zero (the field of complex numbers \(\C\) in the Archimedean case).
To be useful this map should satisfy certain properties, and we list some of them in Conjecture \ref{conj:crude_LLC}, after recalling in Sections \ref{sec:rep} and \ref{sec:langlands_param} parallel features of the classification of representations of \(G(F)\) and Langlands parameters.
For representations of \(G(F)\) we put the emphasis on the case of complex coefficients (\(C = \C\)), sometimes using notions relying on the topology of \(\C\), because of the relatively simple (partial) classification of representations using parabolically induced representations of essentially discrete representations of Levi subgroups.
We do point out however that the map \(\LL\) should be ``algebraic'' (in particular, functorial in \(C\)) and formulate the (purely algebraic) semi-simplified Langlands correspondence (Conjecture \ref{conj:crude_LLC_ss}).
Unfortunately neither conjecture includes a characterization of the map \(\LL\), and proofs of cases of these conjectures use different characterizations.

We then formulate refined versions of the local Langlands correspondence, describing the fibers (``\(L\)-packets'') of the maps \(\LL\) using centralizers of Langlands parameters.
In the case where \(G\) is quasi-split this is fairly straightforward (Conjectures \ref{conj:refined_LLC_qs} and \ref{conj:endo_char_rel_qs}) and includes Shahidi's conjecture (Conjecture \ref{conj:shahidi}).
Formulating the refined correspondence in the non-quasi-split case (Conjectures \ref{conj:refined_LLC_inner} and \ref{conj:endo_char_rel_inner}) is surprisingly difficult in general, and was only relatively recently achieved by Kaletha, generalizing an idea of Vogan from the case of pure inner forms of quasi-split groups.
This entails employing \emph{Galois gerbes} instead of Galois groups, thus generalizing Galois cohomology sets.
In this setting where explicit computations seem to be unavoidable it is favorable to work with a down-to-earth definition of such gerbes as extensions of absolute Galois groups by a group of multiplicative type.

We conclude in Section \ref{sec:gerbe_tann} with a short explanation of the relation between this definition of gerbes with the more conceptual one.
This is motivated by the fact that the same Galois gerbe, corresponding to the Tannakian category of isocrytals, appears in the study of Shimura varieties, in the internal structure of certain \(L\)-packets, and geometrically in the construction by Fargues and Scholze \cite{FarguesScholze} of a semi-simplified local Langlands correspondence.

We are grateful to Naoki Imai, David Schwein, Alex Youcis and an anonymous referee for comments on an earlier version of these notes.

\section{Notations}

Let \(F\) be a local field.
We denote by \(||\cdot||\) the normalized absolute value on \(F\).
In the non-Archimedean case it maps a uniformizer to \(q^{-1}\) where \(q\) is the cardinality of the residue field.
If \(F \simeq \R\) it is the usual absolute value, if \(F \simeq \C\) it is given by \(z \mapsto z \ol{z}\).
We will denote by \(\ol{F}\) a separable closure of \(F\) and \(\Gamma = \Gal(\ol{F}/F)\).
For a group scheme \(u\) of multiplicative type of finite type\footnote{We adopt the convention of \cite{SGA3_II}: group schemes of multiplicative types are not assumed to be of finite type.} over \(F\) we denote \(X^*(u) = \Hom(u_{\ol{F}}, \mathbb{G}_{m,\ol{F}})\), a finitely generated \(\Z\)-module with smooth action of \(\Gamma\).
For a torus \(T\) over \(F\) we also have \(X_*(T) = \Hom(\mathbb{G}_{m,\ol{F}}, T_{\ol{F}}) = \Hom(X^*(T),\Z)\).

\section{Representations of reductive groups} \label{sec:rep}

In this section we focus on the case where \(F\) is a non-Archimedean local field and occasionally indicate the differences for the Archimedean case.

\subsection{Setup}

Let \(G\) be a connected reductive group over \(F\).
We refer to \cite{Borel_lag} \cite{Springer_lag} \cite{BorelTits_gpesred} and \cite{SGA3_III} for fundamental results about reductive groups.
Let \(C\) be an algebraically closed field of characteristic zero, for example \(\C\) or \(\Qellbar\).
We consider \emph{smooth} representations of \(G(F)\) with coefficients in \(C\), i.e.\ pairs \((V,\pi)\) where \(V\) is a vector space over \(C\) and \(\pi: G(F) \to \GL(V)\) is a morphism of groups such the map
\begin{align*}
  G \times V & \longrightarrow{} V \\
  (g,v) & \longmapsto{} \pi(g) v
\end{align*}
is continuous for the natural topology on \(G\) and the discrete topology on \(V\).
If \(\pi\) is implicit we will also denote \(g \cdot v\) for \(\pi(g) v\).
Recall that such a representation is called \emph{admissible} if for any compact open subgroup \(K\) of \(G(F)\) the subspace
\[ V^K = \left\{ v \in V \,\middle|\, \forall k \in K,\, \pi(k)v=v \right\} \]
of \(V\) has finite dimension.
It is a non-trivial but well-known fact that any irreducible representation is admissible.
Denote by \(Z(G)\) the center of \(G\).
By a suitable generalization of Schur's lemma, any irreducible representation has a central character \(Z(G)(F) \to C^\times\).
For a smooth representation \((V,\pi)\) of \(G(F)\) its contragredient \((\tilde{V}, \tilde{\pi})\) is the space of \(K\)-finite linear forms on \(V\).

\begin{remark}
  In the case of an Archimedean field \(F\) we only consider coefficients \(C=\C\).
  The analogue of smooth representations are \((\mathfrak{g},K)\)-modules where \(\mathfrak{g} = \C \otimes_{\R} \Lie G(F)\) and \(K\) is a maximal compact subgroup of \(G(F)\).
  For many notions it is necessary to relate \((\mathfrak{g},K)\)-modules to continuous representations of \(G(F)\) on topological vector spaces.
  See e.g.\ \cite[\S 3.4]{Wallach_realred1} for the relation between the two notions in the case of unitary irreducible representations.
\end{remark}

\subsection{Parabolic induction, the Jacquet functor and supercuspidal representations}

Let \(P\) be a parabolic subgroup of \(G\).
Let \(N\) be the unipotent radical of \(P\) and \(M = P/N\) its reductive quotient.
Recall that there exists a section \(M \to P\), unique up to conjugation by \(N(F)\).
Let
\[ \delta_P(p) = ||\det ( \Ad(p) | \Lie(N))|| : M(F) \longrightarrow{} q^\Z\]
be the modulus character for the action of \(M(F)\) on \(N(F)\).
We choose a square root \(\sqrt{q}\) of \(q\) in \(C\), allowing us to define \(\delta_P^{1/2}\).
If \(C=\C\) we naturally choose \(\sqrt{q} \in \R_{>0}\).

Let \((V,\sigma)\) be a smooth representation of \(M(F)\), which we can see as a representation of \(P(F)\) trivial on \(N(F)\).
The normalized parabolically induced representation \(i_P^G \sigma\) is the space of locally constant functions \(f : G(F) \rightarrow V\) such that for any \(p \in P(F)\) and \(g \in G(F)\) we have \(f(p g) = \delta_P(p)^{1/2} \sigma(p) f(g)\), with left action of \(G(F)\) by \((g \cdot f)(x) = f(xg)\).
If \(\sigma\) is admissible (resp.\ has finite length) then \(i_P^G \sigma\) is admissible (resp.\ has finite length).
The introduction of \(\delta_P^{1/2}\) in the definition is motivated by the fact that if \(C=\C\) and \((V,\sigma)\) is unitary, i.e.\ endowed with a \(M(F)\)-invariant Hermitian inner product, then \(i_P^G \sigma\) has a natural \(G(F)\)-invariant Hermitian inner product.
In particular if \(\sigma\) is admissible and unitarizable then \(i_P^G \sigma\) is semi-simple.

For \((\pi,V)\) a smooth representation of \(G(F)\), denote by \(V_N\) the space of co-invariants for the action of \(N(F)\), which is naturally a smooth representation \(\pi_N\) of \(M(F)\).
The normalized Jacquet functor applied to \((\pi,V)\) is the smooth representation \(r_P^G \pi = \delta_P^{-1/2} \otimes \pi_N\) of \(M(F)\) on the space \(V_N\).
It also preserves admissibility and the property of being of finite length.

Recall that an irreducible smooth representation \((V,\pi)\) of \(G(F)\) is called supercuspidal if \(V_N = 0\) for any parabolic \(P = MN \subsetneq G\).
This is equivalent to all ``matrix coefficients''
\begin{align*}
  G(F) & \longrightarrow{} C \\
  g & \longmapsto \langle \pi(g) v,\tilde{v} \rangle
\end{align*}
for \(v \in V\) and \(\tilde{v} \in \tilde{V}\), being compactly supported modulo center.
Note that if \(\omega_\pi: Z(G(F)) \to C^\times\) is the central character of \(\pi\) then matrix coefficients of \(\pi\) are \(\omega_\pi\)-equivariant.

We recall in the following theorem the notion of supercuspidal support.

\begin{theorem} \label{thm:cusp_supp}
  Let \(\pi\) be an irreducible representation of \(G(F)\).
  \begin{enumerate}
  \item There exists a parabolic subgroup \(P = MN\) of \(G\) and a supercuspidal irreducible representation \(\sigma\) of \(M(F)\) such that \(\pi\) embeds in \(i_P^G \sigma\).
  \item If \(P' = M'N'\) is a parabolic subgroup of \(G\) and \(\sigma'\) is a supercuspidal irreducible representation of \(M'(F)\) then \(\pi\) is isomorphic to a subquotient of \(i_{P'}^G \sigma'\) if and only if there exists an element of \(G(F)\) conjugating \((M, \sigma)\) and \((M', \sigma')\).
  \end{enumerate}
\end{theorem}
\begin{proof}
  The first part is due to Jacquet: see \cite[Theorem 5.1.2]{CasselmanBook}.
  The second part seems to be due to Harish-Chandra: see Theorem 6.3.11 loc.\ cit.\ or \cite[Theorem 4.6.1, \S 5.3.1 and Theorem 5.4.4.1]{Silberger_intro_harm} for the ``if'' part.
  The ``only if'' part can be deduced from Bernstein center theory \cite{Deligne_centre_Bernstein}.
  See also \cite{BernsteinZelevinsky_ind1}.
\end{proof}

The \(G(F)\)-conjugacy class of \((M, \sigma)\) in the previous theorem is called the \emph{supercuspidal support} of \(\pi\).

\subsection{Asymptotic properties}

For the rest of this section we assume \(C = \C\).

\begin{definition}
  Let \((V,\pi)\) be a smooth irreducible representation of \(G(F)\).
  Let \(\omega_\pi: Z(G(F)) \to \C^\times\) be its central character.
  If \(\omega_\pi\) is unitary then we say that \(\pi\) is essentially square-integrable if all of its matrix coefficients are square-integrable modulo center:
  \[ \forall v \in V \ \forall \tilde{v} \in \tilde{V} \ \int_{G(F)/Z(G(F))} \left| \langle \pi(g) v, \tilde{v} \rangle \right|^2 dg < \infty. \]

  In general (without assuming that \(\omega_\pi\) is unitary) there is a unique smooth character \(\chi: G(F) \to \R_{>0}\) such that the central character of \(\chi \otimes \pi\) is unitary \cite[Lemma 5.2.5]{CasselmanBook}, and we say that \(\pi\) is essentially square-integrable if \(\chi \otimes \pi\) is.
\end{definition}

If \(\pi\) is an essentially square-integrable irreducible smooth representation of \(G(F)\) and if \(\omega_\pi\) is unitary then \(\pi\) is unitarizable.

Essential square-integrability can be checked on the Jacquet module of a representation, as recalled in Proposition \ref{pro:char_L2_rep} below.
For a Levi subgroup \(M\) of \(G\) we denote by \(A_M\) the largest split torus in the centre of \(M\).
Denote \(\mathfrak{a}_M^* = X^*(A_M) \otimes_\Z \R\).
We have an isomorphism
\begin{align} \label{eq:char_A_M}
  \mathfrak{a}_M^* & \longrightarrow{} \Hom_\cont(A_M(F), \R_{>0}) \\
  \chi \otimes s & \longmapsto{} (x \mapsto ||\chi(x)||^s). \nonumber
\end{align}

\begin{proposition}[{\cite[Proposition III.1.1]{WaldspurgerPlancherel}}] \label{pro:char_L2_rep}
  Let \((V,\pi)\) be an irreducible smooth representation of \(G(F)\).
  Assume that the central character of \(\pi\) is unitary (we can reduce to this case by twisting).
  Then \((V,\pi)\) is essentially square-integrable if and only if for every parabolic subgroup \(P=MN\) of \(G\), the absolute value of any character of \(A_M(F)\) occurring in \(r_P^G \pi\) is a linear combination with positive coefficients of the simple roots of \(A_M\) in \(N\) (via the isomorphism \eqref{eq:char_A_M}).
\end{proposition}

Replacing ``positive'' by ``non-negative'' in this characterization we get the notion of \emph{tempered} representation.
This is also equivalent to a growth condition on coefficients \cite[Proposition III.2.2]{WaldspurgerPlancherel}.

We have the following implications, for an irreducible smooth representation of \(G(F)\) having unitary central character:
\[ \text{supercuspidal} \Rightarrow{} \text{essentially square-integrable} \Rightarrow{} \text{tempered} \Rightarrow{} \text{unitarizable}. \]
For non-commutative \(G\) none of these implications is an equivalence.
The following result gives a coarse classification of tempered representations in terms of parabolic inductions of essentially square-integrable representations.

\begin{proposition}[{\cite[Proposition III.4.1]{WaldspurgerPlancherel}}] \label{pro:class_temp_L2}
  \begin{enumerate}
  \item Let \(P=MN\) be a parabolic subgroup of \(G\) and \(\sigma\) an essentially square-integrable irreducible smooth representation of \(M(F)\) having unitary central character.
    The induced representation \(i_P^G \sigma\) is semi-simple, has finite length and any irreducible subrepresentation is tempered.
  \item Let \((P, \sigma)\) and \((P', \sigma')\) be two pairs as in the previous point.
    Then \(i_P^G \sigma\) and \(i_{P'}^G \sigma'\) admit isomorphic irreducible subrepresentations if and only if the pairs \((M, \sigma)\) and \((M', \sigma')\) are conjugated by \(G(F)\), and in this case the two induced representations are isomorphic.
  \item For any tempered irreducible smooth representation \(\pi\) of \(G(F)\) there exists a pair \((P, \sigma)\) as in the first point such that \(\pi\) is isomorphic to a subrepresentation of \(i_P^G \sigma\).
  \end{enumerate}
\end{proposition}

\begin{remark}
  For \(G=\GL_n\), parabolically induced representations as in Proposition \ref{pro:class_temp_L2} are always irreducible \cite[\S 0.2]{Bernstein_unitaryGL} and so the proposition completely classifies tempered representations in terms of essentially square-integrable representations of smaller general linear groups.
  Recall that for general linear groups essentially square-integrable representations can be explicitly classified in terms of supercuspidal representations of Levi subgroups \cite[Theorem 9.3]{Zelevinsky_ind2}.

  For arbitrary \(G\) such induced representations are \emph{generically} irreducible (see \cite[Proposition IV.2.2]{WaldspurgerPlancherel} for a precise statement), but decomposing such induced representations is a subtle problem in general. 
\end{remark}

The tempered representations are exactly the ones occurring in Harish-Chandra's Plancherel formula (see \cite{WaldspurgerPlancherel}, \cite{Silberger_plancherel}), expressing the values of any locally constant and compactly supported \(f: G(F) \to \C\) (or more generally, a Schwartz function) in terms of the action of \(f\) in tempered representations (or expressing \(f(1)\) in terms of the traces of \(f\) in tempered representations).

Finally the ``Langlands classification'', that we recall below, classifies irreducible smooth representations of \(G(F)\) in terms of tempered representations of Levi subgroups.
For a connected reductive group \(M\) denote by \(X^*(M)\) the abelian group of morphisms \(M_{\ol{F}} \to \GL_{1,\ol{F}}\), so that \(X^*(M)^\Gamma\) is identified with the group of morphisms \(M \to \GL_{1,F}\).
The restriction morphism \(X^*(M)^\Gamma \to X^*(A_M)\) is an isogeny (it is injective with finite cokernel) and so it induces an isomorphism \(\res^M_{A_M}: X^*(M)^\Gamma \otimes_\Z \R \simeq \mathfrak{a}_M^*\).
We have an isomorphism
\begin{align} \label{eq:pos_char_M}
  X^*(M)^\Gamma \otimes_\Z \R & \longrightarrow{} \Hom_\cont(M(F), \R_{>0}) \\
  \chi \otimes s & \longmapsto{} (x \mapsto |\chi(x)|^s). \nonumber
\end{align}
Fix a minimal parabolic subgroup \(P_0\) of \(G\) and a Levi factor \(M_0\) of \(P_0\).
Let \(Y \subset X^*(A_{M_0})\) be the subgroup of characters which are trivial on \(A_{M_0} \cap G_\der\).
Let \(R(A_{M_0},G)\) be the set of roots of \(A_{M_0}\) in \(G\).
The rational Weyl group \(W_0 := N(A_{M_0}, G(F))/M_0(F)\) acts on \(A_{M_0}\), thus also on \(\mathfrak{a}_{M_0}^*\).
By \cite[Exposé XXVI Théorème 7.4]{SGA3_III}\footnote{See also \cite[Corollaire 5.8]{BorelTits_gpesred}, although the proof seems to be incomplete in the non-reduced case.} there is a unique root datum (possibly non-reduced)
\[ (X^*(A_{M_0}), R(A_{M_0},G), X_*(A_{M_0}), R^\vee(A_{M_0},G)) \]
such that the associated Weyl group (seen as a group of automorphisms of \(A_{M_0}\)) is \(W_0\).
Let \(\Delta \subset X^*(A_{M_0})\) be the set of simple roots for the order corresponding to \(P_0\).
Fix a \(W_0\)-invariant inner product \((\cdot,\cdot)\) on \(\mathfrak{a}_{M_0}^*\) (e.g.\ by averaging an arbitrary inner product).
For \(M\) a standard Levi subgroup of \(G\) the restriction map \(X^*(A_{M_0}) \to X^*(A_M)\) induces a surjective map \(\res^{A_{M_0}}_{A_M}: \mathfrak{a}_{M_0}^* \to \mathfrak{a}_M^*\).
We also have a composite map in the other direction
\[ j^M_{M_0}: \mathfrak{a}_M^* \xrightarrow{(\res^M_{A_M})^{-1}} X^*(M)^\Gamma \otimes_\Z \R \xrightarrow{\res^M_{M_0}} X^*(M_0)^\Gamma \otimes_\Z \R \xrightarrow{\res^{M_0}_{A_{M_0}}} \mathfrak{a}_{M_0}^* \]
and the composition \(\res^{A_{M_0}}_{A_M} \circ j^M_{M_0}\) is \(\id_{\mathfrak{a}_M^*}\).
In fact one can check that \(j^M_{M_0} \circ \res^{A_{M_0}}_{A_M}\) is the orthogonal projection \(\mathfrak{a}_{M_0}^* \to j^M_{M_0}(\mathfrak{a}_M^*)\).

\begin{theorem}[{\cite[Theorem 4.1]{Silberger_lang_class}, \cite[\S XI.2]{BorelWallach}, \cite[Theorem 3.11]{Dat_nutemp}}] \label{thm:Langlands_classification}
  \begin{enumerate}
  \item Let \(P\) be a standard Levi subgroup of \(G\) (with respect to \(P_0\)) and \(M\) its Levi factor containing \(M_0\).
    Let \(\sigma\) be a tempered irreducible smooth representation of \(M(F)\) (in particular we assume that its central character is unitary).
    Let \(\nu \in X^*(M)^\Gamma \otimes_\Z \R\) be such that for any \(\alpha \in \Delta\) not occurring in \(M\) we have \((\res^M_{A_{M_0}} \nu, \alpha) > 0\).
    Consider \(\nu\) as a character of \(M(F)\) via \eqref{eq:pos_char_M}, and denote by \(\sigma_\nu\) the twist of \(\sigma\) by this character.
    Then the induced representation \(i_P^G(\sigma_\nu)\) admits a unique irreducible quotient \(J(P,\sigma,\nu)\).
    Let \(\ol{P}\) be a parabolic subgroup of \(G\) which is opposite to \(P\).
    We have \(\dim_{\C} \Hom_G(i_P^G (\sigma_\nu), i_{\ol{P}}^G (\sigma_\nu)) = 1\) and any non-zero element in this line identifies \(J(P,\sigma,\nu)\) with the unique irreducible subrepresentation of \(i_{\ol{P}}^G(\sigma_\nu)\).
  \item Let \(\pi\) be an irreducible smooth representation of \(G(F)\).
    There exists a unique triple \((P,\sigma,\nu)\) as above such that \(\pi\) is isomorphic to \(J(P,\sigma,\nu)\).
  \end{enumerate}
\end{theorem}

The analogous theorem for \(F=\R\) was proved first (see \cite[Lemma 3.13, 3.14 and 4.2]{Langlands_class} or \cite[Chapter 5]{Wallach_realred1}), and inspired the non-Archimedean version.
The positivity condition appearing in Theorem \ref{thm:Langlands_classification} may be formulated using the absolute root system instead of the relative one, as the following lemma shows.
This will prove useful to translate this positivity condition on the dual side.

\begin{lemma} \label{lem:Langlands_class_abs}
  Let \(P = MN\) be a parabolic subgroup of \(G\) with unipotent radical \(N\) and Levi factor \(M\).
  Assume that \(M\) contains \(M_0\).
  Let \(T\) be a maximal torus of \(M_{\ol{F}}\).
  The following conditions on \(\nu \in X^*(M)^\Gamma \otimes_{\Z} \R\) are equivalent:
  \begin{itemize}
  \item for any root \(\alpha\) of \(A_{M_0}\) in \(N\) we have \((\res^M_{A_{M_0}} \nu, \alpha^\vee) > 0\),
  \item for any root \(\beta\) of \(T\) in \(N_{\ol{F}}\) we have \(\langle \res^M_T \nu, \beta^\vee \rangle > 0\).
  \end{itemize}
\end{lemma}
\begin{proof}
  In this proof we denote by \(\langle -, - \rangle_{A_{M_0}}\) (resp.\ \(\langle -, - \rangle_T\)) the canonical pairing \(X^*(A_{M_0}) \times X_*(A_{M_0}) \to \Z\) (resp.\ \(X^*(T) \times X_*(T) \to \Z\)), and \((-,-)_{A_{M_0}}\) for \((-,-)\).
  For \(\alpha \in R(A_{M_0}, G)\) with corresponding coroot \(\alpha^\vee \in X_*(A_{M_0})\) we have \(\langle -, \alpha^\vee \rangle_{A_{M_0}} = 2 (\alpha, \alpha)_{A_{M_0}}^{-1} (-, \alpha)_{A_{M_0}}\), so the first condition is equivalent to \(\langle \res^M_{A_{M_0}} \nu, \alpha^\vee \rangle_{A_{M_0}} > 0\) for all \(\alpha \in R(A_{M_0}, N)\), and does not depend on the choice of an invariant inner product \((-,-)_{A_{M_0}}\).
  This will allow us to choose a particular invariant inner product below.

  Up to conjugating \(T\) by \(M(\ol{F})\) (which leaves the second condition invariant) we may assume that \(T\) is a maximal torus of \(M_{0,\ol{F}}\), in particular \(A_{M_0,\ol{F}}\) is contained in \(T\).
  We denote by \(W\) the absolute Weyl group \(N(T, G(\ol{F}))/T(\ol{F})\).
  Choose a minimal parabolic subgroup \(P_0\) of \(G\) containing \(M_0\) and contained in \(P\) (this amounts to choosing a minimal parabolic subgroup of \(M\) containing \(M_0\)), and a Borel subgroup \(B \subset P_{0,\ol{F}}\) of \(G_{\ol{F}}\) containing \(T\) (this amounts to choosing a Borel subgroup of \(M_{0,\ol{F}}\) containing \(T\)).
  We have \cite[\S 6.2]{BorelTits_gpesred} an action of \(\Gamma\) (factoring through a finite Galois group) on \(X^*(T)\), leaving \(R(T,G_{\ol{F}})\) and its subset \(\Delta(T,B)\) of simple roots invariant.
  This action extends to an action of \(W \rtimes \Gamma\) on \(X^*(T)\).
  Choose an inner product \((-,-)_T\) on \(\mathfrak{t}^* := X^*(T) \otimes_{\Z} \R\) invariant under this action.
  Consider the restriction map \(\res^T_{A_{M_0}}: \mathfrak{t}^* \to \mathfrak{a}_{M_0}^*\) and its kernel \(K\).
  It identifies \(\mathfrak{a}_{M_0}^*\) with the orthogonal (for \((\cdot,\cdot)_T\)) of \(K\) in \(\mathfrak{t}^*\), and as explained in \cite[\S 6.10]{BorelTits_gpesred} it induces a \(W_0\)-invariant inner product \((\cdot,\cdot)_{A_{M_0}}\) on \(\mathfrak{a}_{M_0}^*\).
  More precisely \cite[\S 6.10]{BorelTits_gpesred} establishes this fact in the case where \(G\) is semi-simple, but we briefly explain how to reduce to this case.
  Denoting by \(Z(G)^0\) the connected center of \(G\), \(T_{\der} = T \cap G_{\der,\ol{F}}\) and \(A_{M_0,\der} = A_{M_0} \cap G_{\der}\) we have commutative diagrams
  \begin{equation} \label{eq:char_G_center_Tder_isog}
    \begin{tikzcd}
      X^*(G) \otimes_{\Z} \R \ar[r, "{\sim}" above, "{\res}" below] \ar[d, hook, "{\res}"] & X^*(Z(G)^0) \otimes_{\Z} \R \ar[d, hook] \\
      \mathfrak{t}^* \ar[r, "{\sim}" above, "{\res}" below] & X^*(Z(G)^0) \otimes_{\Z} \R \oplus \mathfrak{t}_{\der}^*
    \end{tikzcd}
  \end{equation}
  \begin{equation} \label{eq:char_G_center_Ader_isog}
    \begin{tikzcd}
      X^*(G)^\Gamma \otimes_{\Z} \R \ar[r, "{\sim}" above, "{\res}" below] \ar[d, hook, "{\res}"] & \mathfrak{a}_G^* \ar[d, hook] \\
      \mathfrak{a}_{M_0}^* \ar[r, "{\sim}" above, "{\res}" below] & \mathfrak{a}_G^* \oplus \mathfrak{a}_{M_0,\der}^*
    \end{tikzcd}
  \end{equation}
  where the right vertical maps are \(x \mapsto (x,0)\).
  The bottom isomorphism of \eqref{eq:char_G_center_Tder_isog} identifies \((\mathfrak{t}^*)^W\) with \(X^*(Z(G)^0) \otimes_{\Z} \R\), in particular the direct sum is orthogonal for \((\cdot,\cdot)_T\).
  The two diagrams \eqref{eq:char_G_center_Tder_isog} and \eqref{eq:char_G_center_Ader_isog} are part of an obvious commutative cubic diagram (all additional maps are restriction maps, except for the inclusion \(X^*(G)^\Gamma \otimes_{\Z} \R \to X^*(G) \otimes_{\Z} \R\)).
  This shows that the direct sum \(\mathfrak{a}_G^* \oplus \mathfrak{a}_{M_0,\der}^*\) is orthogonal for \((\cdot,\cdot)_{A_{M_0}}\) and since \cite[\S 6.10]{BorelTits_gpesred} shows that the restriction of this inner product to \(\mathfrak{a}_{M_0,\der}^*\) is \(W_0\)-invariant we deduce that \((\cdot,\cdot)_{A_{M_0}}\) is \(W_0\)-invariant

  Let \(C_0\) be the largest quotient of \(M_0\) which is a split torus, i.e.\ \(X^*(C_0) = X^*(M_0)^\Gamma\).
  We now check that the image of
  \[ \res^{C_0}_T: X^*(C_0) \otimes_{\Z} \R \longrightarrow \mathfrak{t}^* \]
  is precisely the orthogonal (for \((\cdot,\cdot)_T\)) of \(K\).
  Since composing with \(\res^T_{A_{M_0}}\) yields an isomorphism \(X^*(C_0) \otimes_{\Z} \R \simeq \mathfrak{a}_{M_0}^*\) it is enough to check that the image of \(\res^{C_0}_T\) is orthogonal to \(K\).
  It follows from \cite[Corollaire 6.9]{BorelTits_gpesred} (and consideration of the commutative diagrams \eqref{eq:char_G_center_Tder_isog} and \eqref{eq:char_G_center_Ader_isog}) that \(K\) is generated by the simple roots \(\beta \in \Delta(T,B)\) in \(K\), the differences \(\beta-\beta'\) where \(\beta, \beta' \in \Delta(T,B)\) are in the same Galois orbits, and the elements of \(X^*(G)^{N_\Gamma=0}\) (elements killed by averaging over \(\Gal(E/F)\) for some large enough finite Galois subextension \(E\) of \(\ol{F}/F\)).
  Let us check on these generators that an arbitrary \(\lambda \in X^*(C_0)\) is orthogonal to \(K\).
  \begin{itemize}
  \item If \(\beta \in \Delta(T,B)\) has trivial restriction to \(A_{M_0}\) then it belongs to \(R(T,M_{0,\ol{F}})\) and so \(\beta^\vee \in X_*(T)\) factors through \(T \cap (M_{0,\ol{F}})_{\der}\), so \(\langle \res^{C_0}_T \lambda, \beta^\vee \rangle_T\) vanishes.
    We have \(\langle -, \beta^\vee \rangle_T = 2 (\beta,\beta)_T^{-1} (\beta,-)_T\) because \((\cdot,\cdot)_T\) is \(W\)-invariant, so \((\res^{C_0}_T \lambda, \beta)_T = 0\).
  \item If \(\beta,\beta' \in \Delta(T,B)\) are in the same orbit under \(\Gamma\) then we have
    \[ (\beta, \res^{C_0}_T \lambda)_T = (\beta', \res^{C_0}_T \lambda)_T \]
    because \(\res^{C_0}_T \lambda\) is invariant under \(\Gamma\) and \((\cdot,\cdot)_T\) is \(\Gamma\)-invariant.
  \item Finally we have to check that for \(\mu \in X^*(G)^{N_\Gamma=0}\) the inner product \((\res^G_T \mu, \res^{C_0}_T \lambda)_T\) vanishes.
    This again follows from the \(\Gamma\)-invariance of \((\cdot,\cdot)_T\) and \(\res^{C_0}_T \lambda\) using the commutative diagram \eqref{eq:char_G_center_Tder_isog} and the fact that the direct sum in this diagram is orthogonal.
  \end{itemize}
  We have proved \(\operatorname{Im} \res^{C_0}_T = K^\perp\), which implies for \(\lambda \in X^*(C_0) \otimes_{\Z} \R\) and \(\beta \in \mathfrak{t}^*\) the equality
  \[ (\res^{C_0}_{A_{M_0}} \lambda, \res^T_{A_{M_0}} \beta)_{A_{M_0}} = (\res^{C_0}_T \lambda, \beta)_T. \]
  We apply this to \(\lambda = \nu\) and \(\beta \in R(T,N_{\ol{F}})\).
  Because the restriction map \(\res^T_{A_{M_0}}\) induces a surjective map \(R(T,N_{\ol{F}}) \to R(A_{M_0}, N)\) the equivalence between the two conditions is now clear.
\end{proof}

In these notes we say nothing of the natural question of classifying unitary representations of connected reductive groups.

\subsection{Harish-Chandra characters}

Denote by \(C^\infty_c(G(F))\) the space of locally constant and compactly supported functions \(G(F) \to \C\).
Recall that any such function is bi-invariant under some compact open subgroup of \(G(F)\).
Fix a Haar measure on \(G(F)\).
Let \((V,\pi)\) be an admissible representation of \(G(F)\).
Any \(f \in C^\infty_c(G(F))\) gives a linear map
\begin{align*}
  \pi(f): V & \longrightarrow{} V \\
  v & \longmapsto{} \int_{G(F)} f(g) \pi(g) v \ dg
\end{align*}
and its image is contained in \(V^K\) for some compact open subgroup \(K\) of \(G(F)\).
In particular \(\pi(f)\) has finite range and we may consider \(\Theta_\pi(f) = \tr \pi(f)\).
The linear form \(\Theta_\pi: C^\infty_c(G(F)) \to \C\) is called the Harish-Chandra character of \(\pi\).
By a standard result in representation theory of associative algebras \cite[\S 20.6]{Bourbaki_alg8} the Harish-Chandra characters \(\Theta_\pi\) of the irreducible smooth representations of \(G(F)\) (up to isomorphism) are linearly independent, and the Harish-Chandra character of a smooth representation of finite length determines its Jordan-Hölder constituents and their multiplicities.

Denote by \(G_\rs\) the regular semi-simple locus in \(G\), an open dense subscheme.
Recall that \(G(F) \smallsetminus G_\rs(F)\) has measure zero.

\begin{theorem}[{\cite[Theorem 16.3]{HC_adm_inv_dist}}]
  Assume that \(F\) is a non-Archimedean local field of characteristic zero.
  Let \((V,\pi)\) be an irreducible smooth representation of \(G(F)\).
  There exists a unique element of \(L^1_\mathrm{loc}(G(F))\), also denoted \(\Theta_\pi\), such that for any \(f \in C^\infty_c(G(F))\) we have
  \[ \tr \pi(f) = \int_{G(F)} \Theta_\pi(g) f(g) dg. \]
  Moreover \(\Theta_\pi\) is represented by a unique locally constant function on \(G_\rs(F)\).
\end{theorem}

To our knowledge this result is unfortunately not known in full generality in positive characteristic, but see  \cite{CluckersGordonHalupczok}.
Harish-Chandra characters behave well with respect to parabolic induction \cite{vanDijk_ind} and Jacquet functors \cite{Casselman_charJac}.

See \cite[Chapter 8]{Wallach_realred1} for the Archimedean case.

\section{Langlands dual groups}

We recall the definition of Langlands dual groups.
We refer to \cite[\S I.2]{Borel_autLfunc} for details not recalled below.
In this section \(F\) could be any field, \(\ol{F}\) is a separable closure of \(F\) and we denote \(\Gamma = \Gal(\ol{F}/F)\).

\subsection{Based root data}

Let \(G\) be a connected reductive group over \(F\).
There exists a finite separable extension \(E/F\) such that \(G_E\) admits a Killing pair (also called Borel pair) \((B,T)\) \cite[Exposé XXII Corollaire 2.4 and Proposition 5.5.1]{SGA3_III}.
We may and do assume that \(E/F\) is a subextension of \(\ol{F}/F\).
Associated to \((G_E, B, T)\) we have a based (reduced) root datum \((X,R,R^\vee,\Delta)\) where \(X\) is the group of characters of \(T\), \(R \subset X\) the set of roots of \(T\) in \(G_E\), \(R^\vee\) the set of coroots (a subset of \(X^\vee = \Hom(X,\Z)\), the group of cocharacters of \(T\)) and \(\Delta \subset R\) the set of simple roots corresponding to \(B\)\footnote{Strictly speaking we should also include in the datum the bijection \(R \to R^\vee\) as in \cite[Exposé XXI]{SGA3_III}, or include the orthogonal of \(R^\vee\) in \(X\) as in \cite[\S 2.1]{BorelTits_gpesred}.}.
The group \(G(E)\) acts (by conjugation) transitively on the set of Killing pairs in \(G_E\) (Exposé XXVI Corollaire 5.7 (ii) and Corollaire 1.8 loc.\ cit.) and the (scheme-theoretic) stabilizer of \((B,T)\) is \(T\) (Exposé XXII Cor 5.3.12 and Proposition 5.6.1 loc.\ cit.), which centralizes \(T\).
It follows that other choices of Killing pair in \(G_E\) yield based root data \emph{canonically} isomorphic\footnote{An isomorphism between two based root data \((X_1,R_1,R_1^\vee,\Delta_1)\) and \((X_2,R_2,R_2^\vee,\Delta_2)\) is an isomorphism of abelian groups \(X_1 \simeq X_2\) identifying \(R_1\) to \(R_2\), \(R_1^\vee\) to \(R_2^\vee\) and \(\Delta_1\) to \(\Delta_2\), and compatible with the bijections \(R_i \to R_i^\vee\).} to \((X,R,R^\vee,\Delta)\), and so do other choices for \(E\).

We also obtain a continuous action of \(\Gamma\) on this based root datum, that we now recall.
The group \(\Gal(E/F)\) acts on the set of closed subgroups of \(G_E\): if \(G = \Spec A\) for a Hopf algebra \(A\) over \(F\) and a closed subgroup \(H\) corresponds to an ideal \(I\) of \(A \otimes_F E\), then for \(\sigma \in \Gal(E/F)\) we let \(\sigma(H)\) be the closed subgroup corresponding to \(\sigma(I)\).
In particular we have \(\sigma(H)(E) = \sigma(H(E))\) as subgroups of \(G(E)\).
If \(K = \Spec B\) is a linear algebraic group over \(F\) and \(\lambda: H \to K_E\) is a morphism, dual to a morphism of Hopf algebras \(\lambda^\sharp: B \otimes_F E \to (A \otimes_F E)/I\), define \(\sigma(\lambda): \sigma(H) \to K_E\) as dual to
\[ \sigma \circ \lambda^\sharp \circ \sigma^{-1}: B \otimes_F E \longrightarrow (A \otimes_F E)/\sigma(I). \]
Now for \(\sigma \in \Gal(E/F)\) there is a unique \(T(E) g_\sigma \in T(E) \backslash G(E)\) such that we have \(\sigma(B,T) = \Ad(g_{\sigma}^{-1})(B,T)\), and we get a well-defined isomorphism \(\Ad(g_\sigma): \sigma(T) \simeq T\).
We obtain an action of \(\Gamma\) on \(X = X^*(T)\) such that \(\sigma \in \Gal(E/F)\) maps \(\lambda: T \to \GL_{1,E}\) to \(\sigma(\lambda) \circ \Ad(g_\sigma)^{-1}\).
It is straightforward to check that this action preserves \(R\) and \(\Delta\) and that the dual action on \(X^\vee\) preserves \(R^\vee\).
It is routine to check that if we choose another triple \((E',B',T')\) instead of \((E,B,T)\) to obtain a based root datum \((X',R',R^{',\vee},\Delta')\) with continuous action of \(\Gamma\), the canonical isomorphism between \((X,R,R^\vee,\Delta)\) and \((X',R',R^{',\vee},\Delta')\) is \(\Gamma\)-equivariant.
We denote by \(\brd_F\) the\footnote{Defining this functor entails choosing a triple \((E,B,T)\) for each connected reductive group \(G\) over \(F\), but of course other choices would yield a canonically isomorphic functor.} resulting functor from the groupoid of connected reductive groups over \(F\) to the groupoid of based root data with continuous action of \(\Gamma\).

\begin{definition}
  Let \(G\) be a connected reductive group over \(F\).
  Define a groupoid \(\mathcal{IT}(G)\) as follows.
  \begin{itemize}
  \item The objects of \(\mathcal{IT}(G)\) are the inner twists of \(G\), i.e.\ pairs \((G', \psi)\) consisting of a connected reductive group \(G'\) over \(F\) and an isomorphism \(\psi: G_{\ol{F}} \simeq G_{\ol{F}}'\) such that for any \(\sigma \in \Gamma\) the automorphism \(\psi^{-1} \sigma(\psi)\) of \(G_{\ol{F}}\) is inner.
  \item A morphism between two inner twists \((G_1, \psi_1)\) and \((G_2, \psi_2)\) of \(G\) is an element \(g \in G_\ad(\ol{F})\) such that for any \(\sigma \in \Gamma\) we have
    \begin{equation} \label{eq:morphism_inner_twists}
      \psi_2^{-1} \sigma(\psi_2) = \Ad(g) \psi_1^{-1} \sigma(\psi_1) \Ad(\sigma(g))^{-1}.
    \end{equation}
  \end{itemize}
\end{definition}

\begin{remark}
  \noindent\par
  \begin{enumerate}
  \item One can check that any inner twist \(\psi: G_{\ol{F}} \to G'_{\ol{F}}\) yields a canonical isomorphism \(\brd_F(G) \simeq \brd_F(G')\).
  \item For an inner twist \(\psi: G_{\ol{F}} \to G'_{\ol{F}}\) the map
    \begin{align*}
      \Gamma & \longrightarrow{} G_\ad(\ol{F}) \\
      \sigma & \longmapsto{} \psi^{-1} \sigma(\psi)
    \end{align*}
    is a \(1\)-cocycle, i.e.\ an element of \(Z^1_\cont(\Gamma, G_\ad) = Z^1(F, G_\ad)\).
  \item The relations \eqref{eq:morphism_inner_twists} imply that the isomorphism
    \[ \psi_2 \Ad(g) \psi_1^{-1}: G_{1,\ol{F}} \longrightarrow{} G_{2,\ol{F}} \]
    is defined over \(F\), i.e.\ descends to an isomorphism \(G_1 \simeq G_2\).
  \item For an inner twist \((G',\psi)\) of \(G\) we have an isomorphism
    \begin{align*}
      \Aut_{\mathcal{IT}(G)}(G',\psi) & \longrightarrow{} G'_\ad(F) \\
      g & \longmapsto{} \psi(g).
    \end{align*}
  \end{enumerate}
\end{remark}
  
\begin{proposition} \label{pro:inner_forms_brd}
  Let \(b\) be a based root datum with continuous action of \(\Gamma\).
  Let \(\mathcal{CRG}_b\) be the groupoid of pairs \((G, \alpha)\) where \(G\) is a connected reductive group over \(F\) and \(\alpha: b \simeq \brd_F(G)\) is an isomorphism of based root data with action of \(\Gamma\), with obvious morphisms.
  (In other words \(\mathcal{CRG}_b\) is the groupoid fiber of \(b\) for \(\brd_F\).)
  \begin{enumerate}
  \item There exists an object \((G^*, \alpha^*)\) of \(\mathcal{CRG}_b\) such that \(G^*\) is quasi-split.
    Two such objects are isomorphic.
  \item Any object \((G, \alpha)\) of \(\mathcal{CRG}_b\) yields equivalences of groupoids
    \[ Z^1(F, G_\ad) \xleftarrow{\sim} \mathcal{IT}(G) \xrightarrow{\sim} \mathcal{CRG}_b. \]
    This gives in particular a bijection between \(H^1(F, G_\ad)\) and the set of isomorphism classes in \(\mathcal{CRG}_b\).
  \end{enumerate}
\end{proposition}
\begin{proof}
  This is a reformulation of \cite[Exposé XXIV Théorème 3.11]{SGA3_III} in the case where the base is the spectrum of a field, also using 3.9.1 loc.\ cit.\ to prove uniqueness in the first point.
\end{proof}

To sum up, we can ``classify'' connected reductive groups over \(F\) as follows:
\begin{itemize}
\item fix a representative in each isomorphism class of based root datum with continuous action of \(\Gamma\),
\item for each such representative \(b\), fix a quasi-split connected reductive group \(G^*\) over \(F\) together with an isomorphism \(\brd_F(G^*) \simeq b\),
\item for each element of \(H^1(F, G^*_\ad)\) choose an inner twist \((G, \psi)\) of \(G^*\) representing it.
\end{itemize}
Up to isomorphism each connected reductive group \(G\) over \(F\) arises in this way.
It can happen that an isomorphism class of connected reductive groups arises more than once, because \(H^1(F, G_\ad) \to H^1(F, \Aut(G))\) is not injective in general.
Nevertheless for \(G\) quasi-split Proposition \ref{pro:inner_forms_brd} implies that the preimage of the trivial class in \(H^1(F, \Aut(G))\) is a singleton.

\subsection{Langlands dual groups} \label{sec:Lgroups}

Let \(C\) be an algebraically closed field of characteristic zero.
Let \(G\) be a connected reductive group over \(F\) and let \(\brd_F(G) = (X,R,R^\vee,\Delta)\) be its associated based root datum endowed with a continuous action of \(\Gamma\).
Let \((\Ghat, \mathcal{B}, \mathcal{T}, (X_\alpha)_{\alpha \in \Delta^\vee})\) be the pinned connected reductive group over \(C\) with associated based root datum \((X^\vee,R^\vee,R,\Delta^\vee)\), i.e.\ the \emph{dual} of \(\brd_F(G)\) (ignoring the action of \(\Gamma\) for now).
The choice of a pinning induces a splitting of the extension
\[ 1 \rightarrow \Ghat_{\ad} \rightarrow \Aut (\Ghat) \rightarrow  \Out(\Ghat) \rightarrow 1 \]
because the subgroup \(\Aut(\Ghat, \mathcal{B}, \mathcal{T}, (X_{\alpha})_{\alpha \in \Delta^{\vee}})\) of \(\Aut(\Ghat)\) maps bijectively onto \(\Out(\Ghat)\) \cite[Exposé XXIV Théorème 1.3]{SGA3_III}.
As explained loc.\ cit.\ we also have an isomorphism
\[ \Out(\Ghat) \simeq \Aut(X^\vee, R^\vee, R, \Delta^\vee) \simeq \Aut(X, R, R^\vee, \Delta) \]
and so we have an action of \(\Gamma\) on \(\Ghat\) (preserving the pinning and factoring through a finite Galois group).
Denote \({}^L G = \Ghat \rtimes \Gamma\) the Langlands dual group, also called L-group.
It is sometimes useful (or just convenient) to replace \(\Gamma\) by a finite Galois group or by the Weil group in this semi-direct product.

One can give a more pedantic definition of Langlands dual groups in order to avoid the inelegant choice of pinning.
Namely, define an L-group for \(G\) as an extension \({}^L G\) of \(\Gamma\) by \(\Ghat\), where \(\Ghat\) is a split connected reductive group endowed with an isomorphism of its based root datum with the dual of that of \(G\), such that the induced morphism \(\Gamma \rightarrow \Out(\Ghat)\) is as above, and endowed with a \(\Ghat\)-conjugacy class of splittings \(\Gamma \rightarrow {}^L G\), called distinguished splittings, such that any (equivalently, one) of these splittings \(s\) preserves a pinning of \(\Ghat\).
It is not necessary to specify the pinning, since for a distinguished splitting \(s\) we have that \(\Ghat^{s(\Gamma)}\) acts transitively on the set of such pinnings: see \cite[Corollary 1.7]{Kottwitz_STFcusptemp}.
In the other direction, for a pinning of \(\Ghat\) fixed by a distinguished splitting \(s\), the set of distinguished splittings fixing this pinning is parametrized by
\begin{equation} \label{eq:param_splittings_pinning}
  \ker \left( Z^1(\Gamma, Z(\Ghat)) \to H^1(\Gamma, s, \Ghat) \right)
\end{equation}
where the notation \(H^1(\Gamma, s, \Ghat)\) means the first cohomology set for the action of \(\Gamma\) on \(\Ghat\) via \(s\) and conjugation in \({}^L G\).
Note that all distinguished splittings induce the same action of \(\Gamma\) on \(Z(\Ghat)\).
By Lemma 1.6 loc.\ cit.\ the kernel \eqref{eq:param_splittings_pinning} is simply the group of coboundaries \(B^1(\Gamma, Z(\Ghat))\), and so the distinguished splittings fixing a given pinning of \(\Ghat\) form a single conjugacy class by \(Z(\Ghat)\).

By Proposition \ref{pro:inner_forms_brd} for two connected reductive groups \(G_1\) and \(G_2\) their Langlands dual groups \({}^L G_1\) and \({}^L G_2\) are isomorphic as extensions of \(\Gamma\) if and only if \(G_1\) and \(G_2\) are inner forms of each other, and in this case they are even isomorphic as extensions endowed with conjugacy classes of distinguished splittings.

The construction of the Langlands dual group is not functorial for arbitrary morphisms between connected reductive groups, however in the following cases functoriality is straightforward.
\begin{itemize}
\item Let \(G\) be a quasi-split connected reductive group and \((B,T)\) a Borel pair (defined over \(F\)).
  Choose a distinguished splitting \(s_G: \Gamma \to {}^L G\) preserving a pinning \((\mathcal{B}, \mathcal{T}, (X_\alpha)_\alpha)\) of \(\Ghat\) and a distinguished splitting \(s_T: \Gamma \to {}^L T\).
  Then the canonical isomorphism \(\That \simeq \mathcal{T}\) extends to an embedding \({}^L T \hookrightarrow {}^L G\) whose composition with \(s_T\) is \(s_G\).
\item For \(G = G_1 \times_F G_2\) we can identify \({}^L G\) with \({}^L G_1 \times_\Gamma {}^L G_2\).
\item A central isogeny \cite[Exposé XXII Définition 4.2.9]{SGA3_III} \(G \rightarrow H\) induces a surjective morphism with finite kernel \({}^L H \rightarrow {}^L G\).
  More generally one can associate to a morphism \(G \to H\) with central kernel, normal image and abelian cokernel a morphism \({}^L H \rightarrow {}^L G\) (reduce to the case of a central isogeny using the previous point).
\item There are weaker forms of functoriality.
  Let \(G\) be a connected reductive group and \(T\) a maximal torus of \(G\) defined over \(F\).
  Choose a Borel subgroup \(B\) of \(G_{\ol{F}}\) containing \(T_{\ol{F}}\) and a splitting \(s: \Gamma \to {}^L G\) preserving a pinning \((\mathcal{B}, \mathcal{T}, (X_\alpha)_\alpha)\) of \(\Ghat\).
  We have a canonical isomorphism \(\That \simeq \mathcal{T}\) (coming from the isomorphism between \(\brd_C(\Ghat)\) and \(\brd_F(G)^\vee\)), but the Galois actions differ by a \(1\)-cocycle taking values in the Weyl group.
  In general we don't have a canonical embedding \({}^L T \hookrightarrow {}^L G\) (see \cite[\S 2.6]{LanglandsShelstad} and \cite{Kaletha_Lemb_tori} however), but note that the induced embedding \(Z(\Ghat) \hookrightarrow \That\) is \(\Gamma\)-equivariant.
\end{itemize}

In the next section we recall how the first case generalizes to parabolic subgroups in arbitrary connected reductive groups.

\subsection{Parabolic subgroups and L-embeddings} \label{sec:parabolic}

A parabolic subgroup \(\mathcal{P}\) of \({}^L G\) is a closed subgroup mapping onto \(\Gamma\) and such that \(\mathcal{P}^0 := \mathcal{P} \cap \Ghat\) is a parabolic subgroup of \(\Ghat\).
The set of parabolic subgroups is clearly stable under conjugation by \(\Ghat\).
If \(\mathcal{P}\) is a parabolic subgroup of \({}^L G\) then \(\mathcal{P}\) is the normalizer of \(\mathcal{P}^0\) in \({}^L G\).

Choose a Killing pair \((\mathcal{B}, \mathcal{T})\) of \(\Ghat\).
Recall that a parabolic subgroup of \(\Ghat\) is conjugated to a unique one containing \(\mathcal{B}\), and that parabolic subgroups of \(\Ghat\) containing \(\mathcal{B}\) correspond bijectively to subsets of \(\Delta^{\vee}\) (or \(\Delta\), using the bijection \(\alpha \mapsto \alpha^\vee\)), by associating to \(\mathcal{P}^0\) the set of \(\alpha \in \Delta^{\vee}\) (seen as characters of \(\mathcal{T}\)) such that \(-\alpha\) is a root of \(\mathcal{T}\) in \(\mathcal{P}^0\).
Embed \(\mathcal{B}\) in a pinning \((\mathcal{B}, \mathcal{T}, (X_{\alpha})_{\alpha \in \Delta^{\vee}})\) of \(\Ghat\), and let \(s: \Gamma \to {}^L G\) be a distinguished section fixing this pinning.
Then \(\mathcal{B} s(\Gamma)\) is a (minimal) parabolic subgroup of \({}^L G\), and any parabolic subgroup of \({}^L G\) is conjugated under \(\Ghat\) to one containing \(\mathcal{B} s(\Gamma)\).
A parabolic subgroup \(\mathcal{P}^0\) of \(\Ghat\) containing \(\mathcal{B}\) is such that its normalizer \(\mathcal{P}\) in \({}^L G\) maps onto \(\Gamma\) (i.e.\ \(\mathcal{P}\) is a parabolic subgroup of \({}^L G\)) if and only if the corresponding subset of \(\Delta^\vee\) is stable under \(\Gamma\).
Therefore \(\Ghat\)-conjugacy classes of parabolic subgroups of \({}^L G\) also correspond bijectively to \(\Gamma\)-stable subsets of \(\Delta^{\vee}\).

Using the bijection between \(\Delta\) and \(\Delta^\vee\) we obtain a bijection between the set of \(\Gamma\)-stable \(G(\ol{F})\)-conjugacy classes of parabolic subgroups of \(G_{\ol{F}}\) and the set of \(\Ghat\)-conjugacy classes of parabolic subgroups of \({}^L G\).
The obvious map from the set of \(G(F)\)-conjugacy classes of parabolic subgroups of \(G\) to the set of \(\Gamma\)-stable \(G(\ol{F})\)-conjugacy classes of parabolic subgroups of \(G_{\ol{F}}\) is injective, and it is surjective if and only if \(G\) is quasi-split.

Recall from \cite[\S 3.4]{Borel_autLfunc} that if \(\mathcal{P}\) is a parabolic subgroup of \({}^L G\) and \(\mathcal{M}^0\) is a Levi factor of \(\mathcal{P}^0\) then the normalizer \(\mathcal{M}\) of \(\mathcal{M}^0\) in \(\mathcal{P}\) maps onto \(\Gamma\) and \(\mathcal{P}\) is the semi-direct product of its unipotent radical and \(\mathcal{M}\).
In this situation we say that \(\mathcal{M}\) is a Levi factor of \(\mathcal{P}\), and a Levi subgroup of \({}^L G\).
See Lemma 3.5 loc.\ cit.\ for another characterization of Levi subgroups.

Let \(P\) be a parabolic subgroup of \(G\).
Choose a distinguished splitting \(s: \Gamma \to {}^L G\) stabilizing a pinning \(\E = (\mathcal{B}, \mathcal{T}, (X_{\alpha})_{\alpha \in \Delta^{\vee}})\) of \(\Ghat\).
Let \(M=P/N\) be the reductive quotient of \(P\).
Taking Killing pairs inside \(P\) in the definition of \(\brd_F\) we obtain an isomorphism between \(\brd_F(M)\) and \((X, R_P, R^\vee_P, \Delta_P)\) where \(\Delta_P\) is the set of simple roots \(\alpha \in \Delta\) such that \(-\alpha\) also occurs in \(P\), \(R_P = R \cap \mathrm{span}(\Delta_P)\), \(\Delta_P^\vee = \{ \alpha^\vee \,|\, \alpha \in \Delta_P \}\) and \(R_P^\vee = R^\vee \cap \mathrm{span}(\Delta_P^\vee)\).
Let \(\E_M = (\mathcal{B}_M, \mathcal{T}_M, (Y_\alpha)_\alpha)\) be a pinning of \(\Mhat\) and \(s_M: \Gamma \to {}^L M\) a corresponding distinguished splitting.
These choices determine an embedding of extensions of \(\Gamma\)
\[ \iota[P, \E, s, \E_M, s_M]: {}^L M \longrightarrow {}^L G \]
characterized by the following properties.
\begin{itemize}
\item It maps \((\mathcal{B}_M, \mathcal{T}_M)\) to \((\mathcal{B}, \mathcal{T})\), and on \(\mathcal{T}_M\) it is the isomorphism \(\mathcal{T}_M \simeq \mathcal{T}\) induced by the above embedding \(\brd_F(M) \hookrightarrow \brd_F(G)\),
\item it maps \(\E_M\) to \(\E\), and
\item we have \(\iota[P, \E, s, \E_M, s_M] \circ s_M = s\).
\end{itemize}
The image of \(\iota[P, \E, s, \E_M, s_m]\) is clearly a Levi subgroup of \({}^L G\), and the subgroup of \({}^L G\) generated by this image and \(\mathcal{B}\) is the parabolic subgroup of \({}^L G\) containing \(\mathcal{B}\) corresponding to \(P\).
The formation of \(\iota[P, \E, s, \E_M, s_M]\) satisfies obvious equivariance properties with respect to conjugation by \(\Mhat\) and \(\Ghat\).
In particular we have an embedding \(\iota_P: {}^L M \to {}^L G\) well-defined up to conjugation by \(\Ghat\).

\begin{lemma} \label{lem:conj_Lemb_Levi}
  Let \(M\) be a Levi subgroup of \(G\).
  Let \(P\) and \(P'\) be parabolic subgroups of \(G\) admitting \(M\) as a Levi factor.
  Then \(\iota_P\) and \(\iota_{P'}\) are conjugate under \(\Ghat\).
\end{lemma}
This statement is contained in \cite[Lemma 2.5]{Langlands_class} but we give a self-contained proof.
\begin{proof}
  First we recall a general construction.
  Fix a pinning \(\E = (\mathcal{B}, \mathcal{T}, (X_\alpha)_\alpha)\) in \(\Ghat\) and a distinguished splitting \(s: \Gamma \to {}^L G\) fixing it.
  For a Killing pair \((B,T)\) in \(G_{\ol{F}}\) we denote by \(\gamma[(B,T), (\mathcal{B}, \mathcal{T})]\) the isomorphism \(X^*(\mathcal{T}) \simeq X_*(T)\).
  Considering Weyl groups inside automorphism groups of tori this also induces an isomorphism
  \[ \omega[(B,T),(\mathcal{B},\mathcal{T})]: W(T,G_{\ol{F}}) \simeq W(\mathcal{T},\Ghat) \]
  and \(\gamma[(B,T), (\mathcal{B}, \mathcal{T})]\) is equivariant for the Weyl group actions via this isomorphism.
  We have an action of \(\Gamma\) on \(W(T,G_{\ol{F}})\): for \(\sigma \in \Gamma\) let \(T(\ol{F}) g_\sigma \in T(\ol{F}) \backslash G(\ol{F})\) be the class for which \(\sigma(B,T) = \Ad(g_\sigma^{-1})(B,T)\), then \(x \mapsto \Ad(g_\sigma)(\sigma(x))\) induces an automorphism of \(W(T,G_{\ol{F}})\).
  One can check that the isomorphism \(\omega[(B,T),(\mathcal{B},\mathcal{T})]\) is \(\Gamma\)-equivariant for this action on \(W(T,G_{\ol{F}})\) and the action via \(s\) on \(W(\mathcal{T}, \Ghat)\).
  
  Fix \(\E\), \(s\), \(\E_M\) and \(s_M\) as above.
  Fix a Borel pair \((B_M, T)\) in \(M_{\ol{F}}\).
  This determines two Borel subgroups \(B\) and \(B'\) in \(G_{\ol{F}}\), characterized by the properties \(B \cap M_{\ol{F}} = B_M\) and \(N_{\ol{F}} \subset B\) and similarly for \(B'\).
  There is a unique \(x \in W(T, G_{\ol{F}})\) for which \(\Ad(x)(B,T) = (B',T)\).
  Let \(n: W(\mathcal{T}, \Ghat) \to N(\mathcal{T}, \Ghat)\) be the set-theoretic splitting determined by \(\E\) \cite[\S 9.3.3]{Springer_lag}.
  Denote \(w = n(\omega[(B,T),(\mathcal{B},\mathcal{T})](x))\).
  We claim that we have
  \begin{equation} \label{eq:conj_Lemb_Levi}
    \Ad(w) \circ \iota[P', \E, s, \E_M, s_M] = \iota[P, \E, s, \E_M, s_M].
  \end{equation}
  To simplify notation in the rest of the proof we abbreviate \(\iota = \iota[P, \E, s, \E_M, s_M]\) and \(\iota' = \iota[P', \E, s, \E_M, s_M]\).
  
  First we check that \(\iota\) and \(\Ad(w)^{-1} \circ \iota'\) coincide on \(\mathcal{T}_M\).
  We have \((B',T) = \Ad(x)(B,T)\) so if we also denote by \(\Ad(x)\) the induced automorphism of \(X_*(T)\) we have \(\Ad(x) \gamma[(B,T), (\mathcal{B}, \mathcal{T})] = \gamma[(B',T), (\mathcal{B}, \mathcal{T})]\).
  We obtain
  \[ \gamma[(B',T), (\mathcal{B}, \mathcal{T})] = \gamma[(B,T), (\mathcal{B}, \mathcal{T})] \circ \omega[(B,T),(\mathcal{B},\mathcal{T})](x). \]
  The isomorphism \(\iota|_{\mathcal{T}_M}: \mathcal{T}_M \simeq \mathcal{T}\) is dual to the isomorphism
  \[ \gamma[(B_M,T), (\mathcal{B}_M, \mathcal{T}_M)]^{-1} \circ \gamma[(B,T), (\mathcal{B},\mathcal{T})] : X^*(\mathcal{T}) \simeq X^*(\mathcal{T}_M). \]
  Similarly \(\iota'|_{\mathcal{T}_M}: \mathcal{T}_M \simeq \mathcal{T}\) is dual to the isomorphism
  \begin{align*}
    & \gamma[(B_M,T), (\mathcal{B}_M, \mathcal{T}_M)]^{-1} \circ \gamma[(B',T), (\mathcal{B},\mathcal{T})] \\
    =& \gamma[(B_M,T), (\mathcal{B}_M, \mathcal{T}_M)]^{-1} \circ \gamma[(B,T), (\mathcal{B},\mathcal{T})] \circ \omega[(B,T),(\mathcal{B},\mathcal{T})](x)
  \end{align*}
  and the equality
  \[ \iota'|_{\mathcal{T}_M} = \omega[(B,T),(\mathcal{B},\mathcal{T})](x)^{-1} \circ \iota|_{\mathcal{T}_M} = \Ad(w)^{-1} \iota|_{\mathcal{T}_M} \]
  follows.
  
  To check that the equality \eqref{eq:conj_Lemb_Levi} holds on \(\Mhat\) it is enough to check that we have \(\iota(Y_\alpha) = \Ad(w) \iota'(Y_\alpha)\) for any \(\alpha \in \Delta(\mathcal{T}_M, \mathcal{B}_M)\).
  We have
  \[ \iota(Y_\alpha) = X_\beta \text{ and } \iota'(Y_\alpha) = X_{\beta'} \]
  where
  \begin{align*}
    \beta &= \gamma[(B,T),(\mathcal{B},\mathcal{T})]^{-1} \gamma[(B_M,T),(\mathcal{B}_M,\mathcal{T}_M)](\alpha), \\
    \beta' &= \gamma[(B',T),(\mathcal{B},\mathcal{T})]^{-1} \gamma[(B_M,T),(\mathcal{B}_M,\mathcal{T}_M)](\alpha) \\
    &= w^{-1}(\beta)
  \end{align*}
  both belong to \(\Delta(\mathcal{T}, \mathcal{B})\).
  By \cite[Proposition 9.3.5]{Springer_lag} we have \(X_\beta = \Ad(w)(X_{\beta'})\).

  Finally we need to check \(\Ad(w) \circ s = s\), i.e.\ that \(w\) commutes with \(s(\Gamma)\).
  For \(\sigma \in \Gamma\) and \(y \in W(\mathcal{T}, \Ghat)\) we have \(s(\sigma) n(y) s(\sigma)^{-1} = n(\sigma(y))\) and so it is enough to check that \(w \mathcal{T} \in W(\mathcal{T}, \Ghat)\) is fixed by \(\Gamma\).
  For any \(\sigma \in \Gamma\) there exists \(g_\sigma \in M(\ol{F})\) such that \(\sigma(B_M,T) = \Ad(g_\sigma^{-1})(B_M,T)\) and this implies \(\sigma(B,T) = \Ad(g_\sigma^{-1})(B,T)\) and \(\sigma(B',T) = \Ad(g_\sigma^{-1})(B',T)\) because \(N\) and \(N'\) are both defined over \(F\).
  A simple computation shows that we have \(\Ad(g_\sigma)(\sigma(x)) = x\) in \(W(T,G_{\ol{F}})\), i.e.\ \(x\) is \(\Gamma\)-invariant.
\end{proof}

The lemma shows that for a Levi subgroup \(M\) of \(G\) we have an embedding \(\iota_M: {}^L M \to {}^L G\), well-defined up to conjugation by \(\Ghat\).
We call the image of such an embedding a \emph{\(G\)-relevant} Levi subgroup of \({}^L G\).

\section{Langlands parameters} \label{sec:langlands_param}

In this section \(F\) is a local field.

\subsection{Weil-Deligne groups} \label{sec:WD}

We briefly recall the definition of Weil-Deligne groups of local fields.
We refer the reader to \cite{Tate_ntbackground} for more details.

If \(F \simeq \C\) define \(\W_F = F^{\times}\).
If \(F \simeq \R\) define \(\W_F\) as the unique non-split central extension
\[ 1 \rightarrow \ol{F}^{\times} \rightarrow \W_F \rightarrow \Gal(\ol{F}/F) \rightarrow 1 \]
where \(\Gal(\ol{F}/F)\) acts on \(\ol{F}^{\times}\) in the natural way.
Explicitly, \(\W_F = \ol{F}^{\times} \sqcup j \ol{F}^{\times}\) with \(j^2 = -1\).

If \(F\) is a non-Archimedean local field, we have a short exact sequence of topological groups
\[ 1 \rightarrow I_F \rightarrow \Gal( \ol{F} / F) \rightarrow \Gal(\ol{k} / k) \rightarrow 1  \]
where \(k\) is the residue field of \(F\) and \(I_F\) is called the inertia subgroup of \(\Gal(\ol{F} / F)\).
Since \(k\) is finite, say of cardinality \(q\), \(\Gal(\ol{k} / k)\) is isomorphic to \(\widehat{\Z}\) and topologically generated by the Frobenius automorphism \(x \mapsto x^q\).
This automorphism generates a natural subgroup \(\Z\) of \(\Gal(\ol{k} / k)\), and the Weil group \(\W_F\) is defined as its preimage, a dense subgroup of \(\Gal(\ol{F} / F)\).
Instead of the induced topology, we endow \(\W_F\) with the topology making \(I_F\) an open subgroup, with its topology induced from that of \(\Gal(\ol{F}/F)\).

Recall that the Artin reciprocity map is an isomorphism \(\W_F^\ab \simeq F^\times\).
Composing with the norm \(||\cdot||: F^\times \to \R_{>0}\) we get a continuous morphism still denoted \(||\cdot||: \W_F \to \R_{>0}\).

For non-Archimedean \(F\), we now recall three possible definitions for the Weil-Deligne group.
\begin{enumerate}
	\item \(\W_F' := \mathbb{G}_a \rtimes \W_F\), where the additive group \(\mathbb{G}_a\) is defined over \(C\) and the action of \(\W_F\) on \(\mathbb{G}_a\) is by \(w(x) = ||w|| x\).
	\item \(\WD_F := \W_F \times \SL_2\), where the second factor is the algebraic group over \(\Q\).
	\item the (unnamed) locally compact topological group \(\W_F \times \mathrm{SU}(2)\), where \(\mathrm{SU}(2)\) is a maximal compact subgroup of \(\SL_2(\C)\).
\end{enumerate}

For Archimedean \(F\) it will be convenient to denote \(\WD_F = \W_F\).

\subsection{Langlands parameters}

First assume that \(F\) is non-Archimedean.

For the first version of the Weil-Deligne group, a Weil-Deligne Langlands parameter\footnote{This terminology is not standard.} is a pair \((\rho, N)\) such that
\begin{itemize}
\item \(\rho: \W_F \to {}^L G\) is a continuous representation, i.e.\ there exists an open subgroup \(U\) of \(I_F\) which acts trivially on \(\Ghat\) and is mapped to \(1 \times U \subset \Ghat \rtimes \Gamma\), such that the composition with the projection \({}^L G \to \Gamma\) is the usual map,
\item \(N \in \Lie \Ghat\) satisfies \(\rho(w) N \rho(w)^{-1} = ||w|| N\) for all \(w \in \W_F\) (this forces \(N\) to be nilpotent),
\item for any \(w \in \W_F\) (equivalently, for some \(w \in \W_F \smallsetminus \I_F\)) we have that \(\rho(w)\) is semi-simple.
\end{itemize}

\begin{remark} \label{rem:ss_via_linear_quot}
  For this last condition note that the group \({}^L G\) is not of finite type over \(C\) but there is a finite Galois subextension \(E/F\) of \(\ol{F}/F\) such that the action of \(\Gamma_E\) on \(\Ghat\) (via any distinguished splitting \(s\)) is trivial, and the quotient of \({}^L G\) by \(s(\Gamma_E)\) (which again does not depend on the choice of \(s\)) may be identified with a semi-direct product \(\Ghat \rtimes \Gal(E/F)\), which is of finite type.
  We say that an element \(x\) of \({}^L G\) is semi-simple if its image in \({}^L G / s(\Gamma_E)\) is semi-simple.
  This is equivalent to the existence of a distinguished splitting \(s\) stabilizing a pinning \(\E = (\mathcal{B}, \mathcal{T}, (X_{\alpha})_{\alpha \in \Delta^{\vee}})\) such that we have \(x = t s(\sigma)\) for some \(t \in \mathcal{T}\) and \(\sigma \in \Gamma\).
  The ``if'' direction is easy using Jordan decompositions \cite[\S I.4]{Borel_lag}, the ``only if'' direction is well-known (see \cite[\S 7]{Steinberg_endo} or \cite[Lemma 6.5]{Borel_autLfunc}).
\end{remark}

Weil-Deligne Langlands parameters correspond bijectively to morphisms \(\W_F' \to {}^L G\) which are algebraic on \(\mathbb{G}_a\) (given by \(x \mapsto \exp(xN)\)), continuous on the factor \(\W_F\) and compatible with the projection to \(\Gamma\) (given by \(\rho\)).
One of the motivations for using the first version \(\W_F'\) of the Weil-Deligne group, rather than the other two, is the \(\ell\)-adic monodromy theorem \cite[Theorem 4.2.1]{Tate_ntbackground}.
This roughly says that for a prime \(\ell\) not equal to the residual characteristic of \(F\) and for \(C = \Qellbar\)\footnote{One could work with a finite extension of \(\Qell\) instead.}, any continuous morphism \(\W_F \to {}^L G\) \emph{for the natural topology on \(\Ghat\)} compatible with \({}^L G \to \Gamma\) is given by a pair \((\rho, N)\) satisfying the first two conditions above.
Continuous \(\ell\)-adic Galois representations occur naturally in algebraic geometry (Tate modules of elliptic curves over \(F\), or more generally in the étale cohomology of varieties defined over \(F\)).
Another reason for preferring \(\W_F'\) is that this version requires fewer ``choices of a square root of \(q\)'' in the local Langlands correspondence, and is more obviously compatible with parabolic induction (property \eqref{it:crude_LL_infchar_nonA} in Conjecture \ref{conj:crude_LLC} below).

For the second version \(\WD_F\), over any algebraically closed field \(C\) of characteristic zero, Langlands parameters are defined as morphisms \(\phi: \W_F \times \SL_2(C) \to {}^L G\) which are compatible with \({}^L G \to \Gamma\), continuous and semi-simple on the first factor and algebraic on the second factor.

For the third version, we need to assume \(C = \C\) and we consider continuous (for the natural topology on \(\Ghat\)) semi-simple morphisms \(\phi: \W_F \times \mathrm{SU}(2) \to {}^L G\) which are compatible with \({}^L G \to \Gamma\).
By restriction via \(\mathrm{SU}(2) \subset \SL_2(\C)\) we obtain exactly the same morphisms as in the second version, essentially because \(\SL_2(\C)\) is the complexification of the compact Lie group \(\mathrm{SU}(2)\).

Recall that we have already chosen a square root of \(q\) in \(C\) in order to normalize parabolic induction.
We have a natural map from Langlands parameters to Weil-Deligne Langlands parameters:
\begin{equation} \label{eq:Langlands_to_WD}
  \phi \mapsto \left( \phi \circ \iota_W, \mathrm{d} \phi|_{\SL_2} \begin{pmatrix} 0 & 1 \\ 0 & 0\end{pmatrix} \right)
\end{equation}
where \(\iota_W(w) = (w, \operatorname{diag}(||w||^{1/2}, ||w||^{-1/2}))\).
By a refinement of the Jacobson-Morozov theorem (see \cite[Lemma 2.1]{GrossReeder_disc}) this induces a bijection between sets of \(\Ghat\)-conjugacy classes of parameters.

If \(F\) is Archimedean we assume \(C=\C\) and define Langlands parameters as semi-simple continuous morphisms \(\phi: \W_F \to {}^L G\) which are compatible with \({}^L G \to \Gamma\).

We will denote by \(\Phi(G)\) the set of \(\Ghat\)-conjugacy classes of Langlands parameters taking values in \({}^L G\).
As explained above all versions of the Weil-Deligne group give equivalent sets of \(\Ghat\)-conjugacy classes.
From now on we will only consider Langlands parameters, i.e.\ morphisms to \({}^L G\) using the second version \(\WD_F\).

\subsection{Reductions} \label{sec:parameters_reductions}

Assume \(C=\C\).
We briefly recall from \cite{SilbergerZink} the Langlands classification for parameters.

For an algebraic group \(H\) over \(\C\) recall from \cite[\S 5.1]{SilbergerZink} that any semi-simple \(x \in H(\C)\) admits a polar decomposition \(x = x_0 x_h\) characterized by the following properties:
\begin{itemize}
\item \(x_0 x_h = x_h x_0\),
\item \(x_0\) and \(x_h\) are semi-simple,
\item in any representation \(\rho\) of \(H\) the eigenvalues of \(\rho(x_0)\) (resp.\ \(\rho(x_h)\)) belong to the unit circle (resp.\ to \(\R_{>0}\)).
\end{itemize}
Uniqueness is clear using a faithful representation.
For a morphism \(f: H_1 \to H_2\) between linear algebraic groups over \(\C\) and a semi-simple \(x \in H_1(\C)\) it is also clear that if \(x = x_0 x_h\) is a polar decomposition then \(f(x) = f(x_0) f(x_h)\) is also a polar decomposition.
Since any semi-simple \(x \in H(\C)\) is contained in a diagonalizable subgroup of \(H\) \cite[Proposition 8.4]{Borel_lag} this reduces the proof of existence of polar decompositions to the case where \(H\) is diagonalizable, which is easy.
One could also show existence and uniqueness using the Tannakian formalism.

Now let \(\mathrm{cl}(\phi) \in \Phi(G)\).
Applying the polar decomposition\footnote{As in Remark \ref{rem:ss_via_linear_quot} we may work in the reductive group \({}^L G / s(\Gamma_E)\), for some large enough finite Galois extension \(E/F\), instead of \({}^L G\).} to \(\phi(w)\) for any \(w \in W_F\) with positive valuation, we obtain a canonical tuple \((\mathcal{P}, \mathcal{M}, \phi_0, \chi)\) giving a decomposition
\begin{equation} \label{eq:param_decomp}
  \phi = \phi_0 \chi
\end{equation}
subject to the following conditions.
\begin{itemize}
\item The subgroup \(\mathcal{P}\) of \({}^L G\) is a parabolic subgroup and \(\mathcal{M}\) is a Levi subgroup of \(\mathcal{P}\).
  We denote by \(\mathcal{N}\) the unipotent radical of \(\mathcal{P}\).
\item \(\phi_0\) is a Langlands parameter taking values in \(\mathcal{M}\) and bounded on \(\W_F\).
\item \(\chi \in Z^1(\W_F, X_*(Z(\mathcal{M})^0) \otimes_{\Z} \R_{>0})\) where \(X_*(Z(\mathcal{M})^0) \otimes_{\Z} \R_{>0}\) is seen as a subgroup of the torus \(X_*(Z(\mathcal{M})^0) \otimes_{\Z} \C^\times = Z(\mathcal{M})^0\).
\item The eigenvalues of \(\chi(\Frob)\) if \(F\) is non-Archimedean (resp.\ \(\chi(x)\) for any real \(x>1\) if \(F\) is Archimedean) on \(\Lie \mathcal{N}\) are all greater than \(1\).
\end{itemize}
This corresponds to the Langlands classification (Theorem \ref{thm:Langlands_classification}) using Lemma \ref{lem:Langlands_class_abs}.
This reduction explains why we are mainly interested in bounded parameters \(\phi\).
We will also call such parameters tempered.
A nice property of tempered (or more generally, essentially tempered, allowing twists by cocycles \(\W_F \to Z(\Ghat)\)) parameters is that the restriction of the map \eqref{eq:Langlands_to_WD}, associating Weil-Deligne parameters to Langlands parameters, to the set of tempered parameters is injective: see \cite[Corollary 3.16]{MeliImaiYoucis}.
This may also be proved by observing that for a tempered Langlands parameter \(\phi\) with associated Weil-Deligne Langlands parameter \((\rho,N)\) we can recover \(\phi|_{\W_F}\) and the semi-simple element in the \(\mathfrak{sl}_2\)-triple corresponding to \(\phi|_{\SL_2}\) from the polar decomposition \eqref{eq:param_decomp} of \(\rho\) using \cite[Corollary 3.5]{Kostant_sl2}.
In particular the centralizers in both versions coincide in the tempered case.
This is not true in general, see \cite[Example 3.8]{MeliImaiYoucis}.

The following proposition does not assume \(C=\C\).
\begin{proposition}[{\cite[Proposition 3.6]{Borel_autLfunc}}]
  Let \(\phi: \WD_F \to {}^L G\) be a Langlands parameter.
  The Levi subgroups of \({}^L G\) which are minimal among those containing \(\phi(\WD_F)\) are all conjugated under the centralizer of \(\phi\) in \(\Ghat\).
\end{proposition}

This proposition may be seen as a generalization of the isotypical decomposition of a semi-simple linear group representation.
A Langlands parameter \(\phi\) is called \emph{essentially discrete} if this Levi subgroup is \({}^L G\), i.e.\ if \(\phi\) is ``\({}^L G\)-irreducible''.
This condition is equivalent to \(\Cent(\phi, \Ghat)/Z(\Ghat)^\Gamma\) being finite.
A Langlands parameter \(\phi\) is called \emph{\(G\)-relevant} if this Levi subgroup is \(G\)-relevant (see Section \ref{sec:parabolic}).

\subsection{Weil restriction} \label{sec:Weil_res_param}

Let \(E\) be  a finite subextension \(E\) of \(\ol{F}/F\) and let \(\Gamma_E=\Gal(\ol{F}/E)\) be the corresponding open subgroup of \(\Gamma\).
Let \(G_0\) be a connected reductive group \(G_0\) over \(E\).
Let \(G=\Res_{E/F} G_0\) be the Weil restriction, a connected reductive group over \(F\) such that the topological groups \(G(F)\) and \(G_0(E)\) are isomorphic.
Recall from \cite[\S 5]{Borel_autLfunc} that we may identify \(\Ghat\) endowed with its action of \(\Gamma\) with the induction from \(\Gamma_E\) to \(\Gamma\) of \(\widehat{G_0}\).
By Shapiro's lemma we have a bijection \(\Phi(G) \simeq \Phi(G_0)\).

\section{The local Langlands conjecture}

\subsection{Crude local Langlands correspondence}

Denote by \(\Pi(G)\) be the set of isomorphism classes of irreducible admissible representations of \(G(F)\) over \(\C\) (in the Archimedean case, \((\mathfrak{g},K)\)-modules).

\begin{conjecture} \label{conj:crude_LLC}
  There should exist maps \(\LL : \Pi(G) \to \Phi(G)\) for all connected reductive groups \(G\) over \(F\), satisfying the following properties.
  Denote \(\Pi_\phi(G) = \LL^{-1}(\mathrm{cl}(\phi))\).
  \begin{enumerate}
  \item \label{it:crude_LL_torus}
    If \(G\) is a torus then \(\LL\) should be the bijection that Langlands deduced from class field theory \cite[\S 9]{Borel_autLfunc}.
  \item \label{it:crude_LL_img}
    For any \(G\) all fibers of \(\LL\) should be finite and the image of \(\LL\) should contain all essentially discrete parameters.
  \item \label{it:crude_LL_prod}
    If \(G = G_1 \times G_2\) then, using the identification of \({}^L G\) with \({}^L G_1 \times_\Gamma {}^L G_2\), for any irreducible admissible representation \(\pi \simeq \pi_1 \otimes \pi_2\) of \(G(F)\) we should have \(\LL(\pi) = (\LL(\pi_1), \LL(\pi_2))\).
  \item \label{it:crude_LL_isog}
    If \(\theta: G \rightarrow H\) is a central isogeny with dual \(\widehat{\theta}: {}^L H \rightarrow {}^L G\) then for \(\pi \in \Pi(H)\) and any constituent \(\pi'\) of the restriction \(\pi|_{G(F)}\) (which is semi-simple of finite length) we should have \(\LL(\pi') = \widehat{\theta} \circ \LL(\pi)\).
    (If this holds for all central isogenies then using \eqref{it:crude_LL_prod} one can deduce that it holds more generally for morphisms \(G \to H\) with central kernel, normal image and abelian cokernel.)
  \item \label{it:crude_LL_Weil_res}
    In the setup of Section \ref{sec:Weil_res_param} (Weil restriction) we should have a commutative diagram
    \[
      \begin{tikzcd}
        \Pi(G) \arrow[d, "{\sim}"] \arrow[r, "{\LL}"] & \Phi(G) \arrow[d, "{\sim}"] \\
        \Pi(G_0) \arrow[r, "{\LL}"] & \Phi(G_0)
      \end{tikzcd}
    \]
    where the left vertical map is induced by the isomorphism \(G(F) \simeq G_0(E)\) and the right vertical map is Shapiro's lemma.
  \item \label{it:crude_LL_L2}
    For an irreducible smooth representation \(\pi\) of \(G(F)\) we should have that \(\pi\) is essentially square-integrable if and only if \(\LL(\pi)\) is essentially discrete.
  \item \label{it:crude_LL_temp}
    Let \(M\) be a Levi subgroup of \(G\).
    Recall from Lemma \ref{lem:conj_Lemb_Levi} that we have an embedding \(\iota_M: {}^L M \hookrightarrow {}^L G\), well-defined up to \(\Ghat\)-conjugacy.
    If \(\sigma\) is an irreducible smooth representation of \(M(F)\) which is essentially square-integrable and has unitary central character then for any constituent \(\pi\) of \(i_P^G \sigma\) we should have \(\LL(\pi) = \iota_M \circ \LL(\sigma)\).
  \item \label{it:crude_LL_Lclass}
    In the situation of Theorem \ref{thm:Langlands_classification} we should have
    \[ \LL(J(P,\sigma,\nu)) = \iota_P \circ \LL(\sigma \otimes \nu). \]
  \item \label{it:crude_LL_infchar}
    For Archimedean \(F\) the maps \(\LL\) should be compatible with infinitesimal characters in the following sense.
    Assume \(F \simeq \R\) (we may reduce to this case if \(F \simeq \C\) by \eqref{it:crude_LL_Weil_res} above) and choose an isomorphism \(\ol{F} \simeq \C\), so that we use the same field \(\C\) for the coefficients, to construct the Weil group of \(F\) and to define \(\mathfrak{g} = \C \otimes_F \Lie G\).
    Let \(\pi\) be an irreducible \((\mathfrak{g},K)\)-module.
    The restriction of \(\LL(\pi)\) to \(\C^\times\) is conjugated to a morphism of the form \(z \mapsto z^\lambda \ol{z}^\mu\) where \(\lambda, \mu \in X_*(\mathcal{T}) \otimes_\Z \C\) satisfy \(\lambda-\mu \in X_*(\mathcal{T})\) and \(z^\lambda \ol{z}^\mu\) is a suggestive notation for \((z\ol{z})^{(\lambda+\mu)/2} (z/|z|)^{\lambda-\mu}\).
    The infinitesimal character of \(\pi\) should be identified to \(\lambda\) by the Harish-Chandra isomorphism.
  \item \label{it:crude_LL_infchar_nonA}
    Assume that \(F\) is non-Archimedean.
    If \(P=MN\) is a parabolic subgroup of \(G\) and \(\sigma\) is an irreducible smooth representation of \(M(F)\), then for any irreducible subquotient \(\pi\) of \(i_P^G \sigma\) we should have \(\LL(\pi) \circ \iota_W = \iota_M \circ \LL(\sigma) \circ \iota_W\).
    Equivalently, the same but just for supercuspidal \(\sigma\).
\end{enumerate}
\end{conjecture}

The list of properties in Conjecture \ref{conj:crude_LLC} is not exhaustive, in particular we did not discuss the relation with \(L\)-functions, \(\epsilon\)-factors and \(\gamma\)-factors.
This list is certainly not enough to characterize the map \(\LL\), and if we omit \eqref{it:crude_LL_infchar_nonA} it may be possible to prove the conjecture in the non-Archimedean case using rather formal arguments (essentially by comparing, for simple and simply connected \(G\), the cardinality of the set of essentially square-integrable irreducible representations of \(G(F)\) with that of the set of essentially discrete parameters), but this would not give a lot of insight.
So it is desirable to have constructions and characterizations of the map \(\LL\) rather than just a proof of Conjecture \ref{conj:crude_LLC}.
We refer the interested reader to \cite{Harris_LL} for a survey of the possible characterizations.

We also warn the reader that there are actually two versions of the conjecture, corresponding to the two possible normalizations of the Artin reciprocity map in local class field theory.
According to \cite[\S 4]{KottwitzShelstad_corr} these should be related by a certain automorphism of \({}^L G\), which according to \cite{Kaletha_genericity}, \cite{AdamsVogan_contra} and \cite{Prasad_contragredient} is itself related to taking contragredient representations.
Thus another way to state the relation between the two normalizations is to say that we should obtain one from the other by composing with the involution \(\pi \mapsto \tilde{\pi}\).

Cases for which the conjecture is known (with a ``natural'' construction or characterization) include the Archimedean case \cite{Langlands_class}, general linear groups over non-Archimedean fields \cite{LaumonRapoportStuhler} \cite{Henniart_LLC} \cite{HarrisTaylor} \cite{Scholze_LLC}, \(\mathrm{GSp}_4\) over finite extensions of \(\Q_p\) \cite{GanTakeda_GSp4}, inner forms of special linear groups over finite extensions of \(\Q_p\) \cite{HiragaSaito}, and quasi-split classical groups over finite extensions of \(\Q_p\) \cite{Arthur_book} \cite{Mok_unitary}.
More cases will be discussed later.

The rest of this section is devoted to remarks on the properties in the conjecture.

\subsubsection{Compatibility with the case of tori}

The functoriality assumptions \eqref{it:crude_LL_prod} and \eqref{it:crude_LL_isog} imply the following compatibilities with the case of tori.
\begin{itemize}
\item The map \(\LL\) should be compatible with central characters in the following sense.
  Let \(Z\) be the maximal central torus in \(G\) so that we have a surjective morphism \({}^L G \rightarrow {}^L Z\).
  Then all elements of \(\Pi_{\phi}(G)\) should have central character determined by composing \(\phi\) with this surjection and applying \(\LL^{-1}\).
\item Langlands defined (see \cite[\S 10.2]{Borel_autLfunc}) a morphism
  \[ H^1_\cont(\W_F, Z(\Ghat)) \to \Hom_{\cont}(G(F), \C^{\times}). \]
  For a continuous \(1\)-cocycle \(c: \W_F \to Z(\Ghat)\) with corresponding character \(\chi: G(F) \to \C^\times\) we should have \(\LL(\pi \otimes \chi) = c \LL(\pi)\).
\end{itemize}

\subsubsection{Reduction to the discrete case} \label{sec:red_crude_LLC}

Using Proposition \ref{pro:class_temp_L2}, the Langlands classification (Theorem \ref{thm:Langlands_classification}, Lemma \ref{lem:Langlands_class_abs}) and the ``Langlands classification for parameters'' (see Section \ref{sec:parameters_reductions}), properties \eqref{it:crude_LL_temp} and \eqref{it:crude_LL_Lclass} imply that \(\pi\) is tempered if and only if \(\LL(\pi)\) is tempered.
In fact we see that these parallel results for smooth representations of reductive groups and Langlands parameters reduce the construction of \(\LL\) to the essentially square-integrable case, and with property \eqref{it:crude_LL_img} we see that the image of \(\LL\) should be the set of \(G\)-relevant Langlands parameters.

\subsubsection{The unramified case}

From properties \eqref{it:crude_LL_torus}, \eqref{it:crude_LL_temp} and \eqref{it:crude_LL_Lclass} it follows that if \(G\) is unramified and \(K\) is a hyperspecial compact open subgroup of \(G(F)\) then on \(K\)-unramified irreducible representations of \(G(F)\) (i.e.\ representations having non-zero \(K\)-invariants) the map \(\LL\) is given by the Satake isomorphism.
More precisely in this case the minimal Levi subgroup \(M_0\) is an unramified torus and unramified representations of \(G(F)\) are parametrized by orbits under the rational Weyl group of continuous characters \(\chi: M_0(F) \to \C^\times\).
The unramified representation \(\pi\) corresponding to the orbit of \(\chi\) is the unique unramified constituent of \(i_B^G \chi\), for any Borel subgroup \(B\) of \(G\) containing \(M_0\).
We have \(\LL(\pi) = \iota_{M_0} \circ \LL(\chi)\), in other words \(\LL(\pi)\) is the parameter associated to \(\pi\) by the Satake isomorphism.
In the tempered case, that is when \(\chi\) is unitary, this follows immediately from property \eqref{it:crude_LL_temp}.
The general case is more subtle, and can be deduced from the Gindikin-Karpelevich formula \cite[Theorem 3.1]{Casselman_unr} (see \cite[p. 219]{CasselmanShalika} for the values of the constants in the case of an unramified group)\footnote{To be honest the arguments in \cite{Casselman_unr} assume that \(\chi\) is regular but similar arguments work using only partial regularity.}.

\subsubsection{The semisimplified correspondence and algebraicity}

For non-Archimedean \(F\) property \eqref{it:crude_LL_infchar_nonA} says that the map \(\LLss: \pi \mapsto \LL(\pi) \circ \iota_W\) is compatible with the notion of supercuspidal support (Theorem \ref{thm:cusp_supp}).
This suggests the following conjecture.

\begin{conjecture} \label{conj:crude_LLC_ss}
  Assume that \(F\) is non-Archimedean.
  Let \(C\) be any algebraically closed field of characteristic zero and choose a square root \(\sqrt{q} \in C\).
  There should exist for each connected reductive group \(G\) over \(F\) a map \(\LLss\) from the set of isomorphism classes of smooth irreducible representations of \(G(F)\) over \(C\) to the set of \(\Ghat\)-conjugacy classes of continous semi-simple morphisms \(\W_F \to {}^L G\) which are compatible with \({}^L G \to \Gamma\), satisfying the obvious analogue of \eqref{it:crude_LL_torus}, \eqref{it:crude_LL_prod}, \eqref{it:crude_LL_isog} in Conjecture \ref{conj:crude_LLC}, as well as the following analogue of property \eqref{it:crude_LL_infchar_nonA} in Conjecture \ref{conj:crude_LLC}.

  If \(P=MN\) is a parabolic subgroup of \(G\) and \(\sigma\) is an irreducible smooth representation of \(M(F)\) then for any irreducible subquotient \(\pi\) of \(i_P^G \sigma\) we should have \(\LLss(\pi)=\iota_M \circ \LLss(\sigma)\).

  Moreover these maps \(\LLss\) should be functorial in \((C,\sqrt{q})\)\footnote{One could certainly avoid the choice of a square root of \(q\) by modifying Langlands dual groups.
  We do not attempt to explain this here, see \cite[\S 5.3]{BuzzardGee_conj} and \cite[\S 2.2]{Imai_LLladic}.}.
\end{conjecture}

For \(C=\Qellbar\), where \(\ell\) does not equal the residue characteristic of \(F\), Genestier-Lafforgue \cite{GenestierLafforgue} (in positive characteristic) and Fargues-Scholze \cite{FarguesScholze} have constructed maps \(\LLss\) satisfying all properties in Conjecture \ref{conj:crude_LLC_ss} except for functoriality with respect to the coefficient field, which seems to remain open.

Conjecture \ref{conj:crude_LLC} implies the case \(C=\C\) of Conjecture \ref{conj:crude_LLC_ss}, again excluding functoriality in \((C,\sqrt{q})\).
Assuming Conjecture \ref{conj:crude_LLC} one can also show that the map \(\LL \circ \iota_W\) determines the map \(\LL\), by considering first the case of tempered representations and using the decomposition \eqref{eq:param_decomp} and the fact that an \(\mathfrak{sl}_2\) triple (here, in the connected centralizer of \(\phi_0\) in \(\Ghat\)) is determined by its semi-simple element up to conjugation.
In the case of general linear groups the construction of the map \(\LL\) was reduced to the supercuspidal case by Zelevinsky \cite{Zelevinsky_ind2}.
In general however Conjecture \ref{conj:crude_LLC_ss} does not immediately imply Conjecture \ref{conj:crude_LLC}.
What is missing is the fact that for any essentially square-integrable irreducible smooth representation \(\pi\) of \(G(F)\), the semi-simplified parameter \(\LLss(\pi)\) comes from an \emph{essentially discrete} Langlands parameter (which as above is automatically unique up to conjugation by the centralizer of \(\LLss(\pi)\) in \(\Ghat\)).
In this direction Gan, Harris, Sawin and Beuzart-Plessis have recently shown \cite[Theorem 1.2]{GanHarrisSawin} that the Genestier-Lafforgue semi-simplified Langlands parameter of an essentially square-integrable representation, say with central character having finite order, comes from a (again, unique) \emph{tempered} Langlands parameter.

Note that properties \eqref{it:crude_LL_L2}, \eqref{it:crude_LL_temp} and \eqref{it:crude_LL_Lclass} in Conjecture \ref{conj:crude_LLC} make essential use of the topology on the coefficient field \(\C\).
The notion of essentially discrete Langlands parameter is purely algebraic (it does not rely on the topology of the coefficient field) so there ought to be a purely algebraic characterization of essentially square-integrable representations.
We check the validity of this intuition in the following proposition.

\begin{proposition}
  Let \(F\) be a non-Archimedean local field.
  Assume Conjecture \ref{conj:crude_LLC}.
  Let \(\pi\) be an irreducible smooth representation of \(G(F)\).
  Assume that its central character \(\omega_\pi\) has finite order (we may reduce to this case by twisting by a continuous character \(G(F) \to \C^\times\)).
  Then \(\pi\) is essentially square-integrable if and only if for any parabolic subgroup \(P=MN\) of \(G\) and for any character \(\chi\) of \(A_M(F)\) occurring in \(r^G_P \pi\) there exists an integer \(N \geq 1\) such that \(\chi^N\) is equal to \(\prod_\alpha ||\alpha||^{n_\alpha}\) for some integers \(n_\alpha>0\), where the product ranges over the simple roots of \(A_M\) in \(\Lie N\).
\end{proposition}
\begin{proof}
  The ``if'' implication is obvious using Proposition \ref{pro:char_L2_rep}, so we are left to prove the ``only if'' implication.
  Let \(P=MN\) and \(\chi\) be as in the proposition.
  There exists an irreducible quotient \(\sigma\) of the representation \(r^G_P \pi\) of \(M(F)\) whose central character \(\omega_\sigma\) satisfies \(\omega_\sigma|_{A_M(F)} =\chi\), and by Frobenius reciprocity we have an embedding \(\pi \hookrightarrow{} i_P^G \sigma\).
  Let \(\phi_M = \LL(\sigma)\) and \(\phi = \LL(\pi)\).
  By property \eqref{it:crude_LL_infchar_nonA} in Conjecture \ref{conj:crude_LLC} we have \(\iota_M \circ \phi_M \circ \iota_W = \phi \circ \iota_W\) (up to conjugacy by \(\Ghat\)).
  By properties \eqref{it:crude_LL_torus}, \eqref{it:crude_LL_prod} and \eqref{it:crude_LL_isog} in Conjecture \ref{conj:crude_LLC} the character \(\chi\) corresponds by local class field theory to the composition
  \[ \WD_F \xrightarrow{\phi_M} {}^L M \rightarrow{} {}^L A_M \]
  which may be seen as a continuous morphism \(\W_F \to \widehat{A_M}\) because the torus \(A_M\) is split.
  Note that pre-composing with \(\iota_W\) does not change this morphism.
  Because we already know Proposition \ref{pro:char_L2_rep} it is enough to prove that some integral power of \(\chi\) is of the form \(\prod_\alpha ||\alpha||^{n_\alpha}\) for some integers \(n_\alpha\).
  By assumption some integral power of \(\chi\) is trivial on \(A_G(F)\).
  The simple roots of \(A_M\) on \(N\) give an isogeny \(A_M/A_G \to \mathbb{G}_m^n\) for some integer \(n \geq 0\), so it is enough to prove that some integral power of \(\chi\) takes values in \(q^\Z\).

  There exists an open subgroup \(U\) of the inertia subgroup \(I_F \subset \W_F\) such that the action of \(U\) on \(\Ghat\) is trivial and \(\phi\) is trivial on \(U\).
  Up to taking a smaller subgroup we may assume that \(U\) is normalized by \(\W_F\) (this follows from a simple Galois-theoretic argument).
  There exists an integer \(N \geq 1\) such that for any \(w \in \W_F\) the action of \(w^N\) by conjugation on \(\W_F/U\) is trivial (this can be proved in two steps, first for \(I_F/U\) and then for a lift of a generator of \(\W_F/I_F\)).
  It will be convenient to see the Langlands dual group \({}^L G\) as \(\Ghat \rtimes \Gal(E/F)\) for some finite Galois extension \(E/F\) (see Remark \ref{rem:ss_via_linear_quot}).
  Thus for any \(w \in \W_F\) we have that \(\phi(w)^N\) centralizes \(\phi(\WD_F)\).
  Up to replacing \(N\) by \(|\Gal(E/F)| \times |\Cent(\phi, \Ghat)/Z(\Ghat)^\Gamma|\) we obtain \(\phi(w)^N \in Z(\Ghat)^\Gamma\) for all \(w \in \W_F\).
  The natural map \(Z(\Ghat)^\Gamma \to \widehat{A_G}\) has finite kernel.
  Because we have assumed that \(\omega_\pi\) has finite order, up to replacing \(N\) by a non-zero multiple we even have \(\phi(w)^N=1\) for all \(w \in \W_F\).
  Up to replacing \(N\) by \(2N\), this implies that for any finite-dimensional algebraic representation \(r: {}^L G \to \GL(V)\) and for any \(w \in \W_F\), any eigenvalue \(\lambda \in \C^\times\) of \(r(\phi(\iota_W(w)))\) satisfies \(\lambda^N \in q^\Z\).
  Taking for \(r\) a closed embedding \({}^L G \hookrightarrow{} \SL(V)\), we obtain the claim because any irreducible representation of \({}^L M\), and in particular any character of \(\widehat{A_M}\), occurs in the restriction of \(V^{\otimes a}\) for some integer \(a \geq 0\).
\end{proof}

\subsubsection{Cuspidality and parameters}

Assume that \(F\) is non-Archimedean.
Property \eqref{it:crude_LL_infchar_nonA} implies that for any irreducible smooth representation \(\pi\) of \(G(F)\), if \(\LL(\pi)\) is essentially discrete and trivial on \(\SL_2\) then \(\pi\) is supercuspidal.
Contrary to the case of \(\GL_n\), in general the converse is not true, i.e.\ there exists supercuspidal representations \(\pi\) whose Langlands parameter \(\LL(\pi)\) is not trivial on \(\SL_2\).
A related matter is that the classification of essentially square-integrable representations in terms of supercuspidal representations (of Levi subgroups) is much more complicated in general than in the case of \(\GL_n\).
See \cite{MoeglinTadic} and \cite{Xu_cusp_supp} for the case of classical groups.
As we will briefly review in Section \ref{sec:known_cases}, Kaletha's construction in \cite{Kaletha_scuspL} of supercuspidal \(L\)-packets (under some assumptions on \(G\)) isolates the supercuspidal representations which correspond (or should correspond, depending on which definition of a local Langlands correspondence one chooses) to essentially discrete parameters trivial on \(\SL_2\).

\subsection{Refined local Langlands for quasi-split groups} \label{sec:refined_LLC_qs}

In some applications having just the map \(\LL\) is too crude, e.g.\ to formulate the global multiplicity formula for the automorphic spectrum of a connected reductive group over a global field, and so we would like to understand the fibers \(\Pi_\phi(G)\).

In this section we assume \(C=\C\) and that \(G\) is quasi-split.
For a Langlands parameter \(\phi: \WD_F \to {}^L G\) denote \(S_\phi=\Cent(\phi,\Ghat)\) (a reductive subgroup of \(\Ghat\)), and define \(\Sbar_\phi=S_\phi/Z(\Ghat)^\Gamma\).
Recall that a parameter \(\phi\) is essentially discrete if and only if \(\Sbar_\phi\) is finite.
It can happen that \(\pi_0(\Sbar_\phi)\) is non-abelian (even in the principal series case, that is if \(\phi\) factors through \(\iota_T : {}^L T \hookrightarrow {}^L G\) where \(T\) is part of a Borel pair \((B,T)\) defined over \(F\)!).
For \(F = \R\) however, it is always abelian, in fact there is a maximal torus \(\mathcal{T}\) of \(\Ghat\) such that \(S_\phi \cap \mathcal{T}\) meets every connected component of \(S_\phi\).
For a finite group \(A\) denote by \(\Irr(A)\) the set of isomorphism classes of irreducible representations of \(A\) over \(\C\).

\begin{conjecture} \label{conj:refined_LLC_qs}
  For each Langlands parameter \(\phi\) there should exist an embedding \(\Pi_\phi(G) \to \Irr(\pi_0(\Sbar_\phi))\).
  For non-Archimedean \(F\) this should be a bijection.
\end{conjecture}

Of course we do not simply seek the existence of embeddings \(\Pi_\phi(G) \to \Irr(\pi_0(\Sbar_\phi))\), we recall how to characterize them in Conjecture \ref{conj:endo_char_rel_qs} below.

Langlands's classification again reduces the construction of embeddings \(\Pi_\phi(G) \to \Irr(\pi_0(\Sbar_\phi))\) to the tempered case.
So we assume from now on that \(\phi\) is tempered.
These embeddings are not canonical in general: they depend on the choice of a Whittaker datum (up to conjugation by \(G(F)\)).

We briefly recall the notions of Whittaker datum and generic representation for a quasi-split connected reductive group \(G\).
Choose a Borel subgroup \(B\) with unipotent radical \(U\).
For a Galois orbit \(\mathcal{O}\) on the set of simple roots, the group \(U_\mathcal{O} = \left( \prod_{\alpha \in \mathcal{O}} U_{\alpha}(\ol{F}) \right)^{\Gal_F}\) is isomorphic to the additive group of a finite separable extension \(F_{\mathcal{O}}\) of \(F\).
We have a natural surjective morphism from \(U(F)\) to \(\prod_\mathcal{O} U_\mathcal{O}\).
Choosing a non-trivial morphism \(U_\mathcal{O} \to \C^{\times}\) for each orbit \(\mathcal{O}\) yields a morphism \(\theta: U(F) \to \C^\times\), called a generic character.
A Whittaker datum \(\mathfrak{w}\) for \(G\) is such a pair \((U,\theta)\).
The adjoint group \(G_\ad(F)\) acts transitively on the set of such pairs.
If \(F\) has characteristic zero there are only finitely many \(G(F)\)-conjugacy classes of Whittaker data.
If \(F\) is non-Archimedean an irreducible smooth representation \((\pi,V)\) of \(G(F)\) is called \(\mathfrak{w}\)-generic if there is a non-zero linear functional \(\lambda: V \rightarrow \C\) such that \(\lambda( \pi(u) v) = \theta(u) \lambda(v)\) for all \(u \in U(F)\) and \(v \in V\).
For Archimedean \(F\) the notion is more subtle because it requires a topology on the representation.


\begin{conjecture}[Shahidi] \label{conj:shahidi}
  There should be a unique \(\mathfrak{w}\)-generic representation in each \(\Pi_\phi(G)\).
  The conjectural embedding \(\iota_\mathfrak{w} : \Pi_\phi(G) \to \Irr(\pi_0(\ol{S}_\phi))\) (which depends on \(\mathfrak{w}\)) should map this \(\mathfrak{w}\)-generic representation to the trivial representation of \(\ol{S}_\phi\).
\end{conjecture}

In order to characterize the embeddings \(\iota_\mathfrak{w}\) we have to introduce endoscopic data.
Let \(s \in S_\phi\) be a semi-simple element.
From the pair \((s,\phi)\) one can construct the following objects.
For \(\pi \in \Pi_\phi(G)\) denote \(\langle s, \pi \rangle_{\mathfrak{w}}=\tr(\iota_\mathfrak{w}(\pi))(s)\).
On the one hand we have
\[ \Theta^{\mathfrak{w}}_{\phi,s} = \sum_{\pi \in \Pi_\phi(G)} \langle s,\pi \rangle_\mathfrak{w} \Theta_\pi. \]
This is a virtual character on \(G(F)\).
In the case \(s=1\) we introduce the special notation
\[ S\Theta_\phi = \Theta_{\phi,1}^{\mathfrak{w}}. \]
The reason for not recording \(\mathfrak{w}\) in the notation in this case will be explained below.

On the other hand we consider the complex connected reductive subgroup \(\mathcal{H}^0 = \Cent(s,\Ghat)^0\) of \(\Ghat\).
It contains \(\phi(1 \times \SL_2)\) and is normalized by \(\phi(\W_F)\).
Thus \(\mathcal{H} = \mathcal{H}^0 \cdot \phi(\W_F)\) is a subgroup of \({}^L G\), which is an extension \(1 \to \mathcal{H}^0 \to \mathcal{H} \to \W_F \to 1\).
The resulting morphism \(\W_F \to \Out(\mathcal{H}^0)\) factors through the Galois group of a finite extension of \(F\).
By Proposition \ref{pro:inner_forms_brd} there exists a quasi-split connected reductive group \(H\) over \(F\) together with an inner class of isomorphisms \(\eta: \mathcal{H}^0 \simeq \Hhat\) such that the above morphism \(\W_F \to \Out(\mathcal{H}^0)\) and the morphism \(\W_F \to \Out(\Hhat)\) used to define \({}^L H = \Hhat \rtimes \W_F\) correspond to each other via \(\eta\), and for any two such groups \(H_1\) and \(H_2\) we have an isomorphism \(H_1 \simeq H_2\), well-defined up to \(H_{1,\ad}(F)\).
It may unfortunately happen that the two extensions \(\mathcal{H}\) and \({}^L H\) of \(\W_F\) are not isomorphic.
We shall ignore this difficulty, as its resolution is not terribly exciting (see \cite[Lemma 2.2.A]{KottwitzShelstad}).
So let's assume there exists an isomorphism of extensions \({}^L \eta: \mathcal{H} \to {}^L H\).
Then \(\mathfrak{e}=(H,s,{^L\eta})\) is called an extended endoscopic triple\footnote{This terminology relates to the notion of ``endoscopic triple'' of \cite[\S 7]{Kottwitz_STFcusptemp}, ``extended'' meaning that we have chosen an extension \({}^L \eta\) of \(\eta\).
Experts should note that we restrict to the case where \(s\) is invariant under \(\Gamma\), as we may in the local setting.}.
By construction we have a unique Langlands parameter \(\phi_H: \WD_F \to {}^L H\) such that we have \({}^L \eta \circ \phi_H = \phi\).
We have the virtual character \(S\Theta_{\phi_H}\) on \(H(F)\).


The two virtual characters \(\Theta^\mathfrak{w}_{\phi,s}\) and \(S\Theta_{\phi_H}\) are expected to be related by a certain kernel function.
This function, called the Langlands-Shelstad transfer factor, is itself non-conjectural and explicit.
It is a function
\[ \Delta[\mathfrak{w},\mathfrak{e}] : H(F)_{G\mathrm{-sr}} \times G(F)_\mathrm{sr} \to \C \]
whose construction depends on the Whittaker datum and the extended endoscopic triple.
We will not recall the definition of \(\Delta[\mathfrak{w}, \mathfrak{e}]\) (which is rather technical, see \cite{LanglandsShelstad} \cite{KottwitzShelstad} \cite{KottwitzShelstad_corr}), but let us recall what its support is (a correspondence between strongly regular semisimple conjugacy classes in \(G(F)\) and \(G\)-strongly regular semisimple stable conjugacy classes in \(H(F)\)), and recall a meaningful variance property.

\begin{definition}
  Recall that an element of \(G(\ol{F})\) is called strongly regular if its centralizer is a torus.
  Two semisimple strongly regular elements \(\delta, \delta'\) in \(G(F)\) are called stably conjugate if there exists \(g \in G(\ol{F})\) such that \(g \delta g^{-1} = \delta'\). 
\end{definition}

Using maximal tori and identifications of Weyl groups one can define \cite[Theorem 3.3.A]{KottwitzShelstad} a canonical map \(m\) from semisimple conjugacy classes in \(H(\ol{F})\) to semisimple conjugacy classes in \(G(\ol{F})\).
A conjugacy class in \(H(\ol{F})\) is called \(G\)-strongly regular if its image under \(m\) is strongly regular.
We denote by \(H(F)_{G\mathrm{-sr}}\) the set of \(G\)-strongly regular elements of \(H(F)\).
The map \(m\) enjoys the following properties.
\begin{enumerate}
\item The map \(m\) is \(\Gamma\)-equivariant.
\item If \(\gamma \in H(F)\) is semisimple \(G\)-strongly regular with associated \(H(\ol{F})\)-conjugacy class \([\gamma]_{\ol{F}}\) then \(m([\gamma]_{\ol{F}}) \cap G(F)\) is a non-empty\footnote{For non-emptiness the fact that \(G\) is quasi-split is essential.} finite union of \(G(F)\)-conjugacy classes.
  In this situation we say that (the stable conjugacy class of) \(\gamma\) and (the conjugacy class) of \(\delta \in m([\gamma]_{\ol{F}}) \cap G(F)\) \emph{match}.
  Given a strongly regular stable conjugacy class for \(G\), there are finitely many stable conjugacy classes for \(H\) in its preimage by \(m\).
\item For any matching pair \((\gamma, \delta) \in H(F)_{G\mathrm{-sr}} \times G(F)_\mathrm{sr}\), denoting \(T_H = \Cent(\gamma, H)\) and \(T = \Cent(\delta, g)\) (maximal tori of \(H\) and \(G\)), there is a canonical isomorphism \(T_H \simeq T\) identifying \(\gamma\) and \(\delta\).
\end{enumerate}
The fact that \(m\) is defined at the level of conjugacy classes over \(\ol{F}\) rather than \(F\) is one justification for introducing the notion of stable conjugacy.

Let \(\delta\) be a strongly regular element of \(G(F)\), and denote \(T = \Cent(\delta, G)\).
The set of \(G(F)\)-conjugacy classes \([\delta']\) which are stably conjugate to \(\delta\) is parametrized by \(\ker \left( H^1(F, T) \rightarrow H^1(F, G) \right)\), by mapping \(\delta'\) to \(\inv(\delta, \delta') := (\sigma \mapsto \sigma(g)^{-1} g)\) where as above \(g \delta g^{-1} = \delta'\).
Because of this description of stable conjugacy it is desirable to better understand these Galois cohomology sets.
Recall from \cite{Tate_cohtori} that the Tate-Nakayama isomorphism for tori over \(F\) identifies \(H^1(F,T)\) with
\begin{equation} \label{eq:TN_torus_prelim}
  \widehat{H}^{-1}(E/F, X_*(T)) = X_*(T)^{N_{E/F}=0} / I_{E/F} X_*(T)
\end{equation}
where \(E/F\) is any finite Galois subextension of \(\ol{F}/F\) splitting \(T\), \(N_{E/F}\) is the norm map, and for a \(\Z[\Gal(E/F)]\)-module \(Y\) we denote by \(I_{E/F} Y\) the submodule \(\sum_{\sigma \in \Gal(E/F)} (\sigma-1) Y\).
Note that the right-hand side of \eqref{eq:TN_torus_prelim} can also be described as the torsion subgroup of the coinvariants \(X_*(T)_\Gamma\).
Kottwitz interpreted this isomorphism in terms of Langlands dual groups and generalized it to connected reductive groups in \cite{Kottwitz_STFellsing}.
Recall that \(\That\) is a torus over \(\C\) endowed with an isomorphism \(X^*(\That) \simeq X_*(T)\).
Using the exactness of the functor mapping a finitely generated abelian group \(A\) to the diagonalizable group scheme \(Z\) with character group \(A\) (considered as a sheaf on the étale site of \(\C\), say) we see that \(X^*(\That)_\Gamma\) is identified with \(X^*(\That^\Gamma)\).
It follows that the Tate-Nakayama isomorphism may be written as
\begin{equation} \label{eq:TN_torus}
  \alpha_T : H^1(F,T) \simeq \Irr \left( \pi_0(\That^\Gamma) \right).
\end{equation}
It is formal to check that this identification is functorial in \(T\).
As for the Artin reciprocity map it would be just as natural to consider the same isomorphism composed with \(x \mapsto x^{-1}\).

\begin{theorem}[{\cite[Theorem 1.2]{Kottwitz_STFellsing} \cite[Theorem 2.1]{NguyenQuocThang_Galcoh}}] \label{thm:TN_redgps}
  There is a unique extension of the above family of isomorphisms to a family of maps of pointed sets
  \[ \alpha_G: H^1(F, G) \to \Irr \left( \pi_0(Z(\Ghat)^{\Gamma}) \right) \]
  for connected reductive \(G\), ``functorial'' in the following sense.
  For any morphism \(H \to G\) which is either the embedding of a maximal torus in a connected reductive group \(G\) or a central isogeny between connected reductive groups we have a commutative diagram
  \[ \begin{tikzcd}
      H^1(F,H) \arrow[r] \arrow[d, "{\alpha_H}"] & H^1(F,G) \arrow[d, "{\alpha_G}"] \\
      \Irr \left( \pi_0(Z(\Hhat)^{\Gamma}) \right) \arrow[r] & \Irr \left( \pi_0(Z(\Ghat)^{\Gamma}) \right)
    \end{tikzcd} \]
  where the bottom horizontal map is the one induced by the \(\Gamma\)-equivariant map \(Z(\Ghat) \to Z(\Hhat)\) recalled (in both cases) at the end of Section \ref{sec:Lgroups}.
  
  For two connected reductive groups \(G_1\) and \(G_2\) we have \(\alpha_{G_1 \times G_2} = \alpha_{G_1} \times \alpha_{G_2}\).
\end{theorem}

In \cite{Kottwitz_STFellsing} this is proved in the case where \(F\) has characteristic zero but the same proof works for all local fields, using Bruhat and Tits' generalization of Kneser's theorem \cite{BruhatTits_galcoh} (see \cite{NguyenQuocThang_Galcoh}).
Kneser's theorem is the special case where \(G\) is semi-simple and simply connected over a \(p\)-adic field, in which case we have \(Z(\Ghat)=1\) and so the theorem says that \(H^1(F,G)\) is trivial.
If \(F\) is non-Archimedean then each \(\alpha_G\) is a bijection, in particular \(H^1(F,G)\) has a commutative group structure.
In the Archimedean case the kernel and image of \(\alpha_G\) are described loc.\ cit.
We will also denote \(\alpha_G(c)(s) = \langle c, s \rangle\).

We resume the above notation: \((H, s, {}^L \eta)\) is an extended endoscopic triple, \((\gamma, \delta) \in H(F)_{G\mathrm{-sr}} \times G(F)_\mathrm{sr}\) is a matching pair, \(T_H = \Cent(\gamma, H)\) and \(T = \Cent(\delta, G)\) and we have a canonical isomorphism \(T_H \simeq T\).
By Theorem \ref{thm:TN_redgps} the kernel of \(H^1(F, T) \rightarrow H^1(F, G)\) is identified with the group of characters of \(\pi_0(\That^{\Gamma})\) which are trivial on \(Z(\Ghat)^{\Gamma}\).
The element \({}^L \eta(s) \in Z(\Hhat)^{\Gamma}\) defines an element \(s_{\gamma, \delta}\) of \(\That_H^{\Gamma} \simeq \That^{\Gamma}\).
We can finally state the variance property of transfer factors: we have
\begin{equation} \label{eq:transfer_fact_equiv} 
  \Delta[\mathfrak{w},\mathfrak{e}](\gamma, \delta') = \Delta[\mathfrak{w},\mathfrak{e}](\gamma, \delta) \langle \inv(\delta, \delta'), s_{\gamma, \delta} \rangle^{-1}.
\end{equation}
As for the Artin reciprocity map and the pairing \eqref{eq:TN_torus} there are several natural normalizations for the transfer factors \cite[\S 4]{KottwitzShelstad_corr}, and for half of these normalizations the exponent \(-1\) on the right-hand side should be removed.
The relation \eqref{eq:transfer_fact_equiv} is far from characterizing \(\Delta[\mathfrak{w},\mathfrak{e}]\) because it does not compare the values at unrelated matching pairs.

\begin{conjecture} \label{conj:endo_char_rel_qs}
  Let \(G\) be a quasi-split connected reductive group over \(F\).
  Let \(\phi: \WD_F \to {}^L G\) be a tempered Langlands parameter.
  \begin{enumerate}
  \item The map \(S\Theta_\phi: G_\rs(F) \to \C\) should be invariant under \emph{stable} conjugacy\footnote{For convenience we only defined stable conjugacy in the strongly regular case, so strictly speaking one should say that the restriction of \(S\Theta_\phi\) to the strongly regular locus should be stable.
    Note that the complement of the strongly regular locus still has measure zero.}.
  \item For any semi-simple \(s \in S_\phi\) and any strongly regular semisimple \(G(F)\)-conjugacy class \([\delta]\) we should have
    \begin{equation} \label{eq:endo_char_rel_qs}
      \Theta_{\phi,s}^{\mathfrak{w}}(\delta)=\sum_{\gamma \in H(F)/\mathrm{st}}		\Delta[\mathfrak{w},\mathfrak{e}](\gamma,\delta)S\Theta_{\phi_H}(\gamma)
    \end{equation}
    where \(((H, s, {}^L \eta), \phi_H)\) is an extended endoscopic triple and Langlands parameter \(\phi_H: \WD_F \to {}^L H\) corresponding to \((\phi,s)\), and \(\mathrm{st}\) denotes stable conjugacy.
  \end{enumerate}
\end{conjecture}

\begin{remark}
  \begin{enumerate}
  \item The equation \eqref{eq:endo_char_rel_qs} uniquely determines \(\iota_\mathfrak{w}\) provided it exists, due to the linear independence of characters.
    In particular, one can deduce how \(\iota_\mathfrak{w}\) should depend on \(\mathfrak{w}\).
    Namely, to each pair \(\mathfrak{w}\) and \(\mathfrak{w}'\) one can associate unconditionally a character \((\mathfrak{w},\mathfrak{w'})\) of \(S_\phi\) and then \(\iota_{\mathfrak{w'}}(\pi)=\iota_{\mathfrak{w}}(\pi)\otimes(\mathfrak{w},\mathfrak{w'})\).
    See \cite[\S 3]{Kaletha_genericity} for details.
    In particular, \(\dim(\iota_\mathfrak{w}(\pi))\) is independent of the choice of \(\mathfrak{w}\), and hence \(S\Theta_\phi\) is also independent.
  \item While Conjecture \ref{conj:crude_LLC} readily reduces to the essentially discrete case using Harish-Chandra's work, the putative analogous reductions for Conjectures \ref{conj:refined_LLC_qs} and \ref{conj:endo_char_rel_qs} appear to be more subtle, involving the study of intertwining operators.
    See \cite{KeysShahidi} for character formulas in the case of principal series representations.
  \item Implicit in the conjecture is the fact that the choice of a semisimple \(s\) in its connected component in \(\pi_0(\ol{S}_{\phi})\) is irrelevant.
    One can reduce to the case where \(s\) is ``generic'' (implying that \(\phi_H\) is essentially discrete) by parabolic induction (which behaves well with respect to \(S \Theta\)).
  \item This conjecture reduces the characterization of the local Langlands correspondence to a characterization of the stable distributions \(S \Theta_{\phi}\).
    This is simpler than characterizing individual characters \(\Theta_\pi\) (stable conjugacy classes are essentially parametrized by ``characteristic polynomials'' whereas for conjugacy classes further arithmetic invariants are needed), and in cases where a formula for \(S \Theta_{\phi}\) is known it indeed has a simpler shape.
    Moreover in a global setting when we try to compare trace formulas for different groups in general we can only compare stable conjugacy classes, and so we need to reduce to stable distributions.
\end{enumerate}
\end{remark}

\subsection{Refined Langlands correspondence for non-quasi-split groups}
\label{sec:refined_LLC_inner}

Recall from Proposition \ref{pro:inner_forms_brd} that two connected reductive groups that are inner forms of each other have isomorphic Langlands dual groups, and thus the ``same'' Langlands parameters.
Vogan's idea is to consider the L-packets \(\Pi_\phi(G)\), for a given \(\phi\) and \(G\) varying in an inner class, as one big L-packet \(\Pi_\phi\).
It is natural to take the quasi-split group given in Proposition \ref{pro:inner_forms_brd} as ``base point'' in the inner class because we already have a satisfying conjecture in this case, and for reasons explained below.
So we fix a quasi-split group \(G^*\).
Recall that isomorphism classes of inner twists of \(G^*\) are parametrized by \(H^1(F, G^*_\ad)\).
We may consider the groupoid of triples \((G,\psi,\pi)\) where \((G,\psi)\) is an inner twist of \(G^*\) and \(\pi\) is an irreducible smooth representation of \(G(F)\), with the obvious notion of isomorphism.
The problem with this definition is that for an inner twist \((G, \psi)\) of \(G^*\) its automorphism group in \(\mathcal{IT}(G^*)\) is \(G_\ad(F)\), which acts non-trivially on the set of isomorphism classes of irreducible smooth representations of \(G(F)\).
For example, for \(G = \SL_{2,\R}\) the element \(\operatorname{diag}(-1,1)\) of \(G_\ad(\R) = \mathrm{PGL}_2(\R)\) swaps holomorphic and anti-holomorphic discrete series representations of \(\SL_2(\R)\) of a given weight.
This motivates the introduction of \emph{pure} inner twists: augment the datum \((G, \psi)\) with a \(1\)-cocycle \(z: \Gamma \to G^*(\ol{F})\) lifting
\begin{align*}
  \Gamma & \longrightarrow{} G_\ad(\ol{F}) \\
  \sigma & \longmapsto{} \psi^{-1} \sigma(\psi).
\end{align*}
This effectively solves the above problem but creates a new one because the map \(H^1(F, G^*) \to H^1(F, G^*_\ad)\) is not surjective in general.
For \(F=\R\) Adams, Barbasch and Vogan \cite{AdamsBarbaschVogan} found an ad-hoc generalization of \(Z^1(\R,G^*)\), called strong real forms, that surjects onto \(H^1(\R,G^*_\ad)\).
Kottwitz suggested using his theory of isocrystals with additional structure \cite{Kottwitz_isoc1} \cite{Kottwitz_isoc2} in the case of non-Archimedean fields of characteristic zero as a generalization of \(H^1(F,G^*)\).
This suggestion was implemented completely by Kaletha and will be recalled below, but unfortunately it does not capture all inner forms of a given quasi-split group in general.
Kaletha later introduced another generalization of inner forms, called \emph{rigid} inner forms, for any local field \(F\) of characteristic zero and which captures all inner forms.
Specializing to \(F=\R\) recovers strong real forms.
It turns out that all of these generalizations can be understood as replacing the Galois group \(\Gamma\) (or the étale site of \(\Spec F\)) by an appropriate \emph{Galois gerbe}.
We summarize the three theories (pure, isocrystal and rigid) for a local field \(F\) of characteristic zero below and refer to \cite{Dillery_ri} for Dillery's generalization to functions fields, which uses Čech cohomology instead of Galois cohomology and also provides a more conceptual point of view using actual gerbes.

We thus assume for the rest of this section that \(F\) has characteristic zero.
In characteristic zero and for a commutative band, following \cite{LanglandsRapoport} the aforementioned Galois gerbes may prosaically be defined as group extensions
\[ 1 \to u(\ol{F}) \to \E \to \Gamma \to 1 \]
where \(u\) is a commutative group scheme over \(F\) and the action by conjugation of \(\Gamma\) on \(u(\ol{F})\) coincides with the usual one.
In practice \(u\) is a projective limit of groups \((u_i)_{i \geq 0}\) of multiplicative type and finite type over \(F\) with surjective morphisms between them, and the extension \(\E\) is built from a class in \(H^2_\cont(\Gamma, u(\ol{F}))\) where \(u(\ol{F})\) is endowed with the topology induced by the discrete topology on each \(u_i(\ol{F})\).
Note that we have set-theoretic sections \(\Gamma \to \E\), endowing \(\E\) with a natural topology.
Define \(H^1_\alg(\E, G) \subset H^1_\cont(\E, G(\ol{F}))\) as the subset of classes of \(1\)-cocycles \(\E \to G(\ol{F})\) whose restriction to \(u(\ol{F})\) is given by an algebraic morphism from \(u_{\ol{F}}\) to \(G_{\ol{F}}\).
Define \(H^1_\bas(\E, G) \subset H^1_\alg(\E, G)\) as the set of classes of cocycles for which the algebraic morphism \(u_{\ol{F}} \to G_{\ol{F}}\) takes values in the center \(Z(G)(\ol{F})\).
By the cocycle condition it descends in this case to a morphism \(u \to Z(G)\) defined over \(F\).
Note that such a morphism is induced from a morphism \(u_i \to Z(G)\) for some index \(i\) because the center of \(G\) has finite type over \(F\).
We will also consider, for \(Z\) a subgroup scheme of \(Z(G)\), the subset \(H^1(u \to \E, Z \to G)\) of \(H^1_\bas(\E, G)\) consisting of classes of cocycles whose associated map \(u \to Z(G)\) factors through \(Z\).
Denote by \(Z^1_\alg(\E, G)\), \(Z^1_\bas(\E, G)\) etc.\ the corresponding sets of 1-cocycles \(\E \to G(\ol{F})\).

We consider three cases in parallel.
\begin{enumerate}
\item If we take \(u=1\) we obtain the trivial extension \(\Epur=\Gamma\), recovering the usual Galois cohomology group \(H^1(F,G)\).
\item Consider the pro-torus \(u\) over \(F\) with character group
  \[ X^*(u) =
    \begin{cases}
      \Q & \text{ if } F \text{ is non-Archimedean}, \\
      \frac{1}{2}\Z & \text{ if } F \simeq \R. \end{cases}\]
  (We exclude the case \(F \simeq \C\) here because it is essentially trivial.)
  We have
  \[ H^2_\cont(\Gamma, u(\ol{F})) \simeq
    \begin{cases}
      \widehat{\Z} \otimes_\Z \Q & \text{ if } F \text{ is non-Archimedean}, \\
      \Z/2\Z & \text{ if } F \simeq \R.
    \end{cases} \]
  Let \(\Eiso\) be the extension of \(\Gamma\) by \(u(\ol{F})\) corresponding to the class of \(1\).
\item Consider the pro-finite algebraic group \(u\) over \(F\) with character group \(X^*(u)\) the set of locally constant functions \(f: \Gamma \to \Q/\Z\) satisfying \(\sum_{\sigma \in \Gamma} f(\sigma) = 0\) if \(F\) is Archimedean.
  We have
  \[ H^2_\cont(\Gamma, u(\ol{F})) \simeq
    \begin{cases}
      \widehat{\Z} & \text{ if } F \text{ is non-Archimedean},\\
      \Z/2\Z & \text{ if } F \simeq \R.
    \end{cases} \]
  (As above we exclude the case \(F \simeq \C\).)
  Let \(\Erig\) be the extension of \(\Gamma\) by \(u(\ol{F})\) corresponding to the class \(-1\).
\end{enumerate}

\begin{definition}
  A pure (resp.\ isocrystal, resp.\ rigid) inner twist of \(G\) is a triple \((G',\psi,z)\) where \((G',\psi)\) is an inner twist of \(G\) and \(z \in Z^1_\bas(\E^?, G)\) lifts
  \begin{align*}
    \Gamma & \longrightarrow{} G_\ad(\ol{F}) \\
    \sigma & \longmapsto{} \psi^{-1} \sigma(\psi)
  \end{align*}
  where \(? = \mathrm{pur}\) (resp.\ \(\mathrm{iso}\), resp.\ \(\mathrm{rig}\)).
\end{definition}
In all three cases the automorphism group of \((G',\psi,z)\) is isomorphic to \(G'(F)\).
Isocrystal inner twists are more commonly known as \emph{extended pure inner twists} (since \cite{Kaletha_iso}).

We have the following generalizations of the Tate-Nakayama isomorphisms.

\begin{theorem} \label{thm:TN_iso}
  We have natural maps
  \[ \kappa_G: H^1_\bas(\Eiso, G) \to X^*(Z(\Ghat)^\Gamma) \]
  extending the maps \(\alpha_G\) of Theorem \ref{thm:TN_redgps}, i.e.\ sitting in commutative diagrams
  \[
    \begin{tikzcd}
      H^1(F,G) \arrow[r, "{\alpha_G}"] \arrow[d, hook] & \Irr(\pi_0(Z(\Ghat)^\Gamma)) \arrow[d, hook] \\
      H^1_\bas(\Eiso, G) \arrow[r, "{\kappa_G}"] & X^*(Z(\Ghat)^\Gamma)
    \end{tikzcd}
  \]
  and functorial in \(G\) similarly to Theorem \ref{thm:TN_redgps} (in the case of an inclusion of a maximal torus \(T \subset G\) we have to restrict to elements of \(H^1_\bas(\Eiso, T)\) for which the induced map \(u \to T\) factors through \(Z(G)\)).

  The map \(\kappa_G\) is bijective if \(F\) is non-Archimedean.
\end{theorem}
\begin{proof}
  See \cite[Proposition 13.1 and Proposition 13.4]{Kottwitz_BG} and \cite[\S 3.1]{Kaletha_rivsbg}.
\end{proof}

For a connected reductive group \(G\) over \(F\) and a finite central subgroup scheme \(Z\) denote \(\ol{G} = G/Z\).
We have a dual isogeny \(\widehat{\ol{G}} \to \Ghat\); denote by \(Z(\widehat{\ol{G}})^+\) be the preimage of \(Z(\Ghat)^\Gamma\) in \(Z(\widehat{\ol{G}})\).
\begin{theorem}[{\cite[Corollary 5.4]{Kaletha_ri}}] \label{thm:TN_rig}
  We have natural maps
  \[ H^1(u \to \Erig, Z \to G) \to X^*(Z(\widehat{\ol{G}})^+) \]
  extending the maps \(\alpha_G\) and functorial in \(Z \to G\) as in Theorem \ref{thm:TN_redgps}.

  These maps are bijective in the non-Archimedean case.
\end{theorem}

We also have natural maps \(H^1(u \to \E^?, Z \to G) \to H^1(F, G/Z)\), and the above generalizations of the Tate-Nakayama morphism are also compatible with \(\alpha_{G/Z}\).
One can deduce that the maps \(H^1_\bas(\Eiso, G) \to H^1(F, G/Z(G)^0)\) and
\[ H^1(u \to \E^\mathrm{rig}, Z(G_\der) \to G) \to H^1(F, G_\ad) \]
are both surjective.
In particular all inner forms can be realized as rigid inner twists, or as isocrystal inner twists if the center of \(G\) is connected.
In general not all inner forms can be realized as isocrystal inner twists, e.g.\ when \(G\) is split semisimple but not adjoint.

There is \cite[\S 3.3]{Kaletha_rivsbg} a natural map of extensions \(\Erig \to \Eiso\), inducing \(H^1_\bas(\Eiso, G) \to H^1_\bas(\Erig, G)\) for any group \(G\).
The relation with Theorems \ref{thm:TN_iso} and \ref{thm:TN_rig} is not so obvious, see Proposition 3.3 loc.\ cit.

To simplify the notation for \(z \in Z^1_\bas(\E, G)\) we denote by \((G_z, \psi_z)\) the associated inner twist of \(G\).

\begin{conjecture} \label{conj:refined_LLC_inner}
  Let \(G^*\) be a quasi-split connected reductive group over \(F\).
  Let \(\mathfrak{w}\) be a Whittaker datum for \(G^*\).
  Let \(\phi: \WD_F \to {}^L G^*\) be a tempered Langlands parameter.
  Let \(? \in \{\mathrm{pur},\mathrm{iso},\mathrm{rig}\}\).
  Define \(\Pi_\phi^?\) as the set of isomorphism classes of pairs \((z,\pi)\) where \(z \in Z^1_\bas(\E^?, G^*)\) and \(\pi \in \Pi_\phi(G^*_z)\).
  Define
  \begin{enumerate}
  \item \(Z^\mathrm{pur} = 1\), \(\mathcal{S}_\phi^\mathrm{pur} = \pi_0(S_\phi)\) and \(\mathcal{Z}^\mathrm{pur} = \pi_0(Z(\Ghat)^\Gamma)\),
  \item \(Z^\mathrm{iso} = Z(G)^0\), \(\mathcal{S}_\phi^\mathrm{iso} = S_\phi / (S_\phi \cap \Ghat_\der)^0\) and \(\mathcal{Z}^\mathrm{iso} = Z(\Ghat)^\Gamma\),
  \item \(Z^\mathrm{rig}\) is any finite subgroup scheme of \(Z(G)\), \(\mathcal{S}_\phi^\mathrm{rig} = \pi_0(S_\phi^+)\) where \(S_\phi^+\) is the preimage of \(S_\phi\) in \(\widehat{\ol{G}}\) and \(\mathcal{Z}^\mathrm{rig} = \pi_0(Z(\widehat{\ol{G}})^+)\).
  \end{enumerate}
  There should exist a bijection \(\iota_\mathfrak{w}\) making the following diagram commutative.
  \[
    \begin{tikzcd}
      \Pi_\phi^? \arrow[r, "{\iota_\mathfrak{w}}" above, "{\sim}" below] \arrow[d] & \Irr(\mathcal{S}_\phi^?) \arrow[d] \\
      H^1(u \to \E^?, Z^? \to G^*) \arrow[r] & X^*(\mathcal{Z}^?)
    \end{tikzcd}
  \]
  Here the left vertical map is induced by the forgetful map \((z,\pi) \mapsto z\), the right vertical map is induced by the obvious map \(\mathcal{Z}^? \to \mathcal{S}_\phi^?\) and the bottom horizontal map is given by Theorem \ref{thm:TN_redgps} (resp.\ \ref{thm:TN_iso}, resp.\ \ref{thm:TN_rig}).
\end{conjecture}

The relation with Conjecture \ref{conj:crude_LLC} is that for any \(z \in Z^1(u \to \E^?, Z^? \to G^*)\) we should have \(\Pi_\phi(G_z^*) = \{ \pi \,|\, (z,\pi) \in \Pi_\phi^? \}\).

As for Conjectures \ref{conj:refined_LLC_qs} and \ref{conj:endo_char_rel_qs}, the map \(\iota_\mathfrak{w}\) in Conjecture \ref{conj:refined_LLC_inner} should be characterized by endoscopic character relations.
In order to state these relations we need normalized transfer factors, which are the values of a function defined on the set of matching pairs of strongly regular elements.
The definition of this function up to multiplication by a constant has been known for some time \cite{LanglandsShelstad}, but removing ``up to a constant'' (this is the meaning of ``normalized'') is a relatively recent achievement for non-quasi-split groups.
Their definition was suggested by Kottwitz and established by Kaletha in the case of pure inner forms \cite[\S 2.2]{Kaletha_endocharid_depth0} and extended to the isocrystal and rigid case by Kaletha \cite{Kaletha_iso} \cite{Kaletha_ri}.

Let \((G,\psi,z)\) be a pure/isocrystal/rigid inner twist of \(G^*\) and \(\phi: \WD_F \to {}^L G\)  a tempered Langlands parameter.
Consider a semi-simple \(s \in S_\phi\) if \(? \in \{\mathrm{pur},\mathrm{iso}\}\) or \(s \in S_\phi^+\) if \(?=\mathrm{rig}\).
As in Section \ref{sec:refined_LLC_qs} we obtain an extended endoscopic triple\footnote{A refined one in the rigid case, i.e.\ \(s\) belongs to the cover \(\widehat{\ol{G}}\) of \(\Ghat\).} \(\mathfrak{e} = (H,s,{}^L \eta)\) and a tempered Langlands parameter \(\phi_H:\WD_F \to {}^L H\).
Consider matching strongly regular \(\gamma \in H(F)\) and \(\delta \in G(F)\).
Using Steinberg's theorem \cite[Theorem I.7]{Steinberg_reg} \cite[\S 8.6]{BorelSpringer_ratprop2} we see that for any strongly regular \(\delta \in G(F)\) there exists \(\delta^* \in G^*(F)\) stably conjugate to \(\delta\), i.e.\ for which there exists \(g \in G^*(\ol{F})\) satisfying \(\psi(g^{-1} \delta^* g) = \delta\).
Clearly \(\delta^*\) is also strongly regular; denote its centralizer in \(G^*\) by \(T^*\).
For \(w \in \E^?\) we have \(g z_w w(g)^{-1} \in T^*(\ol{F})\) essentially because \(\delta^*\) and \(\delta\) are fixed by \(w\).
We denote by \(\inv[\psi,z](\delta^*, \delta) \in H^1(u \to \E^?, Z^? \to T^*)\) the class of the cocycle \(w \mapsto g z_w w(g)^{-1}\).
It does not depend on the choice of \(g\).
Similarly to the quasi-split case we can associate \(s_{\gamma,\delta^*} \in \widehat{T^*}^\Gamma\) (resp.\ \(\widehat{T^*}^\Gamma\), resp.\ \(\widehat{\ol{T^*}}^+\)) to \(s\) and the matching pair \((\gamma,\delta^*)\), and pair it with \(\inv(\delta^*, \delta)\) using Theorem \ref{thm:TN_redgps} (resp.\ \ref{thm:TN_iso}, resp.\ \ref{thm:TN_rig}).
In analogy with \eqref{eq:transfer_fact_equiv} define
\[ \Delta[\mathfrak{w}, \mathfrak{e}, \psi, z](\gamma, \delta) = \Delta[\mathfrak{w}, \mathfrak{e}](\gamma, \delta^*) \langle \inv(\delta^*, \delta), s_{\gamma,\delta^*} \rangle^{-1}. \]
It turns out that this is well-defined, i.e.\ the right-hand side does not depend on the choice of \(\delta^*\), and this defines a normalization of transfer factors for \((H,s,{}^L \eta)\).
Again there are several natural normalizations and in half of these normalizations the exponent \(-1\) should be removed.

We can now formulate the generalization of Conjecture \ref{conj:endo_char_rel_qs}.
As in Section \ref{sec:refined_LLC_qs} we abbreviate \(\langle s, \pi \rangle_{\mathfrak{w},z} = \tr \iota_\mathfrak{w}(z,\pi)(s)\) and define
\[ \Theta^{\mathfrak{w},z}_{\phi,s} = e(G_z) \sum_{\pi \in \Pi_\phi(G_z)} \langle s,\pi \rangle_{\mathfrak{w},z} \Theta_\pi \]
where \(e(G_z)\) is the sign defined by Kottwitz \cite{Kottwitz_sign}.

\begin{conjecture} \label{conj:endo_char_rel_inner}
  In the setting of Conjecture \ref{conj:refined_LLC_inner}, for any \(z \in Z^1(u \to \E^?, Z^? \to G^*)\), any strongly regular \(G_z(F)\)-conjugacy class \([\delta]\) and any semi-simple \(s \in S_\phi\) (resp.\ \(S_\phi\), resp.\ \(S_\phi^+\)) we should have
  \[ \Theta_{\phi,s}^{\mathfrak{w},z}(\delta) = \sum_{\gamma \in H(F)/st} \Delta[\mathfrak{w},\mathfrak{e},\psi,z](\gamma,\delta) S\Theta_{\phi_H}(\gamma). \]
\end{conjecture}

By linear independence of characters the conjecture implies that packets \(\Pi_{\phi_H}(H)\) for all endoscopic groups of \(G^*\) --- all quasi-split groups --- should determine the refined Langlands correspondence for all pure/isocrystal/rigid inner forms of \(G^*\).

If we fix an inner twist \(G,\psi\) of \(G^*\) then it may be realized as a rigid inner twist in more than one way: one can multiply \(z \in Z^1(u \to \Erig, Z \to G^*)\) by any class \(c\) in \(Z^1(u \to \Erig, Z \to Z) = Z^1_\alg(\Erig, Z)\).
By \cite[\S 6]{Kaletha_rivsbg} Conjecture \ref{conj:endo_char_rel_inner} for \(z\) implies the same conjecture for \(c z\).
In particular the same implication holds for pure inner twists.
Presumably the analogous implication should be valid in the isocrystal case.

\subsection{Reduction to the isocrystal case}

Let \(G^*\) be a quasi-split connected reductive group over a \(p\)-adic field \(F\).
As explained above all inner forms of \(G^*\) can be reached using the rigid theory, and one might be tempted to simply forget the pure and isocrystal versions.
They are simpler however, and the relative complexity of the rigid version is exacerbated in the global setting.
Another reason to favor the isocrystal version is that it seems more naturally related to geometric incarnations of the correspondence, as in \cite{FarguesScholze}.
It is thus useful to relate the isocrystal and rigid versions (the relation between the pure and isocrystal versions being rather obvious).

Let \(z^\mathrm{iso} \in Z^1_\bas(\Eiso, G^*)\) and let \(z^\mathrm{rig} \in Z^1_\bas(\Erig, G^*)\) be its pullback via \(\Erig \to \Eiso\).
For \(? \in \{\mathrm{iso},\mathrm{rig}\}\) the class of \(z^?\) defines a character \(\chi^?\) of \(\mathcal{Z}^?\).
As explained in \cite[\S 4]{Kaletha_rivsbg}, for a tempered Langlands parameter \(\phi\) the irreducible representations of \(\mathcal{S}_\phi^?\) with restriction to \(\mathcal{Z}^?\) given by \(\chi^?\) is the same for \(? \in \{\mathrm{iso},\mathrm{rig}\}\), and the endoscopic character relations are also the same.
In \S 5 loc.\ cit.\ Kaletha constructs an embedding \(G^* \to \widetilde{G}^*\) with normal image and abelian cokernel such that the center of \(\widetilde{G}^*\) is connected and such that Conjectures \ref{conj:refined_LLC_inner} and \ref{conj:endo_char_rel_inner} for \(G^*\) and \(\widetilde{G}^*\) are equivalent, under a natural assumption (see \S 5.2 loc.\ cit.).
Since these conjectures for \(\widetilde{G}^*\) can be reduced to the isocrystal case, it would be enough to prove Conjectures \ref{conj:refined_LLC_inner} and \ref{conj:endo_char_rel_inner} for all quasi-split groups in the isocrystal setting (as well as the aforementioned assumption) to deduce them for all quasi-split groups in the rigid setting, yielding ``the'' refined Langlands correspondence for all connected reductive groups.

\subsection{Relation with the crude version}

By \cite[Lemma 5.7]{Kaletha_ri} Conjecture \ref{conj:refined_LLC_inner} recovers the relevance condition on parameters discussed in \ref{sec:red_crude_LLC}.

One can formulate a more precise version of property \eqref{it:crude_LL_isog} in Conjecture \ref{conj:crude_LLC}.
Let \(f:G_1^* \to G_2^*\) be a central isogeny between quasi-split connected reductive groups over \(F\), inducing a dual map \(\hat{f}: {}^L G_2 \to {}^L G_1\).
Let \(\phi_2: \WD_F \to {}^L G_2\) be a tempered Langlands parameter and denote \(\phi_1 = \hat{f} \circ \phi_2\).
Let \(? \in \{\mathrm{pur},\mathrm{rig},\mathrm{iso}\}\).
We use the same notation as in Conjecture \ref{conj:refined_LLC_inner}, choosing finite central subgroups \(Z_i^\mathrm{rig}\) in the rigid case.
Up to enlarging these groups we may assume that \(Z_1^\mathrm{rig}\) contains the kernel of \(f\) and that its image is \(Z_2^\mathrm{rig}\).
Let \(z_1 \in Z^1(u \to \E^?, Z^? \to G_1^*)\) and let \(z_2\) be its image in \(Z^1(u \to \E^?, Z^? \to G_2^*)\).
Denote \(G_1 = G^*_{1,z_1}\) and \(G_2 = G^*_{2,z_2}\).
In all three cases \(\hat{f}\) induces a morphism \(\mathcal{S}_{\phi_2}^? \to \mathcal{S}_{\phi_1}^?\).
Let \(\mathfrak{w}\) be a Whittaker datum for \(G_1^*\) and \(G_2^*\).

\begin{conjecture}
  For any \(\pi_2 \in \Pi_{\phi_2}(G_2)\) we should have
  \[ \pi_2|_{G_1(F)} \simeq \bigoplus_{\pi_1 \in \Pi_{\phi_1}(G_1)} m(\pi_1,\pi_2) \pi_1 \]
  where \(m(\pi_1,\pi_2)\) is the multiplicity of \(\iota_\mathfrak{w}(z_2,\pi_2)\) in the restriction of \(\iota_\mathfrak{w}(z_1,\pi_1)\) to \(\mathcal{S}_{\phi_2}^?\).
\end{conjecture}

For simplicity we have formulated the conjecture for central isogenies, but as in point \eqref{it:crude_LL_isog} of Conjecture \ref{conj:crude_LLC} this conjecture implies a more general version for morphisms \(G_1^* \to G_2^*\) with central kernel, normal image and abelian cokernel.
See \cite{Solleveld_LLCisog} and \cite{BourgeoisMezo} for cases in which this conjecture is known.

\subsection{A non-exhaustive list of known cases}
\label{sec:known_cases}

In the case of real groups Conjectures \ref{conj:refined_LLC_inner} and \ref{conj:endo_char_rel_inner} were proved by Shelstad in many papers, see \cite{Shelstad_tempendo1}, \cite{Shelstad_tempendo2}, \cite{Shelstad_tempendo3} and \cite[\S 5.6]{Kaletha_ri}.

Hiraga and Saito \cite{HiragaSaito} proved Conjectures \ref{conj:refined_LLC_inner} and \ref{conj:endo_char_rel_inner}\footnote{This was before \cite{Kaletha_ri} so one should compare the normalizations of transfer factors.} for inner forms of \(\mathrm{SL}_n\) over non-Archimedean local fields of characteristic zero.

Arthur \cite{Arthur_book} proved Conjectures \ref{conj:refined_LLC_qs} and \ref{conj:endo_char_rel_qs} for quasi-split special orthogonal\footnote{In the even orthogonal case Arthur proved these conjectures ``up to outer automorphism''.} and symplectic groups over non-Archimedean fields of characteristic zero using, among other tools, the stabilization of the twisted trace formula \cite{SFTT1} \cite{SFTT2}.
In this case the stable characters \(S \Theta\) are characterized by \emph{twisted} endoscopy for the group \(\GL_N\) with its automorphism \(\theta: g \mapsto {}^t g^{-1}\) and the correspondence for general linear groups.
Note that endoscopic groups of special orthogonal or symplectic groups are products of similar groups and general linear groups.
Mok \cite{Mok_unitary} followed the same strategy to prove the conjectures for quasi-split unitary groups over non-Archimedean local fields of characteristic zero.
For completeness we recall that to our knowledge the main results of \cite{Arthur_book} and \cite{Mok_unitary} still depend on unpublished results.
These cases were then extended to certain inner forms:
\begin{itemize}
\item using the stabilization of the trace formula: to non-quasi-split unitary groups \cite{KMSW}, to non-quasi-split special orthogonal and unitary groups \cite{MoeglinRenard_inner}, to  non-quasi-split odd special orthogonal groups \cite{Ishimoto_LLC_nqs_SO_odd}, and
\item using theta correspondences: to non-quasi-split unitary groups \cite{ChenZou_LLC_unit_theta}.
\end{itemize}

Gan-Takeda \cite{GanTakeda_GSp4} and Chan-Gan \cite{ChanGan} proved Conjectures  \ref{conj:refined_LLC_inner} and \ref{conj:endo_char_rel_inner} for the groups \(\mathrm{GSp}_4\) over non-Archimedean local fields of characteristic zero, using theta correspondences and the stabilization of the trace formula.

The method of close fields of Deligne and Kazhdan allowed several authors to extend the existence of a map \(\LL\) for certain types of groups over non-Archimedean fields from the characteristic zero case to the positive characteristic case:
\begin{itemize}
\item \cite{Ganapathy_LLCGSp4} for \(\mathrm{GSp}_4\) (assuming that the characteristic is not \(2\)),
\item \cite{GanapathyVarma} for split symplectic and special orthogonal groups (with a restriction on the characteristic),
\item \cite{AubertBaumPlymenSolleveld_LLCinnerSL} for inner forms of \(\SL_n\).
\end{itemize}
This method gives internal structure of L-packets but does not seem to yield endoscopic character relations.

In the non-Archimedean cases mentioned above the characterizations of the local Langlands correspondence are rather indirect (using functoriality, global methods etc).
Of course it is desirable to have a more direct construction, like in the case of real groups \cite{Langlands_class}.
The existence of (many) supercuspidal representations implies that such a direct construction has to be much more complicated than in the real case.
Thanks to the work of many mathematicians (Moy and Prasad, Morris, Adler, Yu, Kim, Fintzen, Hakim and Murnaghan) we now have a ``direct'' classification of supercuspidal representations, i.e.\ one using representation theory rather than Langlands parameters, at least for tamely ramified connected reductive groups such that the residual characteristic \(p\) does not divide the order of the absolute Weyl group.
We refer the reader to \cite[\S 1.2]{Kaletha_ICM22} for more details and references.
Using this classification and building on work of Adler, DeBacker, Reeder and Spice, Kaletha constructed \cite{Kaletha_regscusp} \cite{Kaletha_scuspL}, under the above assumption on \((G,p)\) and for supercuspidal (i.e.\ essentially discrete and trivial on \(\SL_2\)) Langlands parameters \(\phi\), L-packets \(\Pi_\phi\) and natural parametrizations as in the rigid case of Conjecture \ref{conj:refined_LLC_inner}.
Under additional assumptions (\(p\) large enough and \(F\) of characteristic zero) Fintzen, Kaletha and Spice \cite[Theorem 4.4.4]{FintzenKalethaSpice_endo}, improving on previous work of Adler, DeBacker, Kaletha and Spice, proved stability and the endoscopic character relations of the rigid case of Conjecture \ref{conj:endo_char_rel_inner} for \(s\) in a certain subgroup of \(S_\phi^+\).
For the so-called \emph{regular} supercuspidal parameters this subgroup is \(S_\phi^+\), i.e.\ Conjecture \ref{conj:endo_char_rel_inner} holds.

There is much work to be done in this direction to handle, in order of increasing generality: all supercuspidal parameters, all essentially discrete parameters, all tempered parameters.
Recently Aubert and Xu \cite{AubertXu_LLC_G2} constructed a map \(\LL\) for the split group \(G_2\) over a finite extension of \(\Q_p\) for \(p \not \in \{2,3\}\), together with a parametrization of \(L\)-packets (Conjecture \ref{conj:refined_LLC_qs}).
Their work uses Kaletha's parametrization for supercuspidal representations and Hecke algebra techniques for the non-supercuspidal ones.
Gan and Savin \cite{GanSavin_LLC_G2} also constructed a map \(\LL\) for the split group \(G_2\) over a finite extension of \(\Q_p\) using theta correspondences with ``classical'' groups, and gave a parametrization of \(L\)-packets for \(p \neq 3\).

\section{Gerbes, Tannakian formalism and isocrystals} \label{sec:gerbe_tann}

We briefly mention the more conceptual point of view on gerbes and Tannakian categories, and explain how it relates the above definition of \(H^1_\alg(\Eiso,G)\) with the set \(B(G)\) of isocrystals with \(G\)-structure \cite{Kottwitz_isoc1} \cite{Kottwitz_isoc2}.
The latter point of view was motivated by the study of Shimura varieties over finite fields and historically came first.

\subsection{Gerbes and Tannakian formalism}

We first recall the equivalence between certain gerbes and Tannakian categories \cite[Théorème 3]{Saavedra} as corrected by \cite{Deligne_Tannaka}.
We consider fpqc stacks over \(F\).
Recall that a gerbe is a stack in groupoids admitting (fpqc) local sections and such that any two objects are (fpqc) locally isomorphic.
So as to not give a false sense of generality, let us mention right away that we will ultimately not need covers more complicated than \(\Spec \ol{F} \to \Spec F\).
For a stack in groupoids \(\mathcal{C}\) (we leave the functor \(p\) from \(\mathcal{C}\) to the category of schemes over \(F\) implicit) and a scheme \(S\) over \(F\) we denote by \(\mathcal{C}(S)\) the fiber of \(S\), i.e.\ the groupoid with objects those objects \(x\) of \(\mathcal{C}\) satisfying \(p(x) = S\) and morphisms those morphisms of \(\mathcal{C}\) over \(\id_S\).
This notation is meant to suggest thinking of \(\mathcal{C}\) as a sheaf in groupoids (which is another possible definition of stacks in groupoids if one is inclined to using higher categories, namely the \((2,1)\)-category of groupoids).
For a morphism of schemes \(f: U \to V\) and an object \(x\) of \(\mathcal{C}(V)\) we will denote by \(f^* x \xrightarrow{\beta(f,x)} x\) a strongly cartesian morphism in \(\mathcal{C}\) above \(f\).
We will leave the coherence isomorphisms of functors \(\id^* \simeq \id\) and \(f^* g^* \simeq (g f)^*\) implicit, using equalities to lighten the notation.
A gerbe \(\mathcal{C}\) is said to have affine band if for any scheme \(S\) over \(F\) and any two objects \(x,y\) of \(\mathcal{C}(S)\) the sheaf \(\ul{\Isom}(x,y): (T \xrightarrow{f} S) \mapsto \Isom_{\mathcal{C}(T)}(f^* x, f^* y)\) is representable by an affine scheme over \(S\).
If this holds for one non-empty \(S\) and one pair \((x,y)\) then \(\mathcal{C}\) has affine band \cite[p.\ 131]{Deligne_Tannaka}.
If a gerbe has affine band we simply say that it is affine.

The prime example of a gerbe is the stack \(BG\) of right \(G\)-torsors associated to some group scheme \(G\) over \(F\): for any scheme \(S\) over \(F\) the fibre \(BG(S)\) is the groupoid of \(G_S\)-torsors on \(S\).
General gerbes may be thought as twisted versions of this example, not admitting a distinguished global section (for \(BG\), the trivial \(G\)-torsor).

A representation \(R\) of a gerbe \(\mathcal{C}\) is a morphism from \(\mathcal{C}\) to the stack of quasi-coherent sheaves (over varying schemes over \(F\)).
A representation may also be intuitively understood as a quasi-coherent sheaf on \(\mathcal{C}\).
For a scheme \(S\) over \(F\) and an object \(x\) of \(\mathcal{C}(S)\) the quasi-coherent sheaf \(R(x)\) over \(S\) is automatically flat, and if it has finite rank \(n\) for some pair \((S,x)\) then \(R(y)\) has the same rank for any object \(y\) of \(\mathcal{C}\) \cite[\S 3.5]{Deligne_Tannaka}.
In that case we may see \(R\) as a morphism from \(\mathcal{C}\) to the stack of vector bundles of rank \(n\) (equivalently, \(\GL_n\)-torsors).
Finite-dimensional representations of \(\mathcal{C}\) form a category \(\Rep(\mathcal{C})\), that can be endowed with a tensor product (taking tensor products of vector bundles).
In fact \(\Rep(\mathcal{C})\) is a tensor category over \(F\) (in the sense of \cite[\S 2.1]{Deligne_Tannaka}).
Because \(\mathcal{C}\) has local sections the tensor category \(\Rep(\mathcal{C})\) is even Tannakian, i.e.\ it admits a fiber functor \cite[\S 1.9]{Deligne_Tannaka} over some non-empty scheme over \(F\).
For example for a group scheme \(G\) over \(F\) the gerbe \(BG\) has associated Tannakian category \(\Rep(BG)\), which is equivalent to the category \(\Rep(G)\) of finite-dimensional representations of \(G\).

To any tensor category \(\mathcal{T}\) over \(F\) we can associate the fibered category (over schemes over \(F\)) of fiber functors of \(\mathcal{T}\), denoted by \(\Fib(\mathcal{T})\).
If \(\mathcal{T}\) is Tannakian then \(\Fib(\mathcal{T})\) is an affine gerbe and the natural tensor functor \(\mathcal{T} \to \Rep(\Fib(\mathcal{T}))\) is an equivalence.
Conversely for a gerbe \(\mathcal{C}\) we also have a natural morphism of stacks \(\mathcal{C} \to \Fib(\Rep(\mathcal{C}))\) which is an equivalence if and only if \(\mathcal{C}\) is affine.

For an affine gerbe \(\mathcal{C}\) and a linear algebraic group \(G\) over \(F\) we can consider morphisms of stacks from \(\mathcal{C}\) to the gerbe \(BG\) of \(G\)-torsors, generalizing the notion of representation of \(\mathcal{C}\).
Such a morphism may also be interpreted as a \(G\)-torsor on \(\mathcal{C}\) (see \cite[\S 2.4]{Dillery_ri}).
By the correspondence recalled above such a morphism amounts to a morphism of tensor categories \(\Rep(G) \to \Rep(\mathcal{C})\).
The set\footnote{This might not be a set in general, but it turns out to be a set at least if the band of \(\mathcal{C}\) is a commutative group scheme (over \(F\)) that is a countable projective limit of finite type commutative group schemes: see \cite[Corollary 2.57 and Lemma 2.63]{Dillery_ri}.} of isomorphism classes of morphisms \(\mathcal{C} \to BG\) will be denoted by \(H^1(\mathcal{C}, G)\).

We now show how to pass from an abstract gerbe (as considered in this section) to a Galois gerbe in the sense of Section \ref{sec:refined_LLC_inner}.
We now assume that \(F\) has characteristic zero and specialize to the case of an affine gerbe \(\mathcal{C}\) whose band \(u\) is commutative, so that \(u\) is an affine commutative group scheme over \(F\).
It is isomorphic to a projective limit, over a directed poset \(I\), of commutative group schemes of finite type \((u_i)_{i \in I}\).
We assume further that \(I\) may be chosen to be countable.
We make this assumption because it implies that any projective limit over \(I\) of non-empty sets with surjective transition maps is itself not empty.
We may identify \(\mathcal{C}\) with a projective limit of gerbes \(\mathcal{C}_i\) bound by \(u_i\) (equivalently, we may identify the Tannakian category \(\Rep(\mathcal{C})\) with a union of tensor subcategories admitting a tensor generator).
Recall from \cite[Chapitre III Théorème 3.1.3.3]{Saavedra} or \cite[Corollaire 6.20]{Deligne_Tannaka} that \(\mathcal{C}_i\) admits a section over a finite extension of \(F\).
It follows that \(\mathcal{C}\) admits a section over \(\ol{F}\), say \(x\).
By our assumption on \(I\) the fiber \(\mathcal{C}(\Spec \ol{F})\) has only one isomorphism class (i.e.\ every \(u_{\ol{F}}\)-torsor is trivial).
This implies that the group \(\Aut_{\mathcal{C}}(x)\) is an extension \(\E\) of \(\Gamma\) by \(\Aut_{\mathcal{C}(\Spec \ol{F})}(x) = u(\ol{F})\).
Moreover for any \(x \in \mathcal{C}(\Spec \ol{F})\) the two pullbacks \(p_1^* x\) and \(p_2^* x\) in \(\mathcal{C}(\Spec \ol{F}^{\otimes 2})\) are isomorphic \cite[Lemma 2.63]{Dillery_ri}.
Choose such an isomorphism \(\varphi: p_1^* x \simeq p_2^* x\).
For \(\sigma \in \Gamma\) pulling back \(\varphi\) via the morphism \((\sigma \otimes \id)^\sharp: \Spec \ol{F} \to \Spec \ol{F}^{\otimes 2}\) dual to
\begin{align*}
  \sigma \otimes \id: \ol{F} \otimes_F \ol{F} & \longrightarrow{} \ol{F} \\
  x \otimes y & \longmapsto \sigma(x) y
\end{align*}
yields an isomorphism \((\sigma^\sharp)^* x \simeq x\).
We obtain \(\varphi_\sigma := ((\sigma \otimes \id)^\sharp)^* \varphi \circ \beta(\sigma^\sharp,x)^{-1}\) in \(\E\) above \((\sigma^\sharp)^{-1}\).
Thus \(\varphi\) determines  a (set-theoretic) splitting \(\Gamma \to \Aut_{\mathcal{C}}(x), \sigma \mapsto \varphi_\sigma\).
Taking the ``coboundary'' \(d\varphi = (p_{13}^* \varphi)^{-1} \circ (p_{23}^* \varphi) \circ (p_{12}^* \varphi)\) yields an automorphism of the pullback of \(x\) via the first projection \(\Spec \ol{F}^{\otimes 3} \to \Spec \ol{F}\), i.e.\ an element of \(u(\ol{F}^{\otimes 3})\), and one can check that it is a Čech \(2\)-cocycle \cite[Fact 2.31]{Dillery_ri}.
Conversely from any such \(2\)-cocycle one can construct a gerbe bound by \(u\) \cite[Proposition 2.36]{Dillery_ri}, and two gerbes bound by \(u\) are isomorphic if and only if their associated class in \(\check{H}^2(\ol{F}/F,u)\) are equal.
Yet another projective limit argument shows that we have a natural isomorphism \(\check{H}^2(\ol{F}/F,u) \simeq H^2_{\cont}(\Gamma, u(\ol{F}))\).
Thus the two notions of gerbes coincide.

We now compare the cohomology sets valued in a linear algebraic group \(G\) over \(F\), under the same assumptions.
For a morphism of stacks \(R: \mathcal{C} \to BG\), \(R\) factors through \(\mathcal{C}_i\) for some \(i \in I\) (equivalently, \(\Rep(G)\) has a tensor generator \cite[Proposition 2.20 (b)]{DeligneMilne_Tannakian} and so the tensor functor \(\Rep(G) \to \Rep(\mathcal{C})\) factors through a sub-tensor category of \(\Rep(\mathcal{C})\) generated by a single object).
Choosing a trivialization \(t: R(x) \simeq q^* T\) where \(q: \Spec \ol{F} \to \Spec F\) and \(T\) is the trivial \(G\)-torsor on \(\Spec F\) we obtain a morphism
\begin{align*}
  c(R,t)|_u: u_{\ol{F}} & \longrightarrow G_{\ol{F}} \\
  \gamma \in u(S \xrightarrow{f} \Spec \ol{F}) & \longmapsto f^* t \circ R(\gamma) \circ (f^* t)^{-1}.
\end{align*}
Some other choice of trivialization would give the same morphism conjugated by some element of \(G(\ol{F})\).
Furthermore we have an isomorphism
\begin{align*}
  G(\ol{F}) \rtimes \Gamma & \longrightarrow \Aut_{BG}(q^* T) \\
  g \rtimes \sigma & \longmapsto g \circ \beta(\sigma^\sharp, q^* T)^{-1}
\end{align*}
using \((\sigma^\sharp)^* q^* T = q^* T\), and so restricting \(R\) to \(\E\) gives us a morphism
\begin{align*}
  c(R,t): \E & \longrightarrow G(\ol{F}) \rtimes \Gamma \\
  w & \longmapsto t R(w) t^{-1} \beta(\sigma^\sharp, q^* T) \rtimes \sigma
\end{align*}
where \(\sigma\) is the image of \(w\) in \(\Gamma\), i.e.\ \(w\) lies above the automorphism \((\sigma^\sharp)^{-1}\) of \(\Spec \ol{F}\).
We may also see \(c(R,t)\) as a \(1\)-cocycle \(\E \to G(\ol{F})\).
The restriction of \(c(R,t)\) to \(u(\ol{F})\) is the morphism induced by \(c(R,t)|_u\), in particular it is algebraic.
Note that this also implies that if \(c(R,t)|_u\) factors through the center of \(G_{\ol{F}}\) then \(c(R,t)|_u\) is defined over \(F\).
The following proposition could certainly be deduced from \cite[Propositions 2.50 and 2.54]{Dillery_ri}, translating between Čech and Galois cocycles.
For the convenience of the reader we give a self-contained proof.

\begin{proposition} \label{pro:translation_cocycles_abstract_gerbes}
  The map \((R,t) \mapsto c(R,t)\) gives a bijection between isomorphism classes\footnote{An isomorphism between two pairs \((R_1,t_1)\) and \((R_2,t_2)\) is a morphism of functors \(f: R_1 \to R_2\) satisfying \(t_2 \circ f(x) = t_1\).} of pairs \((R, t)\) consisting of a morphism of stacks \(R: \mathcal{C} \to BG\) and a trivialization \(t\) of \(R(x)\), and \(Z^1_{\alg}(\E, G(\ol{F}))\).
  It induces a bijection \(H^1(\mathcal{C}, G) \to H^1_{\alg}(\E, G(\ol{F}))\) (forgetting the trivialization \(t\)).
\end{proposition}
\begin{proof}
  It is easy to deduce from the definition that isomorphic pairs yield the same cocycle, so we have a well-defined map.
  In order to prove that this map is a bijection the key construction consists of associating to \(c \in Z^1_{\alg}(\E, G(\ol{F}))\) a morphism \(\nu_c: \ul{\Isom}(p_1^* x, p_2^* x) \to G_{\ol{F}^{\otimes 2}}\) (of fpqc sheaves over \(\Spec \ol{F}^{\otimes 2}\), both representable by affine schemes).
  Choose \(\varphi \in \ul{\Isom}(p_1^* x, p_2^* x)(\Spec \ol{F}^{\otimes 2})\).
  Since \(\ul{\Isom}(p_1^* x, p_2^* x)\) is a torsor under \(\ul{\Aut}(p_1^* x) = u_{\ol{F}^{\otimes 2}}\) and we already have \(c|_u: u_{\ol{F}} \to G_{\ol{F}}\) (a morphism of group schemes, determined by the restriction of \(c\) to \(u(\ol{F})\)) it is enough to define \(\nu_c(\varphi) \in G(\ol{F}^{\otimes 2})\).
  There exists a quotient \(u \onto v\) such that \(v\) is of finite type over \(F\) and the restriction of \(c\) to \(u\) factors through \(v\).
  There exists a finite Galois subextension \(E/F\) of \(\ol{F}/F\) such that
  \begin{itemize}
  \item the \(2\)-cocycle \(\Gamma \to u(\ol{F})\) associated to the \(1\)-cochain \(\sigma \mapsto \varphi_\sigma\), when projected to \(v(\ol{F})\), takes values in \(v(E)\) and factors through \(\Gamma^2 \onto \Gal(E/F)^2\),
  \item the map \(\sigma \mapsto c(\varphi_\sigma)\) is inflated from a map \(\Gal(E/F) \to G(E)\).
  \end{itemize}
  Then \((c(\varphi_\sigma))_{\sigma \in \Gal(E/F)}\) can be seen as an element of \(G(E \otimes_F E)\), and there is a unique \(\nu_c(\varphi) \in G(E \otimes_F E)\) such that for any \(\sigma \in \Gal(E/F)\) we have \((\sigma \otimes \id) \nu_c(\varphi) = c(\varphi_\sigma)\).
  We then define, for \(f: S \to \Spec \ol{F}^{\otimes 2}\) and \(\gamma \in u(S \xrightarrow{p_1 \circ f} \Spec \ol{F})\),
  \[ \nu_c(f^* \varphi \circ \gamma) := f^* \nu_c(\varphi) \cdot c|_u(\gamma) \in G(S). \]
  A simple computation (left to the reader) using the fact that \(c\) is a cocycle shows that \(\nu_c\) does not depend on the choice of \(\varphi\).
  If \(c\) arises from \((R,t)\) then \(\nu_c\) is equal to the morphism induced by \((R,t)\) in the obvious way:
  \begin{align*}
    ((\sigma \otimes \id)^\sharp)^* \left( p_2^* t \circ R(\varphi) \circ (p_1^* t)^{-1} \right)
    &= t \circ R(\varphi_\sigma \circ \beta(\sigma^\sharp, x)) \circ ((\sigma^\sharp)^* t)^{-1} \\
    &= t \circ R(\varphi_\sigma) \circ t^{-1} \circ \beta(\sigma^\sharp, q^* T) \\
    &= c(\varphi_\sigma) \\
    &= (\sigma \otimes \id) \nu_c(\varphi).
  \end{align*}
  We check the following two properties of \(\nu_c\), that will in particular allow us to check that \(\nu_c\) maps descent data to descent data.
  \begin{enumerate}
  \item Seeing \(\delta := (p_{13}^* \varphi)^{-1} \circ p_{23}^* \varphi \circ p_{12}^* \varphi\) as an element of \(u(\Spec \ol{F}^{\otimes 3} \xrightarrow{p_1} \Spec \ol{F})\) we have in \(G(\ol{F}^{\otimes 3})\)
    \begin{equation} \label{eq:nu_c_rel1}
      p_{23}^* \nu_c(\varphi) \cdot p_{12}^* \nu_c(\varphi) = p_{13}^* \nu_c(\varphi) \cdot c|_u(\delta).
    \end{equation}
  \item For any morphism of schemes \(f: S \to \Spec \ol{F}^{\otimes 2}\) and any section \(\alpha \in (p_2^* u)(f) = u(S \xrightarrow{p_2 \circ f} \Spec \ol{F})\) we have in \(G(S)\)
    \begin{equation} \label{eq:nu_c_rel2}
      \Ad(f^* \nu_c(\varphi)) \left( c|_u((f^* \varphi)^{-1} \circ \alpha \circ f^* \varphi) \right)  = c|_u(\alpha).
    \end{equation}
  \end{enumerate}
  The first property is an equality in \(G(\ol{F}^{\otimes 3}) = \varinjlim_E G(E^{\otimes 3})\) where the colimit ranges over finite Galois subextensions \(E\) of \(\ol{F}/F\), so it is enough to check equality after pullbacks along morphisms of the form \((\sigma \tau \otimes \sigma \otimes \id)^\sharp: \Spec \ol{F} \to \Spec \ol{F}^{\otimes 3}\) for \(\sigma,\tau \in \Gamma\).
  The relation then becomes (details left to the reader)
  \[ c(\varphi_\sigma) \cdot \sigma \left( c(\varphi_\tau) \right) \overset{?}{=} c(\varphi_{\sigma \tau}) \cdot \sigma \tau \left( c(\varphi_{\sigma \tau}^{-1} \varphi_\sigma \varphi_\tau) \right) \]
  and this equality follows from the assumption that \(c\) is a cocycle.
  The second property is an equality between two morphisms of sheaves \(p_2^* u \to G_{\ol{F}^{\otimes 2}}\), and since they factor through \(p_2^* u_i\) for some \(i\) it is again enough to check equality in the case where \(f\) is \((\sigma \otimes \id)^\sharp: \Spec \ol{F} \to \Spec \ol{F}^{\otimes 2}\).
  In this case we simply have \(\alpha \in u(\ol{F})\) and \(f^* \varphi = \varphi_\sigma \circ \beta(\sigma^\sharp, x)\), so the relation becomes
  \[ (\Ad c(\varphi_\sigma)) \sigma(c(\varphi_\sigma^{-1} \alpha \varphi_\sigma)) \overset{?}{=} c(\alpha) \]
  and again this follows from the cocycle relation for \(c\).

  Now given \(c \in Z^1_{\alg}(\E, G(\ol{F}))\) we want to define a preimage \((R,t)\) of \(c\).
  We begin by defining \(R\) on objects.
  Let \(y\) in \(\mathcal{C}\) lying over an \(F\)-scheme \(S\).
  There exists an fpqc cover \(S' \to S_{\ol{F}}\) and an isomorphism \(\psi: a^* y \simeq b^* x\) where \(a: S' \to S\) and \(b: S' \to \Spec \ol{F}\).
  Via this isomorphism we may see the canonical descent datum for \(a^* y\) along \(a\) as a section \(\epsilon := \pi_2^* \psi \circ (\pi_1^* \psi)^{-1}\) of \(\ul{\Isom}(p_1^* x, p_2^* x)\) on \(b_2: S' \times_S S' \to \Spec \ol{F}^{\otimes 2}\), where \(\pi_i: S' \times_S S' \to S'\) are the projections.
  We now check that \(\nu_c(\epsilon)\) is also a descent datum (in \(BG\), for \(b^* q^* T\) along \(a\)).
  Writing \(\epsilon = b_2^* \varphi \circ \gamma\) where \(\gamma \in \Aut(b_2^* p_1^* x) = u(p_1 \circ b_2)\) we compute
  \begin{align*}
    \pi_{23}^* \nu_c(\epsilon) \cdot \pi_{12}^* \nu_c(\epsilon)
    &= \pi_{23}^* b_2^* \nu_c(\varphi) \cdot \pi_{23}^* c|_u(\gamma) \cdot \pi_{12}^* b_2^* \nu_c(\varphi) \cdot \pi_{12}^* c|_u(\gamma) \\
    &\overset{\eqref{eq:nu_c_rel2}}{=} \pi_{23}^* b_2^* \nu_c(\varphi) \cdot \pi_{12}^* b_2^* \nu_c(\varphi) \cdot c|_u \left( (\pi_{12}^* b_2^* \varphi)^{-1} \circ \pi_{23}^* \gamma \circ \pi_{12}^* b_2^* \varphi \circ \pi_{12}^* \gamma \right) \\
    &= b_3^* \left( p_{23}^* \nu_c(\varphi) \cdot p_{12}^* \nu_c(\varphi) \right) \cdot c|_u \left( (b_3^* p_{12}^* \varphi)^{-1} \circ \pi_{23}^* \gamma \circ \pi_{12}^* \epsilon \right) \\
    & \overset{\eqref{eq:nu_c_rel1}}{=} b_3^* p_{13}^* \nu_c(\varphi) \cdot c|_u \left( b_3^* \delta \circ (b_3^* p_{12}^* \varphi)^{-1} \circ \pi_{23}^* \gamma \circ \pi_{12}^* \epsilon \right) \\
    & = b_3^* p_{13}^* \nu_c(\varphi) \cdot c|_u \left( (b_3^* p_{13}^* \varphi)^{-1} \circ b_3^* p_{23}^* \varphi \circ \pi_{23}^* \gamma \circ \pi_{12}^* \epsilon \right) \\
    & = b_3^* p_{13}^* \nu_c(\varphi) \cdot c|_u \left( (b_3^* p_{13}^* \varphi)^{-1} \circ \pi_{23}^* \epsilon \circ \pi_{12}^* \epsilon \right) \\
    & = \pi_{13}^* b_2^* \nu_c(\varphi) \cdot c|_u \left( (\pi_{13}^* b_2^* \varphi)^{-1} \circ \pi_{13}^* \epsilon \right) \\
    & = \pi_{13}^* b_2^* \nu_c(\varphi) \cdot c|_u \left( \pi_{13}^* \gamma \right) \\
    &= \pi_{13}^* \nu_c(\epsilon).
  \end{align*}
  We obtain an object \(R(y)\) of \(BG(S)\) (and an isomorphism \(a^* R(y) \simeq b^* q^* T\)).
  Let us check that it does not depend on the choice of \(S'\) and \(\psi\).
  \begin{itemize}
  \item If we use \(\psi' = \gamma \circ \psi\) instead of \(\psi\), where \(\gamma \in \Aut(b^* x) = u(S' \xrightarrow{b} \Spec \ol{F})\) then denoting \(\epsilon' = \pi_2^* \psi' \circ (\pi_1^* \psi')^{-1}\) and \(\omega = c|_u(\gamma) \in G(S')\) we have
    \[ \nu_c(\epsilon') = c|_u(\pi_2^* \gamma) \cdot \nu_c(\epsilon) \cdot c|_u(\pi_1^* \gamma)^{-1} = \pi_2^* \omega \cdot \nu_c(\epsilon) \cdot (\pi_1^* \omega)^{-1} \]
    thanks to the fact that \(\nu_c\) is compatible with the \((p_1^* u, p_2^* u)\)-bitorsor structure on \(\ul{\Isom}(p_1^* x, p_2^* x)\) via \(c\) (for \(p_1^* u\) by definition of \(\nu_c\) and for \(p_2^* u\) by \eqref{eq:nu_c_rel2}).
    Thus \(\omega\) defines an isomorphism between the descent data \(\nu_c(\epsilon)\) and \(\nu_c(\epsilon')\).
  \item If we compose the cover \(S' \xrightarrow{b} S\) with a cover \(S'' \xrightarrow{d} S\) and use \(d^* \psi: (ad)^* y \simeq (bd)^* x\) to form \(\epsilon' = \rho_2^* \psi' \circ (\rho_1^* \psi')^{-1}\) where \(\rho_i: S'' \times_S S'' \to S''\) denote the canonical projections, then denoting \(d_2 = (d,d): S'' \times_S S'' \to S' \times_S S'\) we have \(\epsilon' = d_2^* \epsilon\), and we obtain an isomorphism between the objects of \(BG(S)\) associated to the descent data \(\nu_c(\epsilon)\) and \(\nu_c(\epsilon') = d_2^* \nu_c(\epsilon)\).
  \end{itemize}
  In the case where \(S=\Spec \ol{F}\) and \(y=x\) we may take \(S'=S\) and \(\psi = \id\), and we obtain a (distinguished) isomorphism between \(t: R(x) \simeq q^* T\).

  Next for a morphism \(d: U \to S\) we define an isomorphism \(R(d^* y) \simeq d^* R(y)\).
  To this end we use the cover \(U' := U \times_S S' \xrightarrow{a'} U\) and the isomorphism \(\psi' := (d')^* \psi: (a')^* d^* y \simeq (b')^* x\), where \(U' \xrightarrow{d'} S'\) and \(b'=b \circ d'\), to form \(\epsilon' := \rho_2^* \psi' \circ (\rho_1^* \psi')^{-1}\) where \(\rho_i: U' \times_U U' \to U'\) are the canonical projections.
  Then \(\epsilon'\) is the pullback of \(\epsilon\) along \(U' \times_U U' \to S' \times_S S'\) and we deduce by fpqc descent an isomorphism \(R(d^* y) \simeq d^* R(y)\).

  To complete the definition of \((R,t)\) it remains to associate an isomorphism \(R(f): R(y_1) \simeq R(y_2)\) to an isomorphism \(f: y_1 \simeq y_2\) in \(\mathcal{C}(S)\).
  Choose isomorphisms \(\psi_i: a^* y_i \simeq b^* x\) and let \(g\) be the element of \(\Aut(b^* x)\) making the following diagram (where all arrows are isomorphisms) commutative.
  \[
    \begin{tikzcd}
      a^* y_1 \ar[r, "{\psi_1}"] \ar[d, "{a^* f}"] & b^* x \ar[d, "{g}"] \\
      a^* y_2 \ar[r, "{\psi_2}"] & b^* x
    \end{tikzcd}
  \]
  We have
  \[ \epsilon_2 := \pi_2^* \psi_2 \circ (\pi_1^* \psi_2)^{-1} = \pi_2^* g \circ \left( \pi_2^* \psi_1 \circ (\pi_1^* \psi_1)^{-1} \right) \circ (\pi_1^* g)^{-1} = \pi_2^* g \circ \epsilon_1 \circ (\pi_1^* g)^{-1} \]
  and so \(c|_u(g) \in G(S')\) defines an isomorphism between the descent data \(\nu_c(\epsilon_1)\) and \(\nu_c(\epsilon_2)\), inducing an isomorphism \(R(f): R(y_1) \simeq R(y_2)\).
  We leave it to the reader to check that \(R(f)\) does not depend on the choice of \(\psi_1\) and \(\psi_2\) and that \(R\) (on morphisms) is compatible with composition.

  This concludes the construction of a preimage \((R,t)\) of \(c\).
  For uniqueness we simply stare at each step of the construction and observe that (up to unique isomorphism) we don't have any other choice to define \((R,t)\).
\end{proof}

In the positive characteristic case we do not have a good analogue of the extension \(u(\ol{F}) \to \E \to \Gamma\), which is why Dillery uses the abstract notion of gerbe in \cite{Dillery_ri}.
One could presumably work with just twisted Čech cocycles \cite[Definition 2.53]{Dillery_ri}, but doing so would obscure all constructions.

\subsection{Isocrystals}

For \(F\) a non-Archimedean local field of characteristic zero the gerbe corresponding to \(\Eiso\) was historically first introduced via its corresponding Tannakian category, the category of isocrystals.
We briefly recall this notion.
Let \(L\) be the completion of the maximal unramified extension of \(F\).
Denote by \(\sigma\) the Frobenius automorphism of \(L\).
An isocrystal is a finite-dimensional vector space \(V\) over \(L\) endowed with a \(\sigma\)-linear bijection \(\Phi: V \to V\).
They form a tensor category \(\Isoc_F\) for the obvious notion of tensor product.
(Among other axioms, we indeed have \(\End_{\Isoc_F}(1) = L^\sigma = F\).)
We have an obvious fiber functor for \(\Isoc_F\) over \(L\), namely \((V,\Phi) \mapsto V\), and so \(\Isoc_F\) is Tannakian.
By the Dieudonné-Manin classification theorem the tensor category \(\Isoc_F\) has a nice structure: it is semi-simple and its simple objects are parametrized by \(\Q\).
We briefly recall this classification and refer the reader to \cite[Chapitre VI \S 3.3]{Saavedra} for more details and references.
Fix a uniformizer \(\varpi\) of \(F\).
For \(r/s \in \Q\) for coprime \(r,s \in \Z\) with \(s>0\) we may construct the corresponding simple object of \(\Isoc_F\) as follows.
Let \(S(r/s)\) be \(L^s\) and define a \(\sigma\)-linear automorphism of \(S(r/s)\) as \(\sigma\) (on coordinates) post-composed with the linear automorphism of \(L^s\) with matrix
\[
  \begin{pmatrix}
     0 & 1 & 0 & \dots & 0 \\
     0 & 0 & 1 \\
       &&& \ddots\\
     0 & 0 & \dots & 0 & 1 \\
    \varpi^r&0&\dots&0&0 \end{pmatrix}. \]
This defines a simple object \(S(r/s)\) in \(\Isoc_F\).
The isomorphism class of \(S(r/s)\) does not depend on the choice of uniformizer \(\varpi\), and any simple object is isomorphic to \(S(q)\) for a uniquely determined \(q \in \Q\).
Denote by \(F_s\) the unramified extension of degree \(s\) of \(F\) in \(L\).
The \(F\)-algebra \(\End_{\Isoc_F}(S(r/s))\) embeds in the matrix algebra \(M_s(L)\), in fact it embeds in \(M_s(F_s)\) and it is a central simple algebra over \(F\) which is a division ring and is split by \(F_s\).
Its invariant in \(H^2(F, \mathbb{G}_m) \simeq \Q/\Z\) is simply the image of \(r/s\).
Any isocrystal \((V,\Phi)\) decomposes canonically as \(\bigoplus_{r/s \in \Q} V_{r/s}\) where
\[ V_{r/s} = L \otimes_{F_s} V^{\varpi^{-r} \Phi^s} \]
is the isotypic component isomorphic to a finite sum of copies of \(S(r/s)\).
The rational numbers \(q\) for which \(V_{q} \neq 0\) are called the slopes of \((V,f)\) and the above decomposition is called the slope decomposition.
An isocrystal \((V,\Phi)\) is said to be pure of slope \(q \in \Q\) if \(V_{q'} = 0\) for all \(q' \neq q\).
The tensor product of two isocrystals which are pure of slopes \(q_1\) and \(q_2\) is also pure, of slope \(q_1+q_2\).
The tensor category \(\Isoc_F\) is the union of its tensor subcategories \(\Isoc_{F,s}\) consisting of all isocrystals \((V,f)\) whose slopes \(q\) all satisfy \(qs \in \Z\).
The Tannakian category \(\Isoc_{F,s}\) admits a fiber functor over \(F_s\), namely
\[ \omega_s: (V,\Phi) \longmapsto \bigoplus_{r \in \Z} V^{\varpi^{-r} \Phi^s}. \]
If \(s\) divides \(s'\) then we have an obvious identification between \(F_{s'} \otimes_{F_s} \omega_s\) and \(\omega_{s'}\).
We obtain a fiber functor \(\omega\) for \(\Isoc_F\) over the maximal unramified extension of \(F\).
Thanks to the description of \(\End_{\Isoc_F}(S(r/s))\) recalled above we can compute the band \(u_s\) of (the gerbe of fiber functors of) \(\Isoc_{F,s}\) as the (commutative!) multiplicative group \(\mathbb{G}_m\) over \(F\).
For an \(F_s\)-algebra \(A\) and \(x \in A^\times\), \(x\) acts on the slope \(r/s\) part \(A \otimes_{F_s} V^{\varpi^{-r} \Phi^s}\) by multiplication by \(x^r\).
For \(s\) dividing \(s'\) the natural morphism \(u_{s'} \to u_s\) can be checked to be \(x \mapsto x^{s'/s}\), and so the band \(u\) of (the gerbe of fiber functors of) \(\Isoc_F\) is the split protorus with character group \(\Q\).
The class of the gerbe in
\[ H^2(F, u) \simeq H^2_\cont(\Gamma, u(\ol{F})) \simeq \varprojlim_s H^2(F, u_s) \simeq \varprojlim_s \Q/\Z \simeq \widehat{\Z} \otimes_{\Z} \Q \]
(the second isomorphism because each \(H^1(F, u_s)\) vanishes and so \(\varprojlim^1_s H^1(F, u_s)\) also vanishes) can be computed from the above description of endomorphisms of simple isocrystals and is simply equal to \(1\).

For a connected linear algebraic group \(G\) over \(F\) we can identify the set of isomorphism classes of tensor functors \(\Rep(G) \to \Isoc_F\) with \(B(G) := G(L)/\sim\) where \(g_1 \sim g_2\) if and only if there exists \(x \in G(L)\) for which \(g_2 = x g_1 \sigma(x)^{-1}\) (\(\sigma\)-conjugacy).
This is because \(H^1(L,G_L)\) is trivial \cite{Lang_quasialg} \cite[Theorem 1.9]{Steinberg_reg} and so there is up to isomorphism only one fiber functor for \(\Rep(G)\) over \(L\), namely \(\omega_{G,L}: (V,\rho) \mapsto L \otimes_F V\).
It follows that any tensor functor \(\Rep(G) \to \Isoc_F\) is isomorphic to one of the form \((V,\rho) \mapsto (L \otimes_F V, \Phi_{V,\rho})\).
It is clear that setting \(\Phi_{V,\rho} = \sigma \otimes \id_V\) gives a tensor functor, and any other tensor functor differs from this one by an automorphism of the fiber functor \(\omega_{G,L}\), i.e.\ by an element of \(G(L)\).
A similar argument shows that two elements of \(G(L)\) induce isomorphic tensor functors if and only if they are \(\sigma\)-conjugated.
This point of view on ``isocrystals with additional structure'' is historically the first one \cite{Kottwitz_isoc1} and was motivated by the study of Shimura varieties and Rapoport-Zink spaces.

We now briefly discuss the set \(B(G) \simeq H^1_\alg(\Eiso, G)\) in the case where \(G\) is a connected reductive group over \(F\).
We refer the reader to \cite{Kottwitz_isoc2} for more details.
The basic subset \(B(G)_\bas \simeq H^1_\bas(\Eiso, G)\) is completely described by the map \(\kappa_G\) of Theorem \ref{thm:TN_iso}.
Kottwitz constructed \cite[Lemma 6.1]{Kottwitz_AnnArbor} maps
\[ \kappa_G: B(G) \to X^*(Z(\Ghat)^\Gamma) \]
which as the notation suggests extend the maps of Theorem \ref{thm:TN_iso}.
(In fact the definition in the general and basic case are not different: Kottwitz first defined isomorphisms \(\kappa_T\) for all tori \(T\) and then extended the map to arbitrary connected reductive groups by reducing to the case where the derived subgroup is simply connected using z-extensions.)
We also have obvious maps of pointed sets
\[ \nu_G: B(G) \to \left( \Hom(u_{\ol{F}}, G_{\ol{F}}) / G(\ol{F})-\mathrm{conj} \right)^\Gamma, \]
called the Newton map.
By definition \(B(G)_\bas\) is the preimage under \(\nu_G\) of the subset \(\Hom(u,Z(G))\) of the target.
The pair \((\nu_G,\kappa_G)\) is injective on \(B(G)\) \cite[\S 4.13]{Kottwitz_isoc2}.
For a more precise description of non-basic classes see \cite[Theorem 5.4]{Kottwitz_isoc2}.

\bibliographystyle{amsalpha}

\begin{thebibliography}{DGA{\etalchar{+}}11}

\bibitem[ABPS16]{AubertBaumPlymenSolleveld_LLCinnerSL}
Anne-Marie Aubert, Paul Baum, Roger Plymen, and Maarten Solleveld, \emph{The
  local {Langlands} correspondence for inner forms of {{\(\mathrm{SL}_{n}\)}}},
  Res. Math. Sci. \textbf{3} (2016), 34 (English), Id/No 32.

\bibitem[ABV92]{AdamsBarbaschVogan}
Jeffrey Adams, Dan Barbasch, and David~A. Vogan, Jr., \emph{The {L}anglands
  classification and irreducible characters for real reductive groups},
  Progress in Mathematics, vol. 104, Birkh\"auser Boston, Inc., Boston, MA,
  1992.

\bibitem[Art13]{Arthur_book}
James Arthur, \emph{The {E}ndoscopic {C}lassification of {R}epresentations:
  {O}rthogonal and {S}ymplectic groups}, American Mathematical Society
  Colloquium Publications, vol.~61, American Mathematical Society, 2013.

\bibitem[AV16]{AdamsVogan_contra}
Jeffrey Adams and David~A. Vogan, Jr., \emph{Contragredient representations and
  characterizing the local {Langlands} correspondence}, Am. J. Math.
  \textbf{138} (2016), no.~3, 657--682 (English).

\bibitem[AX]{AubertXu_LLC_G2}
Anne-Marie Aubert and Yujie Xu, \emph{The {E}xplicit {L}ocal {L}anglands
  {C}orrespondence for $G_2$}, \url{https://arxiv.org/abs/2208.12391}.

\bibitem[Ber84a]{Deligne_centre_Bernstein}
I.~N. Bernstein, \emph{Le ``centre'' de {Bernstein}}, Repr{\'e}sentations des
  groupes r{\'e}ductifs sur un corps local, 1-32 (1984)., 1984.

\bibitem[Ber84b]{Bernstein_unitaryGL}
Joseph~N. Bernstein, \emph{{$P$}-invariant distributions on {${\rm GL}(N)$}\
  and the classification of unitary representations of {${\rm GL}(N)$}\
  (non-{A}rchimedean case)}, Lie group representations, {II} ({C}ollege {P}ark,
  {M}d., 1982/1983), Lecture Notes in Math., vol. 1041, Springer, Berlin, 1984,
  pp.~50--102.

\bibitem[BG14]{BuzzardGee_conj}
Kevin Buzzard and Toby Gee, \emph{The conjectural connections between
  automorphic representations and {G}alois representations}, Automorphic Forms
  and Galois Representations, Volume 1 (Fred Diamond, Payman~L. Kassaei, and
  Minhyong Kim, eds.), LMS Lecture Notes Series, vol. 414, Cambridge University
  Press, 2014, pp.~135--187.

\bibitem[BM25]{BourgeoisMezo}
Ad{\`e}le Bourgeois and Paul Mezo, \emph{A functoriality property for
  supercusipidal {{\(L\)}}-packets}, Pac. J. Math. \textbf{339} (2025), no.~1,
  73--131.

\bibitem[BMIY24]{MeliImaiYoucis}
Alexander Bertoloni~Meli, Naoki Imai, and Alex Youcis, \emph{The
  {Jacobson}-{Morozov} morphism for {Langlands} parameters in the relative
  setting}, Int. Math. Res. Not. \textbf{2024} (2024), no.~6, 5100--5165.

\bibitem[Bor79]{Borel_autLfunc}
A.~Borel, \emph{Automorphic {$L$}-functions}, Automorphic forms,
  representations and {$L$}-functions ({P}roc. {S}ympos. {P}ure {M}ath.,
  {O}regon {S}tate {U}niv., {C}orvallis, {O}re., 1977), {P}art 2, Proc. Sympos.
  Pure Math., XXXIII, Amer. Math. Soc., Providence, R.I., 1979, pp.~27--61.

\bibitem[Bor91]{Borel_lag}
Armand Borel, \emph{Linear algebraic groups}, second ed., Graduate Texts in
  Mathematics, vol. 126, Springer-Verlag, New York, 1991.

\bibitem[Bou12]{Bourbaki_alg8}
N.~Bourbaki, \emph{{\'E}l{\'e}ments de math{\'e}matique. {Alg{\`e}bre}.
  {Chapitre} 8. {Modules} et anneaux semi-simples.}, 2nd revised ed. of the
  1958 original ed., Berlin: Springer, 2012 (French).

\bibitem[BS68]{BorelSpringer_ratprop2}
Armand Borel and T.~A. Springer, \emph{Rationality properties of linear
  algebraic groups. {II}}, T{\^o}hoku Math. J. (2) \textbf{20} (1968), 443--497
  (English).

\bibitem[BT65]{BorelTits_gpesred}
Armand Borel and Jacques Tits, \emph{Groupes r\'{e}ductifs}, Inst. Hautes
  \'{E}tudes Sci. Publ. Math. (1965), no.~27, 55--150.

\bibitem[BT87]{BruhatTits_galcoh}
F.~Bruhat and Jacques Tits, \emph{Groupes alg{\'e}briques sur un corps local.
  {III}: {Compl{\'e}ments} et applications {\`a} la cohomologie galoisienne.
  ({Algebraic} groups over a local field. {III}: {Comments} and applications to
  {Galois} cohomology)}, J. Fac. Sci., Univ. Tokyo, Sect. I A \textbf{34}
  (1987), 671--698 (French).

\bibitem[BW00]{BorelWallach}
A.~Borel and N.~Wallach, \emph{Continuous cohomology, discrete subgroups, and
  representations of reductive groups}, second ed., Mathematical Surveys and
  Monographs, vol.~67, American Mathematical Society, Providence, RI, 2000.

\bibitem[BZ77]{BernsteinZelevinsky_ind1}
Joseph Bernstein and Andrei Zelevinsky, \emph{Induced representations of
  reductive {$\mathfrak{p}$}-adic groups. {I}}, Ann. Sci. \'Ecole Norm. Sup.
  \textbf{10} (1977), no.~4, 441--472.

\bibitem[Cas]{CasselmanBook}
William Casselman, \emph{Introduction to the theory of admissible
  representations of {$p$}-adic reductive groups}, available at
  \url{https://www.math.ubc.ca/~cass/research/pdf/p-adic-book.pdf}.

\bibitem[Cas77]{Casselman_charJac}
W.~Casselman, \emph{Characters and jacquet modules}, Math. Ann. \textbf{230}
  (1977), 101--105.

\bibitem[Cas80]{Casselman_unr}
William Casselman, \emph{The unramified principal series of {$p$}-adic groups.
  {I}. the spherical function}, Compositio Mathematica \textbf{40} (1980),
  no.~3, 387--406.

\bibitem[CG15]{ChanGan}
Ping-Shun Chan and Wee~Teck Gan, \emph{The local {L}anglands conjecture for
  {$\rm GSp(4)$} {III}: {S}tability and twisted endoscopy}, J. Number Theory
  \textbf{146} (2015), 69--133.

\bibitem[CGH14]{CluckersGordonHalupczok}
Raf Cluckers, Julia Gordon, and Immanuel Halupczok, \emph{Local integrability
  results in harmonic analysis on reductive groups in large positive
  characteristic}, Ann. Sci. {\'E}c. Norm. Sup{\'e}r. (4) \textbf{47} (2014),
  no.~6, 1163--1195 (English).

\bibitem[CS80]{CasselmanShalika}
W.~Casselman and J.~Shalika, \emph{The unramified principal series of
  {$p$}-adic groups. {II}. {T}he {W}hittaker function}, Compositio Math.
  \textbf{41} (1980), no.~2, 207--231. \MR{581582}

\bibitem[CZ21]{ChenZou_LLC_unit_theta}
Rui Chen and Jialiang Zou, \emph{Local {Langlands} correspondence for unitary
  groups via theta lifts}, Represent. Theory \textbf{25} (2021), 861--896
  (English).

\bibitem[Dat05]{Dat_nutemp}
J.-F. Dat, \emph{{{\(\nu\)}}-tempered representations of {{\(p\)}}-adic groups.
  i: {{\(l\)}}-adic case}, Duke Math. J. \textbf{126} (2005), no.~3, 397--469
  (English).

\bibitem[Del90]{Deligne_Tannaka}
P.~Deligne, \emph{Cat{\'e}gories tannakiennes. ({Tannaka} categories)}, The
  {Grothendieck} {Festschrift}, {Collect}. {Artic}. in {Honor} of the 60th
  {Birthday} of {A}. {Grothendieck}. {Vol}. {II}, {Prog}. {Math}. 87, 111-195
  (1990)., 1990.

\bibitem[DG70]{SGA3_II}
Michel Demazure and Alexander Grothendieck (eds.), \emph{Sch{\'e}mas en
  groupes. {II}: {Groupes} de type multiplicatif, et structure des sch{\'e}mas
  en groupes g{\'e}n{\'e}raux. {Expos{\'e}s} {VIII} {\`a} {XVIII}.
  {S{\'e}minaire} de {G{\'e}om{\'e}trie} {Alg{\'e}brique} du {Bois} {Marie}
  1962/64 ({SGA} 3) dirig{\'e} par {Michel} {Demazure} et {Alexander}
  {Grothendieck}. {Revised} reprint}, Lect. Notes Math., vol. 152, Springer,
  Cham, 1970 (French).

\bibitem[DGA{\etalchar{+}}11]{SGA3_III}
Michel Demazure, Alexander Grothendieck, M~Artin, J.-E. Bertin, P.~Gabriel,
  M.~Raynaud, and J.-P. Serre (eds.), \emph{S{\'e}minaire de g{\'e}om{\'e}trie
  alg{\'e}brique du {Bois} {Marie} 1962-64. {Sch{\'e}mas} en groupes ({SGA} 3).
  {Tome} {III}: {Structure} des sch{\'e}mas en groupes r{\'e}ductifs}, new
  annotated edition of the 1970 original published bei {Springer} ed., Doc.
  Math. (SMF), vol.~8, Paris: Soci{\'e}t{\'e} Math{\'e}matique de France, 2011.

\bibitem[Dil23]{Dillery_ri}
Peter Dillery, \emph{Rigid inner forms over local function fields}, Adv. Math.
\textbf{430} (2023), 100, Id/No 109204.

\bibitem[DM82]{DeligneMilne_Tannakian}
Pierre Deligne and J.~S. Milne, \emph{Tannakian categories}, Hodge cycles,
  motives, and {Shimura} varieties, {Lect}. {Notes} {Math}. 900, 101-228
  (1982)., 1982.

  \bibitem[FKS23]{FintzenKalethaSpice_endo}
Jessica Fintzen, Tasho Kaletha, and Loren Spice, \emph{A twisted {Yu}
  construction, {Harish}-{Chandra} characters, and endoscopy}, Duke Math. J.
  \textbf{172} (2023), no.~12, 2241--2301.

\bibitem[FS]{FarguesScholze}
Laurent Fargues and Peter Scholze, \emph{Geometrization of the local langlands
  correspondence}, \url{https://arxiv.org/abs/2102.13459}.

\bibitem[Gan15]{Ganapathy_LLCGSp4}
Radhika Ganapathy, \emph{The local {Langlands} correspondence for
  {{\(\mathrm{GSp}\)}} over local function fields}, Am. J. Math. \textbf{137}
  (2015), no.~6, 1441--1534 (English).

\bibitem[GHSBP24]{GanHarrisSawin}
Wee~Teck Gan, Michael Harris, Will Sawin, and Rapha{\"e}l Beuzart-Plessis,
  \emph{Local parameters of supercuspidal representations}, Forum Math. Pi
  \textbf{12} (2024), 41, Id/No e13.

\bibitem[GL]{GenestierLafforgue}
Alain Genestier and Vincent Lafforgue, \emph{Chtoucas restreints pour les
  groupes réductifs et paramétrisation de langlands locale},
  \url{https://arxiv.org/abs/1709.00978}.

\bibitem[GR10]{GrossReeder_disc}
Benedict~H. Gross and Mark Reeder, \emph{Arithmetic invariants of discrete
  {L}anglands parameters}, Duke Math. J. \textbf{154} (2010), no.~3, 431--508.
  \MR{2730575 (2012c:11252)}

\bibitem[GS23]{GanSavin_LLC_G2}
Wee~Teck Gan and Gordan Savin, \emph{The local {Langlands} conjecture for
  {{\(G_2\)}}}, Forum Math. Pi \textbf{11} (2023), 42, Id/No e28.

\bibitem[GT11]{GanTakeda_GSp4}
Wee~Teck Gan and Shuichiro Takeda, \emph{The local {L}anglands conjecture for
  {${\rm GSp}(4)$}}, Ann. of Math. (2) \textbf{173} (2011), no.~3, 1841--1882.

\bibitem[GV17]{GanapathyVarma}
Radhika Ganapathy and Sandeep Varma, \emph{On the local {Langlands}
  correspondence for split classical groups over local function fields}, J.
  Inst. Math. Jussieu \textbf{16} (2017), no.~5, 987--1074 (English).

\bibitem[Har]{Harris_LL}
Michael Harris, \emph{Local langlands correspondences},
  \url{https://arxiv.org/abs/2205.03848}.

\bibitem[HC99]{HC_adm_inv_dist}
Harish-Chandra, \emph{Admissible invariant distributions on reductive
  {$p$}-adic groups}, University Lecture Series, vol.~16, American Mathematical
  Society, Providence, RI, 1999, With a preface and notes by Stephen DeBacker
  and Paul J. Sally, Jr.

\bibitem[Hen00]{Henniart_LLC}
Guy Henniart, \emph{A simple proof of {Langlands} conjectures for
  {{\(\text{GL}_n\)}} on a {{\(p\)}}-adic field}, Invent. Math. \textbf{139}
  (2000), no.~2, 439--455 (French).

\bibitem[HS12]{HiragaSaito}
Kaoru Hiraga and Hiroshi Saito, \emph{On {$L$}-packets for inner forms of
  {$SL_n$}}, Mem. Amer. Math. Soc. \textbf{215} (2012), no.~1013, vi+97.

\bibitem[HT01]{HarrisTaylor}
Michael Harris and Richard Taylor, \emph{The geometry and cohomology of some
  simple {Shimura} varieties. {With} an appendix by {Vladimir} {G}.
  {Berkovich}}, Ann. Math. Stud., vol. 151, Princeton, NJ: Princeton University
  Press, 2001 (English).

\bibitem[Ima24]{Imai_LLladic}
Naoki Imai, \emph{Local {Langlands} correspondences in {{\({{\ell}} \)}}-adic
  coefficients}, Manuscr. Math. \textbf{175} (2024), no.~1-2, 345--364.

\bibitem[Ish24]{Ishimoto_LLC_nqs_SO_odd}
Hiroshi Ishimoto, \emph{The endoscopic classification of representations of
  non-quasi-split odd special orthogonal groups}, Int. Math. Res. Not.
  \textbf{2024} (2024), no.~14, 10939--11012.

\bibitem[Kala]{Kaletha_Lemb_tori}
Tasho Kaletha, \emph{On {L}-embeddings and double covers of tori over local
  fields}, \url{https://arxiv.org/abs/1907.05173}.

\bibitem[Kal23]{Kaletha_ICM22}
\bysame, \emph{Representations of reductive groups over local fields},
  International congress of mathematicians 2022, ICM 2022, Helsinki, Finland,
  virtual, July 6--14, 2022. Volume 4. Sections 5--8, Berlin: European
  Mathematical Society (EMS), 2023, pp.~2948--2975.

\bibitem[Kalc]{Kaletha_scuspL}
\bysame, \emph{Supercuspidal {L}-packets},
  \url{https://arxiv.org/abs/1912.03274}.

\bibitem[Kal11]{Kaletha_endocharid_depth0}
\bysame, \emph{Endoscopic character identities for depth-zero supercuspidal
  {{\(L\)}}-packets}, Duke Math. J. \textbf{158} (2011), no.~2, 161--224.

\bibitem[Kal13]{Kaletha_genericity}
\bysame, \emph{Genericity and contragredience in the local {Langlands}
  correspondence}, Algebra Number Theory \textbf{7} (2013), no.~10, 2447--2474
  (English).

\bibitem[Kal14]{Kaletha_iso}
\bysame, \emph{Supercuspidal {$L$}-packets via isocrystals}, Amer. J. Math.
  \textbf{136} (2014), no.~1, 203--239.

\bibitem[Kal16]{Kaletha_ri}
Tasho Kaletha, \emph{Rigid inner forms of real and $p$-adic groups}, Ann. of
  Math. (2) \textbf{184} (2016), no.~2, 559--632.

\bibitem[Kal18]{Kaletha_rivsbg}
Tasho Kaletha, \emph{Rigid inner forms vs isocrystals}, J. Eur. Math. Soc.
  (JEMS) \textbf{20} (2018), no.~1, 61--101 (English).

\bibitem[Kal19]{Kaletha_regscusp}
\bysame, \emph{Regular supercuspidal representations}, J. Am. Math. Soc.
  \textbf{32} (2019), no.~4, 1071--1170 (English).

\bibitem[KMSW]{KMSW}
Tasho Kaletha, Alberto Minguez, Sug~Woo Shin, and Paul-James White,
  \emph{Endoscopic {C}lassification of {R}epresentations: {I}nner {F}orms of
  {U}nitary {G}roups}, \url{https://arxiv.org/abs/1409.3731}.

\bibitem[Kos59]{Kostant_sl2}
Bertram Kostant, \emph{The principal three-dimensional subgroup and the {Betti}
  numbers of a complex simple {Lie} group}, Am. J. Math. \textbf{81} (1959),
  973--1032 (English).

\bibitem[Kot]{Kottwitz_BG}
Robert Kottwitz, \emph{{$B(G)$} for all local and global fields},
  \url{http://arxiv.org/abs/1401.5728}.

\bibitem[Kot83]{Kottwitz_sign}
Robert~E. Kottwitz, \emph{Sign changes in harmonic analysis on reductive
  groups}, Trans. Amer. Math. Soc. \textbf{278} (1983), no.~1, 289--297.

\bibitem[Kot84]{Kottwitz_STFcusptemp}
\bysame, \emph{Stable trace formula: cuspidal tempered terms}, Duke Math. J.
  \textbf{51} (1984), no.~3, 611--650.

\bibitem[Kot85]{Kottwitz_isoc1}
\bysame, \emph{Isocrystals with additional structure}, Compositio Math.
  \textbf{56} (1985), no.~2, 201--220.

\bibitem[Kot86]{Kottwitz_STFellsing}
Robert Kottwitz, \emph{Stable trace formula: elliptic singular terms},
  Mathematische Annalen \textbf{275} (1986), no.~3, 365--399.

\bibitem[Kot90]{Kottwitz_AnnArbor}
R.~E. Kottwitz, \emph{Shimura varieties and {$\lambda$}-adic representations},
  Automorphic forms, {S}himura varieties, and {$L$}-functions, {V}ol.\ {I}
  ({A}nn {A}rbor, {MI}, 1988), Perspect. Math., vol.~10, Academic Press,
  Boston, MA, 1990, pp.~161--209.

\bibitem[Kot97]{Kottwitz_isoc2}
Robert~E. Kottwitz, \emph{Isocrystals with additional structure. {II}},
  Compositio Math. \textbf{109} (1997), no.~3, 255--339.

\bibitem[KS]{KottwitzShelstad_corr}
Robert~E. Kottwitz and Diana Shelstad, \emph{On splitting invariants and sign
  conventions in endoscopic transfer}, \url{https://arxiv.org/abs/1201.5658}.

\bibitem[KS88]{KeysShahidi}
C.~David Keys and Freydoon Shahidi, \emph{Artin {L}-functions and normalization
  of intertwining operators}, Ann. Sci. {\'E}c. Norm. Sup{\'e}r. (4)
  \textbf{21} (1988), no.~1, 67--89 (English).

\bibitem[KS99]{KottwitzShelstad}
Robert~E. Kottwitz and Diana Shelstad, \emph{Foundations of twisted endoscopy},
  Ast\'erisque (1999), no.~255, vi+190.

\bibitem[Lan52]{Lang_quasialg}
Serge Lang, \emph{On quasi algebraic closure}, Ann. Math. (2) \textbf{55}
  (1952), 373--390 (English).

\bibitem[Lan89]{Langlands_class}
R.~P. Langlands, \emph{On the classification of irreducible representations of
  real algebraic groups}, Representation theory and harmonic analysis on
  semisimple {L}ie groups, Math. Surveys Monogr., vol.~31, Amer. Math. Soc.,
  Providence, RI, 1989, pp.~101--170.

\bibitem[LR87]{LanglandsRapoport}
R.~P. Langlands and M.~Rapoport, \emph{Shimuravariet\"aten und {G}erben}, J.
  Reine Angew. Math. \textbf{378} (1987), 113--220.

\bibitem[LRS93]{LaumonRapoportStuhler}
G.~Laumon, M.~Rapoport, and U.~Stuhler, \emph{{{\({\mathcal D}\)}}-elliptic
  sheaves and the {Langlands} correspondence}, Invent. Math. \textbf{113}
  (1993), no.~2, 217--338 (English).

\bibitem[LS87]{LanglandsShelstad}
R.~P. Langlands and D.~Shelstad, \emph{On the definition of transfer factors},
  Math. Ann. \textbf{278} (1987), no.~1-4, 219--271.

\bibitem[Mok15]{Mok_unitary}
Chung~Pang Mok, \emph{Endoscopic classification of representations of
  quasi-split unitary groups}, Mem. Amer. Math. Soc. \textbf{235} (2015),
  no.~1108, vi+248.

\bibitem[MR18]{MoeglinRenard_inner}
Colette Moeglin and David Renard, \emph{On {Arthur} packets of classical and
  unitary non-quasi-split groups}, Relative aspects in representation theory,
  Langlands functoriality and automorphic forms. CIRM Jean-Morlet Chair, spring
  2016, Cham: Springer; Paris: Soci{\'e}t{\'e} Math{\'e}matique de France
  (SMF), 2018, pp.~341--361 (French).

\bibitem[MT02]{MoeglinTadic}
Colette Moeglin and Marko Tadic, \emph{Construction of discrete series for
  classical {{\(p\)}}-adic groups}, J. Am. Math. Soc. \textbf{15} (2002),
  no.~3, 715--786 (English).

\bibitem[MW16a]{SFTT1}
Colette M{\oe}glin and Jean-Loup Waldspurger, \emph{Stabilisation de la formule
  des traces tordue}, vol.~1, Springer International Publishing, 2016.

\bibitem[MW16b]{SFTT2}
\bysame, \emph{Stabilisation de la formule des traces tordue}, vol.~2, Springer
  International Publishing, 2016.

\bibitem[Pra19]{Prasad_contragredient}
Dipendra Prasad, \emph{Generalizing the {MVW} involution, and the
  contragredient}, Trans. Amer. Math. Soc. \textbf{372} (2019), no.~1,
  615--633.

\bibitem[Sch13]{Scholze_LLC}
Peter Scholze, \emph{The local {Langlands} correspondence for
  {{\(\mathrm{GL}_n\)}} over {{\(p\)}}-adic fields}, Invent. Math. \textbf{192}
  (2013), no.~3, 663--715 (English).

\bibitem[She08a]{Shelstad_tempendo1}
D.~Shelstad, \emph{Tempered endoscopy for real groups. {I}. {G}eometric
  transfer with canonical factors}, Representation theory of real reductive
  {L}ie groups, Contemp. Math., vol. 472, Amer. Math. Soc., Providence, RI,
  2008, pp.~215--246.

\bibitem[She08b]{Shelstad_tempendo3}
\bysame, \emph{Tempered endoscopy for real groups. {III}. {I}nversion of
  transfer and {$L$}-packet structure}, Represent. Theory \textbf{12} (2008),
  369--402.

\bibitem[She10]{Shelstad_tempendo2}
\bysame, \emph{Tempered endoscopy for real groups. {II}. {S}pectral transfer
  factors}, Automorphic forms and the {L}anglands program, Adv. Lect. Math.
  (ALM), vol.~9, Int. Press, Somerville, MA, 2010, pp.~236--276.

\bibitem[Sil78]{Silberger_lang_class}
Allan~J. Silberger, \emph{The {Langlands} quotient theorem for p-adic groups},
  Math. Ann. \textbf{236} (1978), 95--104 (English).

\bibitem[Sil79]{Silberger_intro_harm}
\bysame, \emph{Introduction to harmonic analysis on reductive p-adic groups.
  {Based} on lectures by {Harish}-{Chandra} at {The} {Institute} for {Advanced}
  {Study}, 1971-73}, Math. Notes (Princeton), vol.~23, Princeton University
  Press, Princeton, NJ, 1979 (English).

\bibitem[Sil96]{Silberger_plancherel}
\bysame, \emph{Harish-{Chandra}'s {Plancherel} theorem for
  {{\(\mathfrak{p}\)}}-adic groups}, Trans. Am. Math. Soc. \textbf{348} (1996),
  no.~11, 4673--4686 (English).

\bibitem[Sol20]{Solleveld_LLCisog}
Maarten Solleveld, \emph{Langlands parameters, functoriality and {Hecke}
  algebras}, Pac. J. Math. \textbf{304} (2020), no.~1, 209--302 (English).

\bibitem[Spr98]{Springer_lag}
T.~A. Springer, \emph{Linear algebraic groups}, second ed., Progress in
  Mathematics, vol.~9, Birkhäuser Boston, Inc., Boston, MA, 1998.

\bibitem[SR72]{Saavedra}
Neantro Saavedra~Rivano, \emph{Categories tannakiennes}, Lect. Notes Math.,
  vol. 265, Springer, Cham, 1972 (French).

\bibitem[Ste65]{Steinberg_reg}
Robert Steinberg, \emph{Regular elements of semisimple algebraic groups}, Inst.
  Hautes \'Etudes Sci. Publ. Math. (1965), no.~25, 49--80.

\bibitem[Ste68]{Steinberg_endo}
R.~Steinberg, \emph{Endomorphisms of linear algebraic groups}, Mem. Am. Math.
  Soc., vol.~80, Providence, RI: American Mathematical Society (AMS), 1968
  (English).

  \bibitem[SZ18]{SilbergerZink}
Allan~J. Silberger and Ernst-Wilhelm Zink, \emph{Langlands classification for
  {L}-parameters}, J. Algebra \textbf{511} (2018), 299--357.

\bibitem[Tat66]{Tate_cohtori}
J.~Tate, \emph{The cohomology groups of tori in finite {G}alois extensions of
  number fields}, Nagoya Math. J. \textbf{27} (1966), 709--719.

\bibitem[Tat79]{Tate_ntbackground}
\bysame, \emph{Number theoretic background}, Automorphic forms, representations
  and {$L$}-functions ({P}roc. {S}ympos. {P}ure {M}ath., {O}regon {S}tate
  {U}niv., {C}orvallis, {O}re., 1977), {P}art 2, Proc. Sympos. Pure Math.,
  XXXIII, Amer. Math. Soc., Providence, R.I., 1979, pp.~3--26.

\bibitem[Th{\v{a}}11]{NguyenQuocThang_Galcoh}
Nguy{\^e}{\~n}~Qu{\^o}{\'c} Th{\v{a}}{\'n}g, \emph{On {Galois} cohomology and
  weak approximation of connected reductive groups over fields of positive
  characteristic}, Proc. Japan Acad., Ser. A \textbf{87} (2011), no.~10,
  203--208 (English).

\bibitem[vD72]{vanDijk_ind}
G.~van Dijk, \emph{Computation of certain induced characters of {$p$}-adic
  groups}, Math. Ann. \textbf{199} (1972), 229--240.

\bibitem[Wal88]{Wallach_realred1}
Nolan~R. Wallach, \emph{Real reductive groups. {I}}, Pure and Applied
  Mathematics, vol. 132, Academic Press, Inc., Boston, MA, 1988.

\bibitem[Wal03]{WaldspurgerPlancherel}
J.-L. Waldspurger, \emph{La formule de {P}lancherel pour les groupes
  {$p$}-adiques (d'apr\`es {H}arish-{C}handra)}, J. Inst. Math. Jussieu
  \textbf{2} (2003), no.~2, 235--333.

\bibitem[Xu17]{Xu_cusp_supp}
Bin Xu, \emph{On the cuspidal support of discrete series for {{\(p\)}}-adic
  quasisplit {{\(\mathrm{Sp}(N)\)}} and {{\(\mathrm{SO}(N)\)}}}, Manuscr. Math.
  \textbf{154} (2017), no.~3-4, 441--502 (English).

\bibitem[Zel80]{Zelevinsky_ind2}
A.~V. Zelevinsky, \emph{Induced representations of reductive
  {${\mathfrak{p}}$}-adic groups. {II}. {O}n irreducible representations of
  {${\rm GL}(n)$}}, Ann. Sci. \'Ecole Norm. Sup. (4) \textbf{13} (1980), no.~2,
  165--210.

\end{thebibliography}

\end{document}